\documentclass[12pt,reqno]{amsart}

\usepackage{amssymb}
\usepackage{amscd}
\usepackage{amsfonts}
\usepackage{setspace}
\usepackage{version}
\usepackage{mathrsfs}
\usepackage[colorlinks=true]{hyperref}

\usepackage{graphicx}

\newtheorem{theorem}{Theorem}[section]
\newtheorem{lemma}[theorem]{Lemma}
\newtheorem{proposition}[theorem]{Proposition}
\newtheorem{corollary}[theorem]{Corollary}

\theoremstyle{definition}

\newtheorem{definition}[theorem]{Definition}

\renewcommand{\leq}{\leqslant}
\renewcommand{\geq}{\geqslant}

\newcommand\Av{\operatorname{Av}}

\newcommand\Supp{\operatorname{Supp}}

\newcommand\vp{\Lambda_{\Z/p\Z}}
\newcommand\vxp{\chi_{\Z/p\Z}}
\newcommand\vxpbar{\overline{\chi}_{\Z/p\Z}}
\renewcommand\wp{\tau^2_{\Z/p\Z}}

\newcommand\cond{\operatorname{cond}}
\newcommand\denom{\operatorname{denom}}
\newcommand\support{\operatorname{s}}
\newcommand\sigmax{\sigma_0}

\def\R{\mathbf{R}}
\def\C{\mathbf{C}}
\def\Z{\mathbf{Z}}

\def\Q{\mathbf{Q}}
\def\N{\mathbf{N}}

\newcommand{\trunc}{\operatorname{trunc}}
\newcommand\prim{\operatorname{prim}}
\def\eps{\varepsilon}

\newcommand{\md}[1]{\ensuremath{(\operatorname{mod}\, #1)}}
\newcommand{\mdsub}[1]{\ensuremath{(\mbox{\scriptsize mod}\, #1)}}
\newcommand{\mdlem}[1]{\ensuremath{(\mbox{\textup{mod}}\, #1)}}
\newcommand{\mdsublem}[1]{\ensuremath{(\mbox{\scriptsize \textup{mod}}\, #1)}}

\numberwithin{equation}{section}

\makeatletter
\renewcommand\subsection{\@startsection{subsection}{2}%
	\z@{.5\linespacing\@plus.7\linespacing}{.5\linespacing}%
	{\normalfont\bfseries}}
\makeatother

\makeatletter
\renewcommand\part{\@startsection{part}{2}%
	\z@{0.5\linespacing\@plus2\linespacing}{\linespacing}%
	{\normalfont\large\scshape\bfseries\centering}}
\makeatother

\AtBeginDocument{
  \addtocontents{toc}{\Small}
}

\begin{document}

\title[S\'ark\"ozy for shifted primes]{On S\'ark\"ozy's theorem for shifted primes}



\author{Ben Green}
\address{Mathematical Institute\\
Radcliffe Observatory Quarter\\
Woodstock Road\\
Oxford OX2 6GG\\
England}
\email{ben.green@maths.ox.ac.uk}
\thanks{The author is supported by a Simons Investigator grant and is grateful to the Simons Foundation for their continued support.}


\begin{abstract}
Suppose that $A \subset \{1,\dots, N\}$ has no two elements differing by $p-1$, $p$ prime. Then $|A| \ll N^{1 - c}$.
\end{abstract}

\maketitle

\setcounter{tocdepth}{1}
\tableofcontents

%

\part{Introduction}

\section{Introduction}

Let $\delta(N)$ denote the relative density of the largest set $A \subset \{1,\dots,N\}$ not containing two elements $a, a'$ differing by $p-1$, $p$ a prime (that is to say, $\delta(N) = \frac{|A|}{N}$). A well-known 1978 result of S\'ark\"ozy \cite{sar78} is that $\lim_{N \rightarrow \infty} \delta(N)= 0$, that is to say the set $\{ p - 1 : p \; \mbox{prime}\}$ is a \emph{set of recurrence}. In fact, S\'ark\"ozy's work establishes a bound
\[ \delta(N) \ll  (\log \log N)^{-2 - o(1)}.\]
This was subsequently improved by Lucier \cite{luc08}, who proved that
\[ \delta(N) \ll (\log \log N)^{-\omega(N)}\] for some function $\omega(N) \rightarrow \infty$. 
In 2008 Ruzsa and Sanders \cite{rs07} made a substantial advance on the problem by establishing that
\[ \delta(N) \ll  e^{-c(\log N)^{1/4}}\] for some $c > 0$. The exponent of $\frac{1}{4}$ here was recently improved to $\frac{1}{3}$ by Wang \cite{zoe}. This is the best bound currently in the literature.

Ruzsa and Sanders remark in their paper that, even assuming the Generalised Riemann Hypothesis (GRH), the best bound that seems to be available using their methods is $\delta(N) \ll e^{-c(\log N)^{1/2}}$. They repeat a 1982 question of Ruzsa \cite{ruz82}, asking whether a bound of the shape $\delta(N) \ll N^{- c}$ can be proven under such an assumption. Our main result in this paper is a bound of this strength, without any unproven hypothesis.

\begin{theorem}\label{main-sarkozy}
Let $A \subset \{1,\dots, N\}$ be a set such that $A - A$ contains no number of the form $p - 1$, $p$ a prime. Then $|A| \ll N^{1 - c}$ for some $c > 0$.
\end{theorem}
We remark that essentially the same arguments would prove the analogous result with $p+1$ replacing $p-1$, and that trivial examples show that no such result is possible for $p + a$ with any fixed $a$ other than $\pm 1$.

We do not give a value for $c$ here, but using their explicit zero-density estimates Thorner and Zaman \cite{thorner-zaman} have shown that $c = 10^{-18}$ is admissible. 

If one assumes GRH, it is possible to establish Theorem \ref{main-sarkozy} with any $c < \frac{1}{12}$. Moreover, the argument simplifies dramatically (it is less than one third the length of the unconditional one). The interested reader may find a writeup of this argument in \cite{green-grh}. 

We will give an outline of our method for proving Theorem \ref{main-sarkozy} in the next section. The key new development is that we eschew the density increment strategy common to all of the works cited above, 
and instead study the so-called \emph{van der Corput} property, which is strictly stronger than that of being a set of recurrence, directly.

Let us recall what the van der Corput property for the shifted primes is. For each $N$, set
\[ \gamma(N) := \inf_{T \in \mathscr{T}_N} a_0, \] where $\mathscr{T}_N$ denotes the collection of all cosine polynomials
\[ T(x) = a_0 + \sum_{p \leq N} a_{p-1} \cos(2 \pi (p-1) x)\] where $a_i \in \R$ and which satisfy $T(0) = 1$, $T(x) \geq 0$ for all $x$. The van der Corput property is that $\gamma(N) \rightarrow 0$ as $N \rightarrow \infty$. We note that the van der Corput property may be formulated in different ways, and the above formulation is due to Kamae and Mend\`es France \cite{kamae-mendes-france} and Ruzsa \cite{ruz-connections}. Montgomery's \emph{Ten Lectures} \cite[Chapter 2]{montgomery} provides a good discussion.

We will show the following result.

\begin{theorem}\label{vdc}
We have $\gamma(N) \ll N^{-c}$ for some absolute constant $c > 0$. That is, there is a cosine polynomial
\[ T(x) = a_0 + \sum_{p \leq N} a_{p-1} \cos(2 \pi (p-1) x)\] where $a_i \in \R$, $T(0) = 1$, $T(x) \geq 0$ for all $x$ and $a_0 \ll N^{-c}$.
\end{theorem}

One may of course define the quantities $\delta(N)$ and $\gamma(N)$ with respect to an arbitrary set, rather than just for the shifted primes. In this general setting, roughly speaking we have $\delta \lessapprox \gamma$. More precisely, it follows from \cite[Chapter 2, Lemma 1]{montgomery} that one has $\delta(N) \leq (1 + \frac{H}{N}) \gamma(H)$ for any $H$, $1 \leq H \leq N$, so in particular $\delta(N) \leq 2 \gamma(N)$. We will give a self-contained proof of this latter fact (for shifted primes, but the argument is general) in Section \ref{sec2}. Bourgain \cite{bourgain-recurrence} gave an example showing that inequalities in the opposite direction do not hold in general, that is to say being a van der Corput set is strictly stronger than being a recurrent set. See also \cite{ruzsa-matolcsi-1} for related material.

We now revert to the specific setting of the shifted primes. Regarding previous results in the direction of Theorem \ref{vdc}, Kamae and Mend\`es France showed that the shifted primes have the van der Corput property, but gave no explicit bounds. Currently, the best-known upper bound for $\gamma(N)$ seems to be that of Slijep\v{c}evi\'c \cite{slijepcevic-primes}, who showed that $\gamma(N) \leq (\log N)^{-1 + o(1)}$.

Finally, we note that the best-known lower bound for $\delta(N)$ is rather small, though not as small as one might na\"{\i}vely guess. This is a result of Ruzsa \cite{ruz84}, who showed that 
\begin{equation}\label{lower-f} \delta(N) \geq \frac{1}{N}e^{(\log 2 + o(1)) \frac{\log N}{\log \log N}}.\end{equation} Note that this is asymptotically smaller than $N^{-1 + \eps}$ for any $\eps > 0$, and one would expect this to be the true behaviour of $\delta(N)$. In fact, \eqref{lower-f} may be close to the true bound. I believe (and see \cite{slijepcevic-primes}) that no better lower bound is known for $\gamma(N)$.

Proving upper bounds of this strength is surely hopeless. Indeed, if one could show that $\delta(N) < N^{-1/2 - c}$ for some $c > 0$ then, taking $A$ to be the multiples of $q$ up to $q^{2 - c'}$, one sees that there is a prime $p \leq q^{2 - c'}$ congruent to $1 \md{q}$. Such a bound is beyond what is currently known, even on GRH.

\emph{Acknowledgements.} The initial impetus for thinking about the van der Corput property came after the author became aware that the Lov\'asz $\theta$-function can be used to show that any set $A \subset \Z/q\Z$ of size $> \sqrt{q}$ contains two elements differing by a square $\md{q}$, where $q$ is any product of distinct primes $1 \md{4}$. This connection was also observed by Naslund \cite{naslund}, who says it was known to Ruzsa (unpublished). I thank Zachary Chase for interesting conversations on this topic and for presenting Lov\'asz's paper to me in an easily digestible form which made the connection to the van der Corput property clear.  I would also like to thank Ofir Gorodetsky, Lasse Grimmelt, Emmanuel Kowalski, Valeriya Kovaleva, James Maynard, Eric Naslund, Joni Ter\"av\"ainen and Jesse Thorner for various helpful comments and correspondence. Finally, I am very grateful to the two anonymous referees for a careful reading of the paper.

\emph{Notation.} In this section we introduce notation which is \emph{global} to the paper, most of which is standard. There will be many further pieces of notation of a more local nature which we will introduce later on.

Our use of asymptotic notation such as $O()$, $o()$, $\ll$ and $\gg$ is standard in analytic number theory. Sometimes we will add subscripts to this notation; thus, for example, $O_B(1)$ means a constant depending only on some parameter $B$. We will use the rougher notation $X \lessapprox Y$ to mean $X \ll N^{o(1)}Y$, where here $N$ is a global parameter throughout the paper (we are interested in subsets of $\{1,\dots,N\}$). We write $X = \tilde{O}(Y)$ to mean the same thing. Often, but not always, the $N^{o(1)}$ term will be just a fixed power of $\log N$. We will remind the reader of this notation when we first use it.

We will use the following standard arithmetic functions: the M\"obius function $\mu$, Euler's totient function $\phi$, the von Mangoldt function $\Lambda$, and its restriction to primes $\Lambda'$ (that is, the proper prime powers $p^i$, $i \geq 2$, are removed from the support). Given positive integers $a,b$ we write $(a,b)$ for their gcd and $[a,b]$ for their lcm. 

A slightly unusual notation we will use is to write $\denom(x)$ for the denominator of the rational number $x$ when written in lowest terms, or in other words the least positive integer $q$ such that $qx \in \Z$. The quantity $\denom(x)$ is well-defined for $x \in \Q/\Z$, which is the context in which we will use it.

Given $\theta \in \R$, we write $e(\theta) = e^{2\pi i \theta}$ and $\Vert \theta \Vert$ for the distance from $\theta$ to the nearest integer. There is no other use for $\Vert \cdot \Vert$ in the paper, so we do not need a more pedantic notation such as $\Vert \cdot \Vert_{\R/\Z}$. 

If $q$ is a positive integer, we write $c_q(n) := \sum_{a \in (\Z/q\Z)^*} e(\frac{an}{q})$, the Ramanujan sum.

If $p$ is a prime and $n$ is a non-zero integer, we write $v_p(n)$ to be the power of $p$ dividing $n$, and extend this to a function $v_p : \Q \setminus \{0\} \rightarrow \Z$ in the usual way, defining $v_p(\frac{a}{b}) := v_p(a) - v_p(b)$. We will, by a slight abuse of notation, define $v_p$ on $(\Z/p^n \Z) \setminus \{0\}$ by taking $v_p(x) = v_p(\overline{x})$ for any $\overline{x} \in \Z$ projecting to $x$ under the natural map.

In several later sections of the paper we make considerable use of the notation $F^+(n) := F(n+1)$, $F^-(n) := F(n - 1)$. Whilst I acknowledge that this is a somewhat unattractive and nonstandard piece of notation, it is essential in order to keep expressions manageable later on.

Finally, in a number of sections we will see the constant $M$, whose definition may be found at \eqref{m-def} and which takes the value $M = 7 \cdot 3^6 \cdot 2^{22993}$. The precise value of this constant is not especially important provided that it is large enough, but we choose to fix a value for definiteness.

\section{Non-negative trigonometric polynomials} \label{sec2}

Rather than work with the van der Corput property as defined in the introduction directly, we instead deduce bounds on $\delta(N)$ and $\gamma(N)$ from the following proposition, which encodes a property easily seen to be equivalent to the van der Corput property.

\begin{proposition}\label{main-construct}
Suppose that there is a function $\Psi : \Z \rightarrow \R$ and $\delta_1, \delta_2 > 0$ with the following properties:
\begin{enumerate}
\item $\Psi(n)$ is supported on the set $\{ p-1 : \mbox{$p$ prime}\}$;
\item $\sum_{n = 1}^N \Psi(n) \cos(2\pi n\theta) \geq -\delta_1 N$, for all $\theta \in \R/\Z$;
\item $\sum_{n = 1}^N \Psi(n) \geq \delta_2 N$.
\end{enumerate}
Then $\delta(N) \leq \frac{2\delta_1}{\delta_1 + \delta_2}$ and $\gamma(N) \leq \frac{\delta_1}{\delta_1 + \delta_2}$. Conversely, there are $\delta_1, \delta_2 > 0$ and a function $\Psi$ satisfying \textup{(1)}, \textup{(2)} and \textup{(3)} above and such that $\gamma(N) = \frac{\delta_1}{\delta_1 + \delta_2}$.
\end{proposition}
\emph{Remarks.}  An immediate corollary of this is that $\delta(N) \leq 2\gamma(N)$. 

We remark that no specific properties of the set $\{ p-1 : \mbox{$p$ prime}\}$ are used in this section; the arguments and conclusions apply to recurrence and van der Corput properties of general sets.

\begin{proof}[Proof of Proposition \ref{main-construct}] We begin with the first two statements, the bounds on $\delta(N)$ and $\gamma(N)$. We will use Fourier analysis $\md{2N}$. Suppose that $A \subset \{1,\dots, N\}$ does not contain any elements differing by $p-1$, where $p \leq N$ is a prime. Then it does not contain any elements differing modulo $2N$ by $p-1$, $p \leq N$ a prime.

Now we have
\begin{align} \nonumber (\delta_1 & + \delta_2) N |A|^2 \\ & \leq \sum_{r \in \Z/2N\Z} \bigg|\sum_{n=1}^N 1_A(n) e\big(\frac{rn}{2N}\big)\bigg|^2 \bigg( \delta_1 N + \Re \sum_{x=1}^N \Psi(x) e\big(\frac{rx}{2N}\big)\bigg) ,\label{first-pos}\end{align} since the term with $r = 0$ is $\geq  (\delta_1 + \delta_2)|A|^2 N$ by (3) and all other terms are non-negative by (2). On the other hand, we have
\begin{equation}\label{parseval} \sum_{r \in \Z/2N\Z} \bigg|\sum_{n=1}^N 1_A(n) e\big(\frac{rn}{2N}\big)\bigg|^2  = 2N \!\!\!\!\!\!\sum_{m \equiv n \mdsub{2N}}\!\!\!\!\! 1_A(m) 1_A(n) = 2N |A|,\end{equation}whilst
\begin{align*} \sum_{r \in \Z/2N\Z} & \bigg|\sum_{n=1}^N 1_A(n) e\big(\frac{rn}{2N}\big)\bigg|^2  \sum_{x = 1}^N \Psi(x) e\big(\frac{rx}{2N}\big)  \\ & = 2N \sum_{m=1}^N\sum_{n = 1}^N \sum_{x = 1}^N 1_{m - n + x \equiv 0 \mdsub{2N}} 1_A(m) 1_A(n) \Psi(x) = 0,\end{align*} because $\Psi$ is supported on $p-1$, $p$ prime, and $A$ does not contain two elements differing by such a number $\md{2N}$.

Taking real parts and combining this with \eqref{first-pos}, \eqref{parseval} gives $(\delta_1 + \delta_2)|A|^2 N \leq 2\delta_1 N^2 |A|$, from which we obtain $|A| \leq \frac{2\delta_1}{\delta_1 + \delta_2} N$, as required.

For the van der Corput property, consider
\[ T(x) := \big(\delta_1 N + \sum_{n = 1}^N \Psi(n)\big)^{-1} \big(\delta_1 N + \sum_{n = 1}^N \Psi(n) \cos (2\pi nx)\big).\]
It is clear that $T(0) = 1$, $T(x) \geq 0$ for all $x$ and finally
\[ a_0= \frac{\delta_1 N}{\delta_1 N + \sum_{n = 1}^N \Psi(n)} \leq \frac{\delta_1}{\delta_1 + \delta_2}, \]
which shows that indeed $\gamma(N) \leq \frac{\delta_1}{\delta_1 + \delta_2}$.

Finally we turn to the converse statement. Suppose we have a cosine polynomial $T(x) = a_0 + \sum_{p \leq N} a_{p-1} \cos(2 \pi (p-1) x)$ where $a_i \in \R$, $T(0) = 1$, $T(x) \geq 0$ for all $x$ and $a_0 = \gamma(N)$. Define $\Psi(n)$ to be $a_{p-1}$ if $n = p-1$, $p \leq N$ a prime, and $0$ otherwise. Then property (1) is satisfied. Since $T(x) = a_0 + \sum_{n \leq N} \Psi(n) \cos (2 \pi n x)$, the assumption that $T(x) \geq 0$ implies (2), with $\delta_1 := \frac{a_0}{N}$. The assumption that $T(0) = 1$ implies (3), with $\delta_2  := \frac{1- a_0}{N}$. The claim follows.
\end{proof}

From this, we see that Theorem \ref{main-sarkozy} and Theorem \ref{vdc} are both consequences of the following result (which is in fact equivalent to Theorem \ref{vdc}).

\begin{theorem}\label{mainthm-3}
There are absolute constants $0 < \kappa_2 < \kappa_1$ and a function $\Psi : \Z \rightarrow \R$ with the following properties:
\begin{enumerate}
\item $\Psi(n)$ is supported on the set $\{ p-1 : \mbox{$p$ prime, $p \leq N$} \}$;
\item $\sum_{n = 1}^N \Psi(n) \cos(2\pi \theta n) \geq - N^{1 - \kappa_1}$, for all $\theta \in \R/\Z$;
\item $\sum_{n = 1}^N \Psi(n) \geq N^{1 - \kappa_2}$.
\end{enumerate}
\end{theorem}

The proof of this result occupies the rest of the paper.

\section{Sketch of the construction and plan of the paper}\label{section3}

\subsection{Sketch of the construction}\label{subsec3.1}
The main business of the paper is the construction of the function $\Psi$ in Proposition \ref{main-construct}.

The construction is complicated, so we begin by explaining roughly where it comes from. The starting point is to observe that 
\begin{equation}\label{psi-0} \Psi_0(n) := \Lambda'(n+1)\Lambda'(n - 1) \tau(n)^2  \big(1 - \frac{|n|}{N}\big)_+ \end{equation} ought to have something like the desired properties, where $\Lambda'$ is the von Mangoldt function restricted to primes, $\tau$ is the divisor function, and $x_+$ denotes $\max(x, 0)$. To see why, consider how one would expect this function to behave on progressions $n \equiv a \md p$, where $p$ is a prime of moderate size. The $\Lambda'(n+1)$ term will be uniformly distributed across $a \not\equiv -1 \md{p}$, whilst the $\Lambda'(n - 1)$ term will be uniformly distributed where $a \not\equiv 1 \md{p}$. Thus one expects, assuming that $\Lambda'(n+1)$ and $\Lambda'(n-1)$ behave suitably independently, that $\Lambda'(n-1)\Lambda'(n+1)$ is uniformly distributed where $a \not\equiv \pm 1 \md{p}$. The symmetry about 0 here guarantees that the exponential sum of $\Lambda'(n+1)\Lambda'(n-1)$ at points $\frac{r}{p}$ ought to be essentially real-valued. The divisor function $\tau(n)$ has roughly twice the weight on $a \equiv 0 \md{p}$ as it does on each of the other values of $a$ (for each divisor $d$ of $n/p$, we get both $d$ and $pd$ as divisors of $n$). Very roughly, this leads one to expect that $\Psi_0(n)$ has total weight proportional to $\nu(a)$ on $n \equiv a \md{p}$, where $\nu(0) = 4$, $\nu(\pm 1) = 0$ and $\nu(a) = 1$ for all other $a$. This weight has exponential sum at $\frac{r}{p}$ which is both real \emph{and} positive, which is a very helpful step towards proving (2). Moreover, by considering the multiplicative behaviour of the functions concerned, one can hope for a similar statement at points $\frac{b}{q}$, $q$ squarefree. Finally, the tent function $(1 - \frac{|n|}{N})_+$ has non-negative Fourier transform  (the Fej\'er kernel) and so one might expect $\Psi_0$ to satisfy something like property (2).

Unfortunately, we cannot make any aspect of the preceding discussion rigorous. Indeed, the function $\Lambda'(n+1)\Lambda'(n-1)$ may, for all we know, be essentially zero (the negation of this is the twin prime conjecture) and so certainly understanding the asymptotics of its exponential sum is hopeless.

Our actual construction of $\Psi$ is inspired by \eqref{psi-0} but includes various truncated and modified versions of the arithmetic functions involved as well as some extra terms. Roughly, we will construct it as a product
\begin{equation}\label{psi-prod} \Psi(n) := \Lambda'(n+1) \Lambda_Q(n - 1) H_Q(n) D_{Q'}(n)  E(n) w(n).   \end{equation}
Here, $Q , Q'$ are certain parameters to be specified later, with $Q' < Q$. Let us give a very brief overview of what each of these terms are and their purpose.

The $\Lambda_Q$ term is a kind of Fourier-truncation of the von Mangoldt function given by
\begin{equation}\label{lambda-q-def}  \Lambda_Q(n) := \sum_{q \leq Q} \frac{\mu(q)}{\phi(q)} c_q(n),\end{equation} where
$c_q(n) := \sum_{a \in (\Z/q\Z)^*} e(\frac{an}{q})$ is a Ramanujan sum. This function models the expected distribution of $\Lambda$ on progressions but is much easier to analyse than $\Lambda$ itself. It was used by Heath-Brown \cite{heath-brown} to give an alternative proof of Vinogradov's three primes theorem.

The factor $H_Q$ behaves (somewhat) like a Fourier-truncated truncated version of $\tau^2$. It is defined as
\begin{equation}\label{hq-def} H_Q(n) := \sum_{q \leq Q} \eta(q) c_q(n)\quad \mbox{where} \quad \eta(q) := \mu^2(q)\prod_{p | q} \frac{3}{p+3}.\end{equation}  

Thus one should think of $\Lambda'(n+1) \Lambda_Q(n-1) H_Q(n)$ as a kind of Fourier-truncated variant of the function $\Lambda(n+1)\Lambda(n-1) \tau(n)^2$ discussed at the beginning of the section. We discussed, on a heuristic level, the distribution of this latter function in progressions of small modulus (prime, in the discussion, but we need the general case), and a similar heuristic discussion applies to $\Lambda'(n+1)\Lambda_Q(n-1) H_Q(n)$. This still assumes that $\Lambda$ is well-distributed in progression of small modulus, where ``small'' means up to a power of $N$. Unfortunately, nothing like this is known to be true: possible zeros near to $\Re s = 1$ of Dirichlet $L$-functions $L(s,\chi)$ can significantly skew the distribution.

To account for these irregularities we introduce the term $D_{Q'}$, which we call a \emph{damping term}. This term has the form
\[ D_{Q'}(n) = \sum_{r \leq Q'} \sum_{b \in \Z/r^3 \Z} \alpha_{b,r} r^{\frac{5}{6}} 1_{r | n} e\big(\frac{bn}{r^3}\big),\] where the $\alpha_{b,r}$ are non-negative real numbers summing to $1$. (There is some scope for varying the choice of the numbers $\frac{5}{6}$ and $3$, but it is crucial that $\frac{5}{6} < 1$.)  Its role is essentially to ensure that the irregularities in distribution just discussed still lead to non-negative Fourier coefficients. 

The construction of $D_{Q'}(n)$ is rather complicated. It is defined in Section \ref{sec13} as an infinite series \eqref{d-series}, but this definition depends on material developed in several earlier sections. Roughly, the main contributions to the $\alpha_{b,r}$ come from characters $\chi$ which have zeros particularly close to $\Re s = 1$. The need for the twists by $e(\frac{bn}{r^3})$ arises because of the properties of Dirichlet characters to prime power moduli with large exponent, in particular the fact that these can behave like additive characters on large subgroups. 

The function $E(n)$ is either the constant function $1$ or an indicator $1_{q_1 | n}$. The latter case arises, essentially, when there is an exceptional character $\chi_1$ of conductor $q_1$, that is to say a character such that $L(s,\chi)$ has a zero very close to $s = 1$.

Finally, $w(n)$ is an ``archimedean'' factor. The tent function $w(n) = (1 - \frac{|n|}{N})_+$ turns out not to work, essentially because one needs the exponential sum of $n^{\rho - 1} w(n)$ to be pointwise dominated by that of $w(n)$, for $\rho$ a possible zero of some $L(s,\chi)$. This is not the case for the tent function, since the exponential sum of $w(n)$ in that case (the Fej\'er kernel) has zeros. We are instead forced to make a much more elaborate construction using properties of the function $W(x) = x^{-\frac{1}{2}} e^{-x}$, whose Fourier transform has a relatively slow and controlled decay.

\subsection{Plan of the paper}

The main business of the paper is to properly define $\Psi$ as given by \eqref{psi-prod} and to show that it satisfies the conclusions of Theorem \ref{mainthm-3}. This is a lengthy task.

Part \ref{partii} of the paper may be read independently of the rest of the paper and may plausibly be useful elsewhere. In this section we construct another approximant $\Lambda_{\sharp, Q,\sigmax}$ to $\Lambda$. Unlike $\Lambda_Q$, this approximant sees possible nontrivial zeros with $\Re \rho \geq 1 - \sigmax$ up to height $Q$ of all Dirichlet $L$-functions $L(s,\chi)$ with $\chi$ of conductor $\leq Q$. Including this information about zeros makes $\Lambda_{\sharp, Q,\sigmax}$ a more complicated object than $\Lambda_Q$, but it allows us to rigorously show that exponential sums $\sum_{n \leq N} (\Lambda(n) - \Lambda_{\sharp, Q,\sigmax}(n)) e(n \theta)$ enjoy a \emph{power saving} $\lessapprox NQ^{-\frac{1}{6}}$ over the trivial bound of $N$, uniformly in $\theta$. This allows us to reduce matters to the analysis of, instead of $\Psi$, the modification $\Psi'$ in which $\Lambda(n+1)$ is replaced by $\Lambda_{\sharp, Q,\sigmax}(n+1)$, for suitable $Q,\sigmax$.  

Part \ref{partiii} of the paper assembles various preliminaries for the main argument. In Section \ref{sec7}, we describe some zero-density estimates and their consequences. The material here is similar to what one would find in a proof of Linnik's theorem on the least prime in a progression, which is not surprising since Theorem \ref{main-sarkozy} is easily seen to imply that the progression $1 \md{q}$ contains a prime of size $q^{O(1)}$. What we need here is a little stronger than what is required for Linnik's theorem, where the modulus is fixed. Our main reference is Bombieri's book \cite{bombieri}. 

Section \ref{rat-fourier-sec} discusses periodic functions on $\Z$, their rational Fourier expansions, and the notion of positivity. It also introduces some of the key examples and functions in the paper such as $H_Q$ as described above. 

Section \ref{comparison-sec} and the related Section \ref{sec10} are rather long and technical, though crucial to the paper. They are concerned with passing between truncated series $\sum_{q \leq Q} \alpha(q) c_q(n)$, which are good for Fourier analysis, and untruncated series such as $\sum_{q | Q!} \alpha(q) c_q(n)$, which are much easier to understand combinatorially. Section 9 concerns triple products of such series and their shifts, and in Section 10 characters are introduced as well. Despite the technical nature of these sections, the statement of the main result of Section 9, Proposition \ref{ram-truncate}, is self-contained and potentially applicable in other contexts.

Section \ref{sec11} develops material necessary for the construction of the damping term $D_{Q'}$. A crucial input here is the Postnikov character formula which, due to the lack of a reference suitable for our purposes, is developed from first principles in Appendix \ref{appA}.

Section \ref{sec12} constructs the archimedean weight $w(n)$, modelled on $x^{-\frac{1}{2}} e^{-x}$. The main result here, Proposition \ref{prop15.2}, is self-contained and is also potentially of some use elsewhere. 

Finally, we come to Part \ref{p-adic-part}, where the function $\Psi$ is rigorously constructed and shown to have the desired properties. The construction of the damping function $D_{Q'}$ is probably the most conceptually interesting part of this analysis. The fully detailed definition of $\Psi$ is given in Definition \ref{psi-definition}, and finally in the rest of that section it is shown that $\Psi$ satisfies Theorem \ref{mainthm-3}.\vspace*{11pt}

\emph{Some comments for the reader.} I have endeavoured to make key intermediate results self-contained where possible, so that it is possible to understand the basic structure of the argument without reading all the proofs. There are a very large number of hyperlinks in the paper to facilitate recall of key definitions and intermediate results. Consequently, I recommend that anyone interested in reading the paper use a suitable PDF viewer rather than the printed version.

\part{An approximant for the primes}\label{partii}

The objective of this part of the paper is define certain approximants $\Lambda_{\sharp}$ to the von Mangoldt function $\Lambda$ which are very close to it in Fourier space. The definition (which requires a little setting up) may be found in Definition \ref{sharp-approx-def}, and the key result is Proposition \ref{approx-fourier}. Recall that $N$ is a large integer, a global parameter in the paper, and that $X \lessapprox Y$ means that $X \ll N^{o(1)} Y$.

\section{Definition of the approximant}

We adopt the following fairly standard convention in analytic number theory: $\rho$ denotes a nontrivial zero of some $L$-function, and we write $\rho = \beta + i \gamma$; somewhat less standardly, we always write $\sigma := 1 - \beta$, thus $\Re \rho = 1 - \sigma$.

\begin{definition}\label{zero-ht-def}
Let $Q$ be a positive real parameter, and let $\sigmax > 0$. Write $\Xi_Q(\sigmax) \subset \C$ for the multiset consisting of all $\rho = \beta + i \gamma = 1 - \sigma + i \gamma$ with $\sigma \leq \sigmax$ and $|\gamma| \leq Q$ which are zeros of some Dirichlet $L$-function $L(s,\chi)$ with conductor at most $Q$. 
\end{definition}
\emph{Remarks.} Later on we will fix $\sigmax = \frac{1}{48}$, a convenient but not optimal choice which works for all our arguments. 

When we say that $\Xi_Q(\sigmax)$ is a multiset, we mean the following. If $\rho$ occurs as a zero of multiplicity $r_i$ for primitive Dirichlet characters $\chi_i$, $i = 1,\dots, m$, then $\rho$ should appear in $\Xi_Q(\sigmax)$ with multiplicity $r_1 + \dots + r_m$. Conjecturally, every $\rho$ is a simple zero of a unique $L(s,\chi)$ (and so $\Xi_Q(\sigmax)$ is simply a set) but this is not known. In fact $\Xi_Q(\sigmax)$ is conjecturally \emph{empty} if $\sigmax < \frac{1}{2}$. If $\rho \in \Xi_Q(\sigmax)$, we write $\chi_{\rho}$ for the primitive Dirichlet character associated to it. 

Suppose $\psi$ is a Dirichlet character (not necessarily primitive) to modulus $r$. Recall that the \emph{Gauss sum} $\tau(\psi)$ is defined by
\[ \tau(\psi) := \sum_{b \in (\Z/r\Z)^*} \psi(b) e\big(\frac{b}{r}\big).\]

The following three definitions are important ones in the paper.

\begin{definition}\label{c-chi-def}
Let $\chi$ be a primitive Dirichlet character to modulus $q$. Let $r$ be a positive integer and let $b \in (\Z/r\Z)^*$. Then if $q \mid r$ we define
\begin{equation}\label{c-def} c_{\chi}(b,r) := \frac{1}{\phi(r)} \chi (b) \mu\big(\frac{r}{q}\big) \overline{\chi\big(\frac{r}{q}\big)} \overline{\tau(\chi)},\end{equation} and set $c_{\chi}(b,r) = 0$ if $q \nmid r$.
\end{definition}

\begin{definition}\label{def43}
Let $\chi$ be a primitive Dirichlet character to modulus $q$. Then we define
\begin{equation}\label{fourier-trunc} F_{\chi, Q}(n) :=  \sum_{\substack{ q | r \\ \frac{r}{q} \leq Q}} \sum_{b \in (\Z/r\Z)^*} c_{\chi}(b,r) e\big(\frac{bn}{r}\big),\end{equation} where $c_\chi$ is as given in \eqref{c-def}.\end{definition}

Finally, we can define the approximants $\Lambda_{\sharp}$ themselves.
\begin{definition}\label{sharp-approx-def}
If $n$ is a positive integer, define
\begin{equation}\label{lam-sharp-def} \Lambda_{\sharp,Q,\sigmax}(n) :=  \Lambda_Q(n) - \sum_{\rho \in \Xi_Q(\sigmax)} n^{\rho - 1} F_{\chi_{\rho}, Q}(n).\end{equation}
\end{definition}

\emph{Remarks.} As we progress with our analysis we will understand more about why these are relevant definitions. For instance, we will see in Lemma \ref{completed-forms} that $F_{\chi, Q}$ can be thought of as a kind of ``Fourier-truncated'' version of $\overline{\chi}$ restricted to numbers with no prime factors $\leq Q$.

When $\chi = \chi_0$, the (primitive) principal character, we have $c_{\chi}(b,r) = \frac{\mu(r)}{\phi(r)}$ and so 
\begin{equation}\label{hb-link} F_{\chi_0, Q}(n) = \Lambda_Q(n),\end{equation} where $\Lambda_Q$ is the approximant \eqref{lambda-q-def}.
Given this, it is possible and often convenient to write 
\begin{equation}\label{lam-sharp-conv} \Lambda_{\sharp, Q,\sigmax}(n) = \sum_{\rho \in \Xi_Q^*(\sigmax)} \eps_{\rho} n^{\rho - 1} F_{\chi_{\rho}, Q}(n).\end{equation}
Here, we set $\Xi_Q^*(\sigmax) = \Xi_Q(\sigmax) \cup \{1\}$ and we adopt the convention that $\chi_1$ is the principal character (of which $1$ is of course a pole, not a zero), and $\eps_{\rho} = -1$ or $1$ according as $\rho$ is a zero or a pole (i.e. $\rho = 1$).

I am not aware of $\Lambda_{\sharp, Q,\sigmax}$ having appeared previously in the literature. Ter\"av\"ainen \cite[Definition 5.2]{joni} makes a vaguely related definition (but only using the zeros of $\zeta$), and there is some similarity with the work of Pintz \cite{pintz1,pintz2} on the exceptional set for the Goldbach problem.

As a final remark, we note that on the assumption of GRH we have $\Lambda_{\sharp, Q,\sigmax} = \Lambda_Q$ for all $\sigmax < \frac{1}{2}$, since $\Xi_Q(\sigmax)$ is empty.

\begin{lemma}\label{lem45} $\Lambda_{\sharp, Q,\sigmax}$ is real-valued.
\end{lemma}
\begin{proof}
One may check using $\tau(\overline{\chi}) = \chi(-1) \overline{\tau(\chi)}$ that $c_{\overline{\chi}}(b,r) = \overline{c_{\chi}(-b, r)}$ and thus that 
\begin{equation}\label{conjs} F_{\overline{\chi}, Q} = \overline{F_{\chi, Q}}.\end{equation}
Suppose that $\rho \in \Xi_Q(\sigmax)$, thus $L(\rho, \chi_{\rho}) = 0$. Then $L(\overline{\rho}, \overline{\chi_{\rho}}) = 0$ (note that $L(s,\chi) = \overline{L(\overline{s}, \overline{\chi})}$, since both sides are meromorphic and they agree for $\Re (s) > 1$). That is, $\overline{\rho} \in \Xi_Q(\sigmax)$ and $\chi_{\overline{\rho}} = \overline{\chi_{\rho}}$ .
It follows from this and \eqref{conjs} that $n^{\rho - 1} F_{\chi_{\rho}, Q}(n)  = \overline{n^{\overline{\rho} - 1} F_{\chi_{\overline{\rho}}, Q}(n)}$, thus the summands in \eqref{lam-sharp-def} come in conjugate pairs.
\end{proof}

Finally we state the key Fourier approximation property.

\begin{proposition}\label{approx-fourier}
Suppose that $\sigmax \leq \frac{1}{9}$ and that $Q \leq N^{\sigmax}$. Then we have
\[ \big| \sum_{n \leq N} (\Lambda(n) - \Lambda_{\sharp, Q,\sigmax}(n)) e(n\theta)\big| \lessapprox NQ^{-\frac{1}{6}}\] and
\[ \big| \sum_{n \leq N} (\Lambda'(n) - \Lambda_{\sharp, Q,\sigmax}(n)) e(n\theta)\big| \lessapprox NQ^{-\frac{1}{6}}\]
uniformly for $\theta \in \R/\Z$.
\end{proposition}
Here, note that the second statement (which is the one we will need later on) follows easily from the first, since $\sum_{n \leq N} |\Lambda(n) - \Lambda'(n)| \lessapprox N^{\frac{1}{2}}$.

\section{Proof of the Fourier approximation property}\label{sec5}

The business of this section is to prove Proposition \ref{approx-fourier}.  The basic idea will be to consider the cases of $\theta$ close to a rational with small denominator (the ``major arcs'') and $\theta$ not close to a rational with small denominator (the ``minor arcs'') separately. In other words, we use the circle method. The heart of the argument begins in Subsection \ref{minor-lam-q} below.

We will need various ingredients in the proof, which we assemble now before moving on to the main argument.

\subsection{Zero-density estimates}
We have the following result, sometimes known as a grand log-free zero-density estimate.

\begin{proposition}\label{jut-est}
We have, uniformly for $\alpha \geq \frac{4}{5}$,
\begin{equation}\label{unexceptional-jutila} \# \{ \rho \in \Xi_Q(\textstyle\frac{1}{5}\displaystyle) : \Re \rho \geq \alpha\} \ll Q^{8(1 - \alpha)}.\end{equation} Here, zeros are counted with multiplicity.\end{proposition}
This follows from \cite[Theorem 1]{jutila} (with any constant larger than $6$ in place of $8$). That zeros are counted with multiplicity is not explicitly stated there (or indeed in many results of this type) but the proofs do give this. Any result of this general type (that is, a ``log-free'' zero-density estimate over all moduli) would be suitable for our purposes. Jutila (\emph{ibid.}) describes some previous results of this type and remarks that the first such estimates were due to work of Fogels \cite{fogels} and Gallagher \cite{gallagher}, and can be found explicitly in the latter paper.

\emph{Remark.} In Section \ref{sec7}, we will need more refined estimates coming from the phenomenon of ``exceptional zero repulsion''; see Proposition \ref{prop75}.

A consequence of Proposition \ref{jut-est} is the following convergence result. 

\begin{lemma}\label{crude-zero-density}
Suppose that $Q \leq N^{\frac{1}{9}}$. Then $\sum_{\rho \in \Xi_Q(\frac{1}{5})} N^{\Re \rho - 1} \ll 1$.
\end{lemma}
\begin{proof}
For $k \geq 1$, write $\Xi^{(k)}_Q$ for the multiset of all $\rho \in \Xi_Q(\frac{1}{5})$ for which $k-1 < (1 - \Re \rho) \log Q \leq k$. By Proposition \ref{jut-est}, $|\Xi_Q^{(k)}| \ll e^{8k}$ for all $k$. Summing over $k$, we obtain
\[ \sum_{\rho \in \Xi_Q} N^{\Re \rho - 1} \ll \sum_{k \geq 1} N^{-\frac{k-1}{\log Q}} e^{8k} \leq \sum_{k \geq 1} e^{-9k} e^{8k} \ll 1.\]
\end{proof}

\subsection{Exponential sum estimates}

We will need the following estimate for exponential sums twisted by $n^{\rho - 1} = n^{\beta} e^{i \gamma \log n}$. This is material of a standard type, but it seems hard to locate a usable reference in the literature.

\begin{proposition}\label{prop43}
Suppose that $\rho = \beta + i \gamma$, where $\frac{1}{2} \leq \beta \leq 1$ and $|\gamma| \leq N^{1/2}$. Then we have 
\begin{equation}\label{three-ests} \sum_{n \leq N} n^{\rho - 1} e(\theta n) \ll  N^{\beta} \min\big( 1, (1 + |\gamma|)N^{-1} \Vert \theta \Vert^{-1},  |\gamma|^{-1/2} \big). \end{equation}
\end{proposition}

The proof will require the following two well-known estimates for exponential sums.

\begin{lemma}[Kuzmin--Landau] \label{kl-bound} Let $I \subset \R$ be a bounded interval. Let $f : I \rightarrow \R$ be continuously differentiable such that $f'$ is monotonic and $\Vert f'(t) \Vert \geq \lambda$ on $I$. Then 
\[ \sum_{n \in I \cap \Z} e(f(n)) \ll \lambda^{-1}.\]
\end{lemma}\begin{proof} See \cite[Part I, Theorem 6.7]{tenenbaum} or \cite[Theorem 2.1]{gk}. Recall that here $\Vert \theta \Vert$ is the distance to the nearest integer.\end{proof}

\begin{lemma}[van der Corput] \label{vdc-bd} Let $I \subset \R$ be a bounded interval. Let $f : I \rightarrow \R$ be twice continuously differentiable with $\lambda \leq |f''(x)| \leq \alpha \lambda$. Then 
\[ \sum_{n \in I \cap \Z} e(f(n)) \ll \alpha |I| \lambda^{1/2} + \lambda^{-1/2}.\]
\end{lemma}\begin{proof} See the comments after the proof of Theorem 6.7 in \cite[Part I]{tenenbaum} or \cite[Theorem 2.2]{gk} for the exact version we state here, or \cite[Corollary 8.13]{ik}. \end{proof}

\begin{proof}[Proof of Proposition \ref{prop43}]
Since $e(\theta n)$ is periodic, we may assume $0 \leq |\theta| \leq \frac{1}{2}$ throughout, and hence that $\Vert \theta \Vert = |\theta|$.
Set \[ S(\gamma, \theta, N) := \sum_{n \leq N} n^{i \gamma} e( \theta n) = \sum_{n \leq N} e\big(\frac{\gamma}{2\pi} \log n + \theta n\big).\]
From the identity 
\[ n^{\beta - 1} = N^{\beta - 1} - (\beta - 1) \int^N_n x^{\beta - 2} dx\] we see that 
\begin{equation}\label{partial-sum}  \sum_{n \leq N} n^{\rho - 1} e(\theta n) = N^{\beta - 1} S(\gamma, \theta , N) - (\beta - 1) \int^N_1 x^{\beta - 2} S(\gamma, \theta, x) dx.\end{equation}

Now \eqref{three-ests} consists of three bounds. The first one, that the sum is $\ll N^{\beta}$, is essentially trivial:
\[ \big|\sum_{n \leq N} n^{\rho - 1}\big| \leq \sum_{n \leq N} |n^{\rho - 1}| = \sum_{n \leq N} n^{\beta - 1} \ll N^{\beta}.\] (Note that the implied constant \emph{is} uniform in $\beta$ in the range $\frac{1}{2} \leq \beta \leq 1$.)

For the second estimate in \eqref{three-ests}, we apply the Kuzmin-Landau bound (Lemma \ref{kl-bound}) with $f(t) = \frac{\gamma}{2\pi} \log t + \theta t$, we see that if $M:=  \frac{|\gamma|}{|\theta|}$ then 
\[ \sum_{M \leq n \leq x} n^{i \gamma} e(\theta n) \ll | \theta |^{-1}.\] Estimating the contribution from $M \leq \frac{|\gamma|}{|\theta|}$ trivially using $|n^{i \gamma} e(\theta n)| \leq 1$ gives $S(\gamma, \theta, x)  \ll  (1 + |\gamma|) | \theta|^{-1}$. Substituting into \eqref{partial-sum} gives the second estimate in \eqref{three-ests}.

For the third estimate in \eqref{three-ests}, suppose that $I \subset [2^j, 2^{j+1}]$ is some interval. Apply the van der Corput bound (Lemma \ref{vdc-bd}) on this interval with $f(t) = \frac{\gamma}{2\pi} \log t + \theta t$. Since $f''(t) = - \frac{\gamma}{2\pi t^2}$ we can take $\lambda = 2^{-2j - 2} |\frac{\gamma}{2\pi}|$, $\alpha = 4$ and obtain the estimate
\[ \sum_{n \in I} n^{i \gamma} e(n \theta) \ll |\gamma|^{1/2} + 2^j|\gamma|^{-1/2}.\] Summing over $j$ gives
\[ S(\gamma, \theta, x) \ll |\gamma|^{1/2} \log N + N|\gamma|^{-1/2},\] uniformly for $x \leq N$.
Substituting this into \eqref{partial-sum} gives 
\[ \sum_{n \leq N} n^{\rho - 1} e(\theta n) \ll N^{\beta}|\gamma|^{-1/2} + |\gamma|^{1/2} N^{\beta - 1} \log N,\] which certainly implies the stated bound if $|\gamma| \leq N^{1/2}$ (in fact the weaker condition $|\gamma| \leq \frac{N}{\log N}$ would suffice).\end{proof}

\subsection{Minor arcs estimates}\label{minor-lam-q}

We now come to the heart of the proof of Proposition \ref{approx-fourier}. We look first at the minor arcs, obtaining estimates for those $\theta$ not close to a rational with small denominator. Recall that the notation $X \lessapprox Y$ means $|X| \leq N^{o(1)}Y$; in this section, the $N^{o(1)}$ terms are just fixed powers of $\log N$, but this is not important for our purposes in this paper.

The following is our key minor arc estimate, which we phrase in the contrapositive.

\begin{lemma}\label{minor-lam-sharp}
Suppose that $\sigmax \leq \frac{1}{5}$ and that $\max(Q,\frac{1}{\delta}) \leq N^{\frac{1}{8}}$. Suppose that, for some $\theta \in \R/\Z$, \[ \big|\sum_{n \leq N} \Lambda_{\sharp, Q,\sigmax}(n) e(n \theta)\big| \geq \delta N.\]
Then there is some integer $r$, $0 < |r| \lessapprox \delta^{-3}$, such that $\Vert \theta r \Vert \lessapprox \delta^{-3}N^{-1}$.\end{lemma}
\begin{proof}
From the definition \eqref{lam-sharp-def} of $\Lambda_{\sharp, Q,\sigmax}$ (or, more precisely, from \eqref{lam-sharp-conv}) it follows that 
\[ \sum_{\rho \in \Xi_Q^*(\sigmax)} \big|\sum_{n \leq N} n^{\rho - 1} F_{\chi_{\rho}, Q}(n) e(n \theta)\big| \geq \delta N.\]
Recall that we set $\Xi_Q^*(\sigmax) = \Xi_Q(\sigmax) \cup \{1\}$ and we adopt the convention that $\chi_1$ is the principal character. By Lemma \ref{crude-zero-density} it follows that there is some $\rho \in \Xi_Q^*(\sigmax)$ such that 
\[ \big|\sum_{n \leq N} n^{\rho-1  } F_{\chi_{\rho}, Q}(n)e(\theta n)\big| \gg \delta N^{\beta},\]
where $\beta = \Re \rho$. Fix this $\rho$ and, for notational brevity, write $\chi = \chi_{\rho}$. As usual, write $\rho = \beta + i \gamma$, and let $q$ be the conductor of $\chi$. Note that $q \leq Q$ (or else $\rho$ would not appear in $\Xi^*_Q(\sigmax)$) and also that $|\gamma| \leq Q$. Substituting in the definition \eqref{fourier-trunc} of $F_{\chi, Q}(n)$ and using the triangle inequality, we obtain
\begin{equation}\label{nugget-2} \sum_{\substack{ q | r \\ r/q \leq Q}}\sum_{b \in (\Z/r\Z)^*} |c_{\chi}(b,r)| \bigg|  \sum_{n \leq N} n^{\rho - 1} e\big((\theta + \frac{b}{r})n\big)\bigg| \gg \delta N^{\beta}.\end{equation}
From the definition \eqref{c-def} of $c_{\chi}(b,r)$ and the basic bound 
\begin{equation}\label{gauss-square} |\tau(\chi)|^2 = q.\end{equation}
for the Gauss sum of a primitive character $\chi$ we have \begin{equation}\label{cj-pointwise} |c_{\chi}(b, r)| \leq \frac{q^{1/2}}{\phi(r)} .\end{equation} Also, $c_{\chi}(b,r)$ is supported where $(q, r/q) = 1$ (and where $r/q$ is squarefree). Therefore from \eqref{nugget-2} we have
\begin{equation}\label{nugget-3} q^{1/2} \sum_{\substack{q | r \\ r/q \leq Q \\ (q, r/q) = 1}} \frac{1}{\phi(r)} \sum_{b \in (\Z/r\Z)^*} \bigg|  \sum_{n \leq N} n^{\rho - 1} e\big((\theta + \frac{b}{r})n\big)\bigg| \gg \delta N^{\beta}.\end{equation}
Now we use Proposition \ref{prop43} to get an upper bound on 
\[ S := \sum_{b \in (\Z/r\Z)^*} \bigg|  \sum_{n \leq N} n^{\rho - 1} e\big((\theta + \frac{b}{r})n\big)\bigg| .\]
For the (at most two) values of $b$ with $\Vert \theta + \frac{b}{r} \Vert \leq \frac{1}{r}$, we use the first and third bounds in \eqref{three-ests}, to get a contribution of $\ll N^{\beta} \min(1,|\gamma|)^{-\frac{1}{2}}$. For all other values of $b$ we instead use the second bound in \eqref{three-ests}, obtaining a contribution of
\[ \ll N^{\beta-1} (1 + |\gamma|) \sum_{\substack{b \in (\Z/r\Z)^* \\ \Vert \theta + \frac{b}{r} \Vert \geq \frac{1}{r}}} \Vert \theta + \frac{b}{r} \Vert^{-1} \ll N^{\beta-1} (1 + |\gamma|) r\log r,\] uniformly in $\theta$. Now since $|\gamma| \leq Q$, $r \leq Q^2$ and $Q \leq N^{\frac{1}{8}}$, this is dominated by the bound $N^{\beta} \min(1, |\gamma|^{-\frac{1}{2}})$, and so 
\begin{equation}\label{S-bd} S \ll N^{\beta} \min(1, |\gamma|^{-\frac{1}{2}}).\end{equation}

Since
\[ q^{1/2}  \sum_{\substack{q | r \\ r/q \leq Q \\ (q, r/q) = 1}} \frac{1}{\phi(r)} \leq \frac{q^{1/2}}{\phi(q)} \sum_{k \leq Q} \frac{1}{\phi(k)} \ll \frac{q^{1/2}}{\phi(q)} \log Q,\]
comparing with \eqref{S-bd} and \eqref{nugget-3} gives
\[  \frac{q^{1/2} \log Q}{\phi(q)} \min(1, |\gamma|^{-\frac{1}{2}}) \gg \delta.\]
Using $\phi(q) \gtrapprox q$ and $q \geq 1$, and considering the cases $|\gamma| \leq 1$ and $|\gamma| \geq 1$ separately, this gives
\begin{equation}\label{qj-J-bd} q^{-\frac{1}{2}} (1 + |\gamma|)^{-1} \gtrapprox \delta^2,\end{equation} an inequality we will use shortly.

Returning to \eqref{nugget-3}, and dropping the conditions $q | r$ and $(q,r/q) = 1$, we have
\[  \sum_{r \leq Q^2} \frac{1}{\phi(r)} \sum_{b \in (\Z/r\Z)^*} \bigg| \sum_{n \leq N} n^{\rho - 1} e\big((\theta + \frac{b}{r}) n\big) \bigg| \gg \delta N^{\beta} q^{-\frac{1}{2}}.\] By the first and second estimates in \eqref{three-ests}, this implies
\begin{align}\nonumber \sum_{r \leq Q^2} \frac{1}{\phi(r)} \sum_{b \in (\Z/r\Z)^*}   \min\big(N,   \Vert \theta + \frac{b}{r} \Vert^{-1}\big) & \gg \delta N q^{-\frac{1}{2}} (1 + |\gamma|)^{-1} \\ & \gtrapprox \delta^3 N,\label{nugget-4}\end{align}
where the second step follows from \eqref{qj-J-bd}.

By Dirichlet's theorem, we may write
\[ \theta = \frac{a}{d} + \eta, \quad \mbox{where} \quad d \leq 2Q^2, \quad |\eta| \leq \frac{1}{2dQ^2},\] with $\frac{a}{d}$ in lowest terms.
If $\frac{a}{d} \neq -\frac{b}{r}$, this means that
\[ \Vert \theta + \frac{b}{r} \Vert \geq \Vert \frac{a}{d} + \frac{b}{r}\Vert - \frac{1}{2dQ^2} \geq \frac{1}{dr}  - \frac{1}{2dQ^2} \geq \frac{1}{2rd}.\]
Hence, if $r \neq d$ then we we have
\[ \sum_{b \in (\Z/r\Z)^*} \Vert \theta + \frac{b}{r}\Vert^{-1} \ll dr + r \log r \ll Q^4,\] and so the contribution to the left-hand side of \eqref{nugget-4} from $r \neq d$ will be bounded above by $Q^4 \sum_{r \leq Q^2} \frac{1}{\phi(r)} \lessapprox Q^4$. By our choice of parameters, this is negligible compared to $\delta^3 N$.
Thus the bulk of the contribution to the left-hand side of \eqref{nugget-4} comes from $r = d$, that is to say
\begin{equation}\label{nugget-5}   \frac{1}{\phi(d)} \sum_{b \in (\Z/d\Z)^*}   \min\big(N, \Vert \theta + \frac{b}{d} \Vert^{-1}\big)  \gtrapprox \delta^3 N.\end{equation}

Now
\[ \sum_{\substack{b \in (\Z/d\Z)^* \\ b \neq -a}} \Vert \theta + \frac{b}{d} \Vert^{-1} \ll d \log d \lessapprox Q^2.\]
Therefore the bulk of the contribution to the left-hand side of \eqref{nugget-5} comes from $b = -a$, and so
\[  \frac{1}{\phi(d)} \min\big(N, \Vert \theta - \frac{a}{d} \Vert^{-1}\big)  \gtrapprox \delta^3 N.\]

This implies both $d \lessapprox \delta^{-3}$ and $\Vert \theta d \Vert \lessapprox \delta^{-3} N^{-1}$, that is to say the two statements claimed.
This completes the proof.
\end{proof}

Combining Lemma \ref{minor-lam-sharp} with standard minor arc estimates for the von Mangoldt function confirms Proposition \ref{approx-fourier} unless $\theta$ is close to a rational with small denominator (that is, in the minor arc case), as the following corollary shows.

\begin{corollary}\label{minor-arc-approx}
Suppose that $\sigmax < \frac{1}{5}$ and that $\max(Q, \frac{1}{\delta})\leq N^{\frac{1}{8}}$. Suppose that, for some $\theta \in \R/\Z$, 
\[ \big|\sum_{n \leq N} (\Lambda(n) - \Lambda_{\sharp, Q,\sigmax}(n)) e(n \theta)\big|  \geq \delta N.\]
Then there is some integer $r$, $0 < |r| \lessapprox \delta^{-3}$, such that $\Vert \theta r \Vert \lessapprox \delta^{-3} N^{-1}$.
\end{corollary}
\begin{proof}
One of $|\sum_{n \leq N} \Lambda(n) e(\theta n)|$ and $|\sum_{n \leq N} \Lambda_{\sharp, Q,\sigmax}(n) e(\theta n)|$ is at least $\frac{\delta}{2} N$, by the triangle inequality. In the latter case, the conclusion follows immediately from Lemma \ref{minor-lam-sharp}. Suppose, then, that we are in the former case, so
\begin{equation}\label{lam-assump} \big|\sum_{n \leq N} \Lambda(n) e(\theta n)\big| \geq \frac{\delta}{2} N.\end{equation}
In \cite[Chapter 25]{davenport} it is shown that if $|\theta - \frac{a}{q}| \leq q^{-2}$ then
\begin{equation}\label{vino-est} \sum_{n \leq N} \Lambda(n) e(\theta n) \ll \big(\frac{N}{q^{1/2}} + N^{4/5} + N^{1/2}q^{1/2}\big)\log^4 N.\end{equation}
Let $X = c\delta^2 N \log^{-8} N$ for some $c > 0$ which we will specify shortly. By Dirichlet's theorem, there is some $q \leq X$ such that $|\theta - \frac{a}{q}| \leq \frac{1}{qX} \leq \frac{1}{q^2}$. By the assumption that $\delta \geq N^{-\frac{1}{8}}$ (which is a comfortable one, but stronger than necessary) and that $N$ is sufficiently large, the middle term in \eqref{vino-est} is $\leq \delta N/6$. By the choice of $X$ (if $c$ is small enough) the last term in \eqref{vino-est} is also $\leq \delta N/6$. It follows from these observations that if \eqref{lam-assump} holds then the first term in \eqref{vino-est} is $\geq \delta N/6$, which implies that $q \lessapprox \delta^{-2}$. 
Note also that $\Vert \theta q\Vert \leq X^{-1} \lessapprox \delta^{-2} N^{-1}$. Thus the conclusion of Corollary \ref{minor-arc-approx} holds in this case also (in fact with slightly superior exponents).
\end{proof}

\subsection{Major arc estimates}

In this section, we handle the cases of Proposition \ref{approx-fourier} not covered by Corollary \ref{minor-arc-approx}. The main result is the following.

\begin{lemma}[Major arcs] \label{major-arc-lams} Suppose that $\sigmax \leq \frac{1}{9}$ and that $Q \leq N^{\sigmax}$.
Suppose that there is some positive integer $d \leq Q^{\frac{1}{2}}$ such that $\Vert \theta d\Vert \leq Q^{\frac{1}{2}} N^{-1}$. Then 
\begin{equation}\label{lam-lam-sharp-maj} \sum_{n \leq N} (\Lambda(n) - \Lambda_{\sharp, Q,\sigmax}(n))e(\theta n)  \lessapprox N Q^{-\frac{1}{4}}.   \end{equation}
\end{lemma}

Before proceeding to the proof, let us see how Corollary \ref{minor-arc-approx} and Lemma \ref{major-arc-lams} together imply Proposition \ref{approx-fourier}. 

\begin{proof}[Proof of Proposition \ref{approx-fourier}] Let $\eps > 0$ be arbitrary. Let $\delta = Q^{-\frac{1}{6}} N^{\eps}$, and suppose that for some $\theta$ we have \[ \big|\sum_{n \leq N} (\Lambda(n) - \Lambda_{\sharp, Q,\sigmax}(n)) e(n \theta)\big| \geq \delta N.\]  Then, by Corollary \ref{minor-arc-approx}, there is some integer $r$, $0 < |r| \lessapprox \delta^{-3}$, such that $\Vert \theta r \Vert \lessapprox \delta^{-3} N^{-1}$. If $N$ is big enough in terms of $\eps$, $|r| \leq Q^{\frac{1}{2}}$ and $\Vert \theta r\Vert \leq Q^{\frac{1}{2}} N^{-1}$, so we are in a position to apply Lemma \ref{major-arc-lams}, and this tells us that $N Q^{-\frac{1}{4}} \gtrapprox \delta N$, which is a contradiction for $N$ sufficiently large in terms of $\eps$. Since $\eps$ was arbitrary, Proposition \ref{approx-fourier} follows.\end{proof}

To prove Lemma \ref{major-arc-lams}, we will need some preliminaries on Dirichlet characters and on the coefficients $c_{\chi}$ as defined in \eqref{c-def}. In particular, we will see why these coefficients play a key role in the paper. Here and in what follows, $\cond(\chi)$ means the conductor of the primitive Dirichlet character $\chi$.

\begin{lemma}\label{dirich-add}
Let $r$ be an integer and $b \in (\Z/r\Z)^*$. Then
\begin{equation}\label{dir-add-eq} e\big(-\frac{b n}{r}\big) = \sum_{\chi : \cond(\chi) | r} c_{\chi}(b,r) 1_{(n,r) = 1} \chi(n),\end{equation} where the sum is over all primitive Dirichlet characters $\chi$.
\end{lemma}
\begin{proof}
For any character $\psi$ to modulus $r$ we have
\begin{equation}\label{gauss-sum-ext}\sum_{t \in (\Z/r\Z)^*} \psi(t) e\big(\frac{bt}{r}\big) = \overline{\psi(b)} \tau(\psi).\end{equation}
By \cite[Lemma 3.1]{ik}, if $\psi$ is induced by a primitive character $\chi$ with conductor $q$ then
\[ \tau(\psi) = \mu\big(\frac{r}{q}\big) \chi\big(\frac{r}{q}\big) \tau(\chi).\]
Therefore by orthogonality of characters, and writing $\sum_{\psi \mdsub{r}}$ for the sum over all $\phi(r)$ Dirichlet characters to modulus $r$, we have
\begin{align*} e\big(-\frac{b n}{r}\big) &  = \frac{1}{\phi(r)} \sum_{\psi \mdsub{r}} \big( \overline{\sum_{t \in (\Z/r\Z)^*} \psi(t) e\big(\frac{b t}{r}\big)} \big) \psi(n)  \\ & = \frac{1}{\phi(r)}\sum_{\psi \mdsub{r}} \psi(b) \overline{\tau(\psi)} \psi(n)  \\ & = \frac{1}{\phi(r)} \sum_{\substack{\chi : \cond(\chi) | r \\ \mbox{\scriptsize $\chi$ primitive}}}  \chi(b) \mu\big(\frac{r}{q}\big) \overline{\chi\big(\frac{r}{q}\big)} \overline{\tau(\chi)} 1_{(n,r)=1} \chi(n). \end{align*}
To get the last line, for each $\psi$ we consider the primitive character $\chi$ which induces it, noting that the induced character $\psi$ to modulus $r$ is then $1_{(n,r) = 1} \chi(n)$. The result follows immediately by the definition of $c$, namely \eqref{c-def}.\end{proof}

We will also need a version of the explicit formula. If $\chi$ is a primitive Dirichlet character, denote by $Z(\chi)$ the multiset of nontrivial zeros of $L(s,\chi)$ in the strip $0 < \Re s < 1$, together with the pole at $1$ if $\chi$ is principal. Note that $\Xi^*_Q(\sigmax)$ is then the disjoint union of the (multi-)sets $Z(\chi) \cap \Xi^*_Q(\sigmax)$, as $\chi$ ranges over all primitive $\chi$ with $\cond(\chi) \leq Q$; this follows from the definition of $\Xi_Q(\sigmax)$, Definition \ref{zero-ht-def}.

\begin{lemma}\label{explicit-formula-lem}
Suppose $\sigma_0 \leq \frac{1}{2}$. Let $\psi$ be a Dirichlet character \textup{(}not necessarily primitive\textup{)} to some modulus $r$, $r \leq Q \leq x$, and let $\chi$ be the primitive character inducing $\psi$. Then 
\begin{equation}\label{explicit-formula} \sum_{n \leq x} \Lambda(n) \psi(n) = \sum_{\rho \in Z(\chi) \cap \Xi^*_Q(\sigmax)} \eps_{\rho} \frac{x^{\rho}}{\rho} + \tilde{O}\big(x^{1 - \sigmax}+\frac{x}{Q}\big).\end{equation}
Here, $\eps_1= 1$ and $\eps_{\rho} = -1$ when $\rho$ is a zero.
\end{lemma}
Recall that $\tilde O(Y)$ means a quantity that is bounded in magnitude by $N^{o(1)} Y$. 
\begin{proof} 
This is a standard result (for instance, it appears as \cite[Exercise 45 (iv)]{tao-blog}), but it is not straightforward to find a completely quotable reference. First, we observe that one may quickly reduce to the primitive case $\psi = \chi$. Indeed, the sums over $\rho$ are the same since $L(s,\chi)$ and $L(s,\psi)$ have the same nontrivial zeros (as they differ only by a finite product of terms $1 - p^{-s}$), whilst on the other hand $\sum_{n \leq x} \Lambda(n) \psi(n)$ and $\sum_{n \leq x} \Lambda(n) \chi(n)$ differ by $\ll \log^2 x$, as explained in the displayed equation following \cite[Chapter 19, equation (12)]{davenport}.

Suppose, then, that $\psi = \chi$ is primitive. Suppose first that $\chi$ is not the principal character $1$. Then \cite[Chapter 19, equations (13) to (15)]{davenport} state that
\begin{equation}\label{ten-quote} \sum_{n \leq x} \Lambda(n) \chi(n) = - \sideset{}{^*}\sum_{|\Im \rho| \leq Q} \frac{x^{\rho}}{\rho} + O\big(\frac{x (\log rx)^2}{Q} + x^{1/4}\log x\big).\end{equation} Here the sum is over all zeros of $L(s,\chi)$ in the critical strip $0 < \Re s < 1$, and the asterisk denotes that if $\chi$ is real has a simple real zero $\beta_1$ with $\beta_1 \geq 1 - \frac{c}{\log Q}$, then $\rho = 1 - \beta_1$ is omitted from the sum. To derive \eqref{explicit-formula}, we must remove the contribution of the zeros $\rho$ with $0 < \Re \rho < 1 - \sigmax$, which do not lie in $\Xi_Q$. To estimate this, divide the set of such $\rho$ into sets $\Omega_k$, $k = 0,1,\dots$, where $2^k \leq |\Im \rho| \leq 2^{k+1}$, and the set $\Omega_*$ where $|\Im \rho| \leq 1$. By standard estimates for the number of zeros (see, for example, \cite[Chapter 16]{davenport}) we have $|\Omega_k| \lessapprox 2^k$, and $|\Omega_*| \lessapprox 1$. If $\rho \in \Omega_k$, $|\frac{x^{\rho}}{\rho}| \ll 2^{-k}x^{1 - \sigmax}$, so $\sum_{\rho \in \Omega_k} |\frac{x^{\rho}}{\rho}| \lessapprox x^{1 - \sigmax}$. If $\rho \in \Omega_*$, we have $|\frac{x^{\rho}}{\rho}| \ll x^{1 - \sigmax} \log Q$ (since $\frac{1}{|\rho|} \ll \log Q$ for all the zeros under consideration). Summing the contributions from $\Omega_*$ and all the $\Omega_k$ with $k \leq \log_2 X$, we obtain the stated result. 

Suppose now that $\chi$ is the principal character. Here, \eqref{ten-quote} still holds provided we remember to include a main term of $x$ coming from the simple pole of $\zeta$, as follows from \cite[Chapter 18, equations (9) and (10)]{davenport}. 
\end{proof}

We turn now to the proof of Lemma \ref{major-arc-lams} itself.

\begin{proof}[Proof of Lemma \ref{major-arc-lams}]
Recall the definition of $\Lambda_{\sharp, Q,\sigmax}$ in the form \eqref{lam-sharp-conv}, viz
\begin{equation}\label{lem58-star} \Lambda_{\sharp, Q,\sigmax}(n) = \sum_{\rho \in \Xi_Q^*(\sigmax)} \eps_{\rho} n^{\rho - 1} F_{\chi_{\rho}, Q}(n),\end{equation} where (Definition \ref{def43})
\[ F_{\chi, Q}(n) := \sum_{\substack{\cond(\chi) | r \\ r \leq Q\cond(\chi)} } \sum_{b \in (\Z/r\Z)^*} c_{\chi}(b,r) e\big(\frac{bn}{r}\big),\] 
with $c_{\chi}$ defined in \eqref{c-def}, and $\cond(\chi)$ denotes the conductor of $\chi$.

Write $\theta = -\frac{a}{d} + \eta$ with $d$ as in the statement of Lemma \ref{major-arc-lams}. Replacing $d$ by a divisor of itself if necessary, we may assume that $\frac{a}{d}$ is in lowest terms. By the assumptions of Lemma \ref{major-arc-lams}, 
\begin{equation}\label{eta-upper} |\eta| \leq  \frac{Q^{\frac{1}{2}}}{dN} < \frac{1}{2Q^3}.\end{equation}
Multiplying \eqref{lem58-star} through by $e(n\theta)$, summing over $N$ and interchanging the order of summation, we obtain 

\begin{align}\nonumber  \sum_{n \leq N} & \Lambda_{\sharp, Q,\sigmax}(n) e( \theta n) \\ & = \sum_{r \leq Q^2} \sum_{b \in (\Z/r\Z)^*} \sum_{\substack{\rho \in \Xi_Q^*(\sigmax) \\  q_{\rho} | r \\ r/q_{\rho} \leq Q}} \eps_{\rho}  c_{\chi_{\rho}}(b,r)  \sum_{n \leq N} n^{\rho - 1} e\big((\frac{b}{r} - \frac{a}{d} + \eta) n\big).\label{lam-maj-exp} \end{align}
Here, we have written $q_{\rho} := \cond(\chi_{\rho})$ for brevity.

We claim that the contribution from $\frac{b}{r} \neq \frac{a}{d}$ is $\lessapprox Q^5$, which can be subsumed into the right hand side of the desired bound \eqref{lam-lam-sharp-maj} because of the assumption that $Q \leq N^{\frac{1}{9}}$. To prove this claim, we use the second bound in \eqref{three-ests}, the fact that $|\Im \rho| \leq Q$, and the trivial bound $|c_{\chi_{\rho}}(b,r)| \leq 1$ to see that this contribution is bounded by
\[ Q \sum_{\substack{r \leq Q^2 \\ b \in (\Z/r\Z)^*\\ \frac{b}{r} \neq -\frac{a}{d}}} \Vert \frac{b}{r} - \frac{a}{d} + \eta\Vert^{-1} \sum_{\rho \in \Xi_Q^*(\sigmax)} N^{\Re \rho - 1} \ll Q \sum_{\substack{r \leq Q^2 \\ b \in (\Z/r\Z)^*\\ \frac{b}{r} \neq -\frac{a}{d}}}  \Vert \frac{b}{r} - \frac{a}{d} + \eta\Vert^{-1}, \] where the inequality follows from Lemma \ref{crude-zero-density}. Now if $\frac{b}{r} \neq -\frac{a}{d}$ then $\Vert \frac{b}{r} - \frac{a}{d} + \eta \Vert \geq \frac{1}{rd} - |\eta| > \frac{1}{2}Q^{-3}$, since $r \leq Q^2$, $d < Q$ and $|\eta| \leq (2Q^3)^{-1}$ by \eqref{eta-upper}. Since the different $\frac{b}{r}$ are $Q^{-4}$-separated, 
\[ \sum_{\substack{r \leq Q^2 \\ b \in (\Z/r\Z)^*\\ \frac{b}{r} \neq -\frac{a}{d}}}  \Vert \frac{b}{r} - \frac{a}{d} + \eta\Vert^{-1} \ll Q^4\log Q.\]
Thus the contribution of $\frac{b}{r} \neq -\frac{a}{d}$  to \eqref{lam-maj-exp} is $\ll Q^5 \log Q \lessapprox Q^5$, as claimed. It therefore follows from \eqref{lam-maj-exp} that 
\begin{equation}\label{compare-1}  \sum_{n \leq N} \Lambda_{\sharp, Q,\sigmax}(n) e(\theta n) = \sum_{\rho \in \Xi_Q^*(\sigmax)} \eps_{\rho} c_{\chi_{\rho}}(a,d)\sum_{n \leq N} n^{\rho - 1} e(\eta n) + \tilde{O}(Q^5).\end{equation}
(Note that the conditions $q_{\rho} \mid d$ and $d/q_{\rho} \leq Q$ are redundant, since $c_{\chi_{\rho}}(a,d) = 0$ unless the first is satisfied, and $d \leq Q^{\frac{1}{2}}$ implies the second is automatically satisfied.)

Now we look at the Fourier transform of the von Mangoldt function $\Lambda$ itself. The key input is the explicit formula, Lemma \ref{explicit-formula-lem}. 
We want to twist this by $e(\eta n)$ for small $\eta$, which can be done by partial summation. Specifically, observe that
\[ e(\eta n) = e (\eta N) - 2\pi i \eta \int^N_n e(\eta t ) dt.\] Multiplying by $\Lambda(n) \psi(n)$, summing over $n \leq N$, interchanging the order of integration and summation and substituting $t = Nu$ gives
\begin{align*} \sum_{n \leq N} \Lambda(n) \psi(n) e(\eta n)  = e(\eta N) & \sum_{n \leq N}  \Lambda(n) \psi(n) - \\ & - 2 \pi i \eta N \int^{1}_0 e(\eta Nu) \big( \sum_{n \leq Nu} \Lambda(n) \psi(n)\big) du.\end{align*} Applying the explicit formula \eqref{explicit-formula} and using the identity
\[ \int^1_0 x^{\rho - 1} e(\eta N x) dx = \frac{1}{\rho} e(\eta N) - 2\pi i \eta N \int^1_0 \frac{x^{\rho}}{\rho} e(\eta N x) dx\] (which may be established by integration by parts)
we obtain
\begin{align}\nonumber& \sum_{n \leq N} \Lambda(n) \psi(n) e(\eta n) \\ & = \sum_{\rho \in Z(\psi) \cap \Xi^*_Q(\sigmax)} \eps_{\rho}N^{\rho} \int^1_0 x^{\rho - 1} e(\eta N x) dx + \tilde{O}\big((1 + |\eta N|) (N^{1 - \sigmax} + \frac{N}{Q})\big) \nonumber \\ & = \sum_{\rho \in Z(\psi) \cap \Xi^*_Q(\sigmax)} \eps_{\rho}N^{\rho} \int^1_0 x^{\rho - 1} e(\eta N x) dx + \tilde{O}(NQ^{-\frac{1}{2}})\label{twist-1}\end{align} by the assumption that $Q \leq N^{\sigmax}$, and since $|\eta| \leq Q^{\frac{1}{2}} N^{-1}$.

We want to approximate the integral here by a sum, to which end we use the crude approximation
\[ \int^b_a w(x) dx = \frac{1}{N} \sum_{aN \leq n < bN} w\big(\frac{n}{N}\big) + O\big(\frac{1}{N}(\Vert w \Vert_{L^{\infty}(a,b)} + \Vert w' \Vert_{L^{\infty}(a,b)})\big),\]
valid for any interval $[a,b]$ of length at most $1$ and for any differentiable function $w : (a,b) \rightarrow \C$ with bounded derivative. We take $w(x) := x^{\rho - 1} e(\eta N x)$ and $[a,b] = [N^{-\frac{1}{2}}, 1]$. Then $\Vert w \Vert_{L^{\infty}(a,b)} \leq N^{\frac{1}{2}-\frac{1}{2}\beta}$ (where $\rho = \beta + i \gamma$) and $\Vert w' \Vert_{L^{\infty}(a,b)} \ll N^{1 - \frac{1}{2}\beta} (1 + |\gamma|) + |\eta| N \cdot N^{\frac{1}{2} - \frac{1}{2}\beta} \ll N^{1-\frac{1}{2}\beta}Q + Q^{\frac{1}{2}} N^{\frac{1}{2} - \frac{1}{2} \beta} \ll N^{1 - \frac{1}{2}\beta}Q$, so we get
\[ \int^1_{N^{-\frac{1}{2}}} x^{\rho - 1} e(\eta N x) dx  = N^{-\rho} \sum_{N^{\frac{1}{2}} \leq n < N} n^{\rho - 1} e(\eta n) + O(N^{-\frac{1}{2}\beta} Q).\]
Since also
\[ \big| \int^{N^{-\frac{1}{2}}}_0 x^{\rho - 1}e(\eta N x) dx\big| \leq \int^{N^{-\frac{1}{2}}}_0 |x|^{\beta - 1} dx \ll N^{-\frac{1}{2}\beta}\]
and
\[ \big|N^{-\rho} \sum_{1 \leq n < N^{\frac{1}{2}}} n^{\rho - 1} e(\eta n)\big| \leq N^{-\beta} \sum_{1 \leq n < N^{\frac{1}{2}}} |n|^{\beta - 1} \ll N^{-\frac{1}{2}\beta},\]
we have 
\[ \int^1_0 x^{\rho-1} e(\eta N x) dx = N^{-\rho} \sum_{n = 1}^N n^{\rho - 1} e(\eta n) + O(N^{-\frac{1}{2}\beta} Q).\]
Now \eqref{twist-1} implies that
\begin{align}\nonumber \sum_{n \leq N} \Lambda(n) & \psi(n) e(\eta n)  = \sum_{\rho \in Z(\psi) \cap \Xi^*_Q(\sigmax)} \eps_{\rho} \sum_{n = 1}^N n^{\rho - 1} e(\eta N) + \\ & + O\big( Q \sum_{\rho \in Z(\psi) \cap \Xi^*_Q(\sigmax)} N^{\frac{1}{2}\beta} \big) + \tilde{O}(N Q^{-\frac{1}{2}}).\label{twist-2}  \end{align}
To estimate the first error term, we use the zero-density bound $|Z(\psi) \cap \Xi^*_Q(\sigmax)| \leq 1 + |\Xi_Q(\sigmax)| \ll Q$ (which follows from Proposition \ref{jut-est}, since $\sigmax \leq \frac{1}{8}$). This shows that the total error term in \eqref{twist-2} is $\tilde{O}( N^{\frac{1}{2}} Q^2 + NQ^{-\frac{1}{2}}) = \tilde{O} (N Q^{-\frac{1}{2}})$ since $Q < N^{\frac{1}{5}}$.

Now \eqref{dir-add-eq} gives
\[ e\big(-\frac{a n}{d}\big) = \sum_{\chi : \cond(\chi) | d} c_{\chi}(a,d) 1_{(n,d) = 1} \chi(n),\] where the sum is over primitive $\chi$. Recalling that $\theta = -\frac{a}{d} + \eta$, it follows that
\begin{align*} \sum_{n \leq N} \Lambda(n) e(n\theta) & = \sum_{\chi: \cond(\chi) | d} c_{\chi}(a,d) \sum_{\substack{n \leq N \\ (n,d) = 1}} \Lambda(n) \chi(n) e(\eta n) \\ & = \sum_{\chi: \cond(\chi) | d} c_{\chi}(a,d) \sum_{n \leq N} \Lambda(n) \chi(n) e(\eta n) + \tilde{O}(Q^{\frac{1}{4}}), \end{align*} where for the second bound we have used $\sum_{n \leq N : (n,d) \neq 1} \Lambda(n) \lessapprox 1$ and the bound
\begin{equation}\label{use-twice} \sum_{\chi : \cond(\chi) | d} |c_{\chi}(a,d)| \leq \phi(d) \frac{d^{\frac{1}{2}}}{\phi(d)} \leq Q^{\frac{1}{4}},\end{equation} which follows from \eqref{cj-pointwise} and the assumption that $d \leq Q^{\frac{1}{2}}$. Substituting into \eqref{twist-2} (with $\psi = \chi$, and using the fact that the total error term in \eqref{twist-2} is $\tilde{O}(N Q^{-\frac{1}{2}})$ as explained after \eqref{twist-2}) gives

\begin{align*} \sum_{n \leq N} \Lambda(n) e(\theta n)  = \sum_{\substack{\rho \in \Xi_Q^*(\sigma_0) \\ \cond(\chi_{\rho}) | d}} \eps_{\rho}& c_{\chi_{\rho}}(a,d)  \sum_{n = 1}^N n^{\rho - 1} e(\eta n) \\ & + \tilde{O}\big(Q^{\frac{1}{4}} + NQ^{-\frac{1}{2}}\sum_{\chi : \cond(\chi) | d} |c_{\chi}(a,d)|\big).\end{align*}

Using \eqref{use-twice} again, we see that the error term is $\tilde{O}(N Q^{-\frac{1}{4}})$. Note also that, in the sum over $\rho$, the condition that $\cond(\chi_{\rho}) \mid d$ can be omitted since $c_{\chi}(a,d) = 0$ when $\cond(\chi) \nmid d$.

Comparing with \eqref{compare-1}, and since $Q < N^{\frac{1}{6}}$, the proof of Lemma \ref{major-arc-lams} is complete.
\end{proof}

\section{Completed approximants}\label{sec6}

Recall (Definition \ref{def43}) the definition 
\begin{equation}\label{f-chi-t-def-rpt} F_{\chi, Q}(n) = \sum_{r \leq qQ} \sum_{b \in (\Z/r\Z)^*} c_{\chi}(b,r) e\big(\frac{bn}{r}\big),\end{equation} where here $q := \cond(\chi)$.

If one relaxes the constraint on $r$ here, say to $r \mid R!$ for some $R > Q$, we get an object which we call a ``completed'' approximant $\tilde F_{\chi, R}$. These completed approximants are much easier to understand combinatorially, and in fact provide the motivation for studying the truncated objects such as $F_{\chi, Q}$. The interplay between completed and truncated approximants is a key theme in this paper, and this section is devoted to the basic properties of completed approximants. 

\begin{definition}\label{f-chi-definition}
Suppose that $\chi$ is a primitive Dirichlet character to modulus $q$. Let $R$ be some integer parameter. Then we define
\begin{equation}\label{f-chi-def} \tilde F_{\chi, R}(n) := \sum_{r | R! } \sum_{b \in (\Z/r\Z)^*} c_{\chi}(b,r) e\big(\frac{bn}{r}\big).\end{equation}
\end{definition}

 At this point we introduce the functions $\vp $ defined, for prime $p$, by 
 \begin{equation}\label{vp-definition} \vp (n) := \left\{ \begin{array}{ll} 0 & \mbox{if $p | n$} \\ (1 - \frac{1}{p})^{-1} & \mbox{if $(p,n) = 1$} \end{array}\right. \end{equation}
These functions will feature quite heavily in the paper. The notation $\Lambda_{\Z/p\Z}$ is (somewhat) standard, and reflects the fact that $\Lambda_{\Z/p\Z}$ is the most natural $\md{p}$ analogue of the von Mangoldt function. 
Observe that, for $p$ prime, 
\begin{equation}\label{ram-vp} 1 - \frac{1}{p-1} c_p(n)=  1 + \frac{\mu(p)}{\phi(p)} c_p(n) = \vp (n),\end{equation} where here (recall) the definition of the Ramanujan sum $c_r(n) := \sum_{b \in (\Z/r\Z)^*} e(\frac{bn}{r})$.

\begin{lemma}\label{completed-forms}
Let $Q, R \geq 1$ be parameters. Let $\chi$ be a primitive Dirichlet character with conductor $q$. Then 
\begin{equation}\label{f-chi-comb} \tilde F_{\chi, R}(n) = \big(\prod_{p \leq R}(1 - \frac{1}{p})^{-1}\big) \overline{\chi}(n) 1_{(n, R!) = 1} = \frac{q}{\phi(q)} \overline{\chi}(n) \prod_{\substack{p \nmid q \\ p \leq R}} \vp (n)\end{equation} and 
\begin{equation}\label{f-chi-t-comb}  F_{\chi, Q}(n) = \frac{q}{\phi(q)} \overline{\chi}(n) \sum_{\substack{r \leq Q \\ (r, q) = 1}} \frac{\mu(r)}{\phi(r)} c_r(n) .\end{equation}
\end{lemma}
\begin{proof}
Recall the definition of $c_{\chi}$, given in Definition \ref{c-chi-def}. In particular $c_{\chi}(b,r)$ is supported where $q \mid r$ and where $r/q$ is squarefree and coprime to $q$. Consider the definition \eqref{f-chi-def}. To begin with, fix $r$ (with $q \mid r$ and $r/q$ squarefree) and look at the inner sum over $b$. We have $r = q r'$, with $r'$ squarefree and coprime to $q$. We may write (uniquely) $\frac{b}{r} = \frac{a}{q} + \frac{b'}{r'}$, where $a \in (\Z/q\Z)^*$ and $b' \in (\Z/r'\Z)^*$. Note that $b \equiv a r' \md{q}$ and therefore $\chi(b) = \chi(a) \chi(r')$, and so
\[ c_{\chi}(b,r) = \frac{1}{\phi(q)} \chi(a)  \overline{\tau(\chi)} \frac{\mu(r')}{\phi(r')}.\]
Therefore we may factor \[ \sum_{b \in (\Z/r\Z)^*} c_{\chi}(b,r)e\big(\frac{bn}{r}\big) = E_1 E_2\]  where \begin{equation}\label{factors}  E_1 := \frac{\overline{\tau(\chi)} }{\phi(q)}\sum_{a \in (\Z/q \Z)^*} \chi(a) e\big(\frac{a n}{q}\big) , \quad E_2 :=   \frac{\mu(r')}{\phi(r')} \sum_{b' \in (\Z/r'\Z)^*} e\big(\frac{b' n}{r'}\big) .\end{equation}
Now it is easy to see (and we have already used in this in \eqref{gauss-sum-ext}) that $\sum_{a \in (\Z/q \Z)^*} \chi(a) e(\frac{a n}{q})  = \overline{\chi(n)} \tau(\chi)$ if $(n,q) = 1$. If $\chi$ is \emph{primitive}, as here, the condition $(n,q) = 1$ can be dropped since both sides of the expression vanish if $(n,q) > 1$, as explained in \cite[Section 3.4]{ik}. Therefore, using \eqref{gauss-square}, $E_1 = \frac{q}{\phi(q)} \overline{\chi(n)}$. We may write $E_2$ more succinctly as $\frac{\mu(r')}{\phi(r')} c_{r'}(n)$.
Summarising, we have
\begin{equation} \label{c-chi-b} \sum_{b \in (\Z/r\Z)^*} c_{\chi}(b,r)e\big(\frac{bn}{r}\big) = \frac{q}{\phi(q)} \overline{\chi(n)} \frac{\mu(r')}{\phi(r')} c_{r'}(n)\end{equation} (where $r = q r'$ with $(r', q) = 1$).

The expression \eqref{f-chi-t-comb} follows immediately from the definition \eqref{f-chi-t-def-rpt} of $F_{\chi, Q}$, summing \eqref{c-chi-b} over $r' \leq Q$ (and coprime to $q$).

To obtain \eqref{f-chi-comb}, we recall the definition \eqref{f-chi-def}  of $\tilde F_{\chi, R}$, and instead sum \eqref{c-chi-b} over all $r \mid R!$. We obtain

\begin{equation}\label{caq-half} \tilde F_{\chi, R}(n) = \frac{q}{\phi(q)}\overline{\chi(n)}\sum_{\substack{r' | R! \\ (r', q) = 1}}  \frac{\mu(r')}{\phi(r')} c_{r'}(n).\end{equation}
(We will use this equation in its own right in Section \ref{sec10}.)

Using the multiplicativity property $c_{ab}(n) = c_a(n) c_b(n)$ when $(a,b) = 1$, as well as \eqref{ram-vp}, we see that sum over $r'$ factors as
\[ \prod_{\substack{p \leq R \\ p \nmid q }} \big(1 + \frac{\mu(p)}{\phi(p)}c_p(n)\big) = \prod_{\substack{p \leq R \\ p \nmid q}} \vp (n),\] from which \eqref{f-chi-comb} follows.
\end{proof}

To conclude this section, let us record explicitly what happens in the case $\chi = \chi_0$, the principal character. Recall from \eqref{hb-link} that in this case, $F_{\chi_0, Q}$ is in fact equal to $\Lambda_Q$, the classical approximant defined in \eqref{lambda-q-def}. It is therefore natural to write $\tilde\Lambda_R$ instead of $\tilde F_{\chi_0, R}$ for the corresponding completed approximant, thus

\begin{equation}\label{tilde-lam-def}  \tilde\Lambda_R(n) := \sum_{q | R!} \frac{\mu(q)}{\phi(q)} c_q(n).\end{equation} The special case $\chi = \chi_0$ of Lemma \ref{completed-forms} is then the statement that
\begin{equation}\label{tilde-lam-def-2} \tilde\Lambda_R(n) = \prod_{p \leq R} \vp (n),\end{equation} which is a normalised characteristic function of the almost primes at level $R$ (that is, numbers all of whose prime factors are $> R$). This statement also follows rather directly from \eqref{ram-vp} and multiplicativity.

\subsection{Pointwise bounds}

An application of the above expressions are pointwise bounds for the $F_{\chi, Q}$. We first prove a general fact about what we will later call Ramanujan series.

\begin{lemma}\label{lem63}
Suppose that 
\begin{equation}\label{ram-series} f(n) = \sum_{r \leq Q} v(r) c_r(n).\end{equation} Then
\begin{equation}\label{star63}  f(n) = \sum_{\substack{d | n \\ d \leq Q}} \lambda_d \quad \mbox{where} \quad \frac{\lambda_d}{d} = \sum_{\substack{d | r \\ r \leq Q}} \mu\big(\frac{r}{d}\big) v(r).\end{equation}
\end{lemma}
\begin{proof}
With $\lambda_d$ as defined in \eqref{star63} we have
\begin{equation}\label{63-min} \sum_{\substack{d | n \\ d \leq Q}} \lambda_d = \sum_{\substack{d | n \\ d \leq Q}} d \sum_{\substack{d | r \\ r \leq Q}} \mu\big(\frac{r}{d}\big) v(r) = \sum_{r \leq Q} v(r) \sum_{d | (r,n)} d \mu\big(\frac{r}{d}\big).\end{equation}
On the other hand, 
\[ 1_{d | n} = \frac{1}{d} \sum_{a \in \Z/d\Z} e\big(\frac{an}{d}\big) = \frac{1}{d} \sum_{r | d} c_r(n),\] and so by M\"obius inversion we have
\[ \sum_{d |(r,n)} d \mu\big(\frac{r}{d}\big) = c_r(n).\]
Substituting into \eqref{63-min} gives the claimed result.
\end{proof}

\emph{Remark.} The sum $\sum_{d | n : d \leq Q} \lambda_d$ is a kind of sieve weight. Conversely, any such weight can be expressed in the form \eqref{ram-series}, with $v(r) = \sum_{d \leq Q : r | d} \frac{\lambda_d}{d}$, a straightforward exercise we leave to the reader.

\begin{corollary}\label{cor6.4}
We have $|F_{\chi, Q}(n)| \ll \tau(n) \log^3 Q$ pointwise. Consequently,
\begin{equation}\label{lam-sharp-ptwise} |\Lambda_{\sharp, Q,\sigmax}(n)| \ll \tau(n) \log^3 Q \big(\sum_{\rho \in \Xi_Q^*(\sigmax)} n^{\Re \rho - 1}\big) .\end{equation}
\end{corollary}
\begin{proof}
Set $f(n) := \sum_{r \leq Q : (r,q) = 1} \frac{\mu(r)}{\phi(r)} c_r(n)$. Applying Lemma \ref{lem63} with $v(r) = \frac{\mu(r)}{\phi(r)} 1_{(r,q) = 1}$, we obtain $f(n) = \sum_{d | n : d \leq Q} \lambda_d$ with 
\[ \frac{\lambda_d}{d} = \sum_{\substack{d | r \\ r \leq Q \\ (r,q) = 1}} \frac{\mu(\frac{r}{d}) \mu(r)}{\phi(r)}  = \frac{\mu(d)}{\phi(d)} \sum_{\substack{k \leq Q/d \\ (dk, q) = 1 \\ (d,k) = 1}} \frac{\mu^2(k)}{\phi(k)} .\] (Here, of course, we wrote $r = dk$, noting further that the sum is supported where $r$ is squarefree, which means that $d$ and $k$ are coprime.)
It follows that 
\[ |\lambda_d| \leq \frac{d}{\phi(d)} \sum_{k \leq Q} \frac{1}{\phi(k)} \ll  \log^2 Q.\] Thus $|f(n)| \leq \tau(n) \log^2 Q$, from which the stated bound follows immediately using \eqref{f-chi-t-comb}. (We could, of course, easily save almost two powers of log by using less crude bounds on $\frac{m}{\phi(m)}$ if we wanted to.)

The bound \eqref{lam-sharp-ptwise} is immediate from this and the definition (Definition \ref{sharp-approx-def}) of $\Lambda_{\sharp, Q,\sigmax}$.
\end{proof}

\part{Preliminaries to the main construction}\label{partiii}

\section{zeros and density estimates}\label{sec7}

In this section we assemble some more sophisticated zero-density estimates needed for the main construction. Recall the definition (Definition \ref{zero-ht-def}) of the multiset $\Xi_Q(\sigmax)$. One main aim of this section will be to show the existence of a ``good scale'' $T$ at which any exceptional zero in $\Xi_T(\sigmax)$ (that is, a real zero of a real $L(s,\chi)$ very close to $s = 1$) has fairly small modulus. The key result of this section is Proposition \ref{sec4-takeaway}.

The manipulations in this section are rather similar to those in the standard proof of Linnik's theorem as presented, for example, in \cite[Section 18.4]{ik}. The main difference here is that we need to consider characters to all moduli $q \leq Q$, rather than just a single $q$, which means we must appeal to more general zero-density estimates.

Let us recall some facts about zeros of $L$-functions. In this section $\lambda_1,\lambda_2,\lambda_3,\lambda_4$ are fixed absolute positive constants. Once again we adopt the standard convention in analytic number theory of writing, for a nontrivial zero $\rho$ of some $L$-function, $\rho = \beta + i \gamma$. Somewhat less standardly, we always write $\sigma := 1 - \beta$, thus $\Re \rho = 1 - \sigma$.

The first fact we recall is Page's theorem, in the form stated at the beginning of \cite[Section 6]{bombieri}. See \cite[Corollary 11.10]{mv} for another reference.

\begin{proposition} \label{prop42} There is an absolute, computable constant $\lambda_1 \in (0,1]$ such that the following is true for any parameter $Q \geq 10$: with possibly one exception, all $\rho \in \Xi_Q(\frac{1}{2})$ satisfy $\sigma \geq \frac{\lambda_1}{\log Q}$. The exception, if it exists, is a simple real zero of $L(s,\chi_*)$ for some real Dirichlet character $\chi_*$ of conductor $q_* \leq Q$. This zero, if it exists, is called the \emph{exceptional zero} at scale $Q$.
\end{proposition}

The second is a result of Landau. For the proof, see for instance \cite[Theorem 11.7]{mv}.
\begin{proposition}
Suppose that $\chi_1, \chi_2$ are real characters modulo $q_1, q_2$, with $\chi_1\chi_2$ not principal and with real zeros at $1 - \sigma_1, 1 - \sigma_2$ respectively. Then 
\[ \max(\sigma_1, \sigma_2) \geq \frac{\lambda_1}{\log q_1 q_2}.\]
\end{proposition}

Using these facts together, we can locate scales $Q$ at which (essentially) any exceptional zero has relatively small modulus. 

\begin{proposition}\label{two-scales-prop}
Let $\kappa \in (0,1)$ be given. Then for any $X \geq 10^{1/\kappa}$, and for either $Q = X$ or $Q = X^{\kappa}$, at least one of the following holds:
\begin{enumerate}
\item \textup{(first case)}  all zeros in $\Xi_Q(\frac{1}{2})$ satisfy $\sigma \geq \frac{\lambda_1\kappa }{2 \log Q}$;
\item \textup{(second case)} there is an exceptional zero at scale $Q$, but its modulus $q_*$ satisfies $q_* < Q^{\kappa}$.
\end{enumerate}
\end{proposition}
\emph{Remarks.} To avoid confusion, for the time being we call these the first and second cases rather than the unexceptional and exceptional cases since the first case does not strictly preclude an exceptional zero at scale $Q$ in the sense of Proposition \ref{prop42}; there could be a zero with $\frac{\lambda_1 \kappa}{2\log Q} \leq \sigma < \frac{\lambda_1}{\log Q}$. 

Proposition \ref{two-scales-prop} is very closely related to, for example, \cite[Corollary 11.9]{mv}.

\begin{proof}[Proof of Proposition \ref{two-scales-prop}]
Consider first of all the scale $Q = X$. If there is no exceptional zero at this scale (in the sense of Proposition \ref{prop42}) then we are in case (1) and indeed we have the stronger bound $\sigma \geq \frac{\lambda_1}{\log Q}$. Suppose then that there is an exceptional zero $\rho_* = 1 - \sigma_*$ at this scale. If $\sigma_* \geq \frac{\lambda_1 \kappa}{2 \log Q}$ then we are again in case (1). If $q_* \leq Q^{\kappa}$ then we are in case (2). Suppose, then, that $X^{\kappa} < q_* \leq X$ and $\sigma_* < \frac{\lambda_1\kappa }{2\log X}$.

Now consider the scale $Q = X^{\kappa}$. By analogous reasoning, we are done unless there is an exceptional zero $\rho'_*$ at this scale, with $X^{\kappa^2} < q'_* \leq X^{\kappa}$ and $\sigma'_* < \frac{\lambda_1\kappa}{2 \log X^{\kappa}}$.
 
Since $q'_* \leq X^{\kappa} < q_*$, the zeros $\rho_*, \rho'_*$ correspond to distinct real characters. By Landau's result we obtain
\[ \frac{\lambda_1\kappa}{2\log X^{\kappa}}\geq \max(\sigma_*, \sigma'_*) \geq \frac{\lambda_1}{\log q_* q'_*} > \frac{\lambda_1}{\log X^2},\] a contradiction.
\end{proof}

We quote the following two results from \cite{bombieri}. Here, $\lambda_2, \lambda_3$ are computable absolute constants. 

\begin{proposition}\label{prop74}
Let $Q \geq 10$. We have, uniformly in $\alpha \leq 1$,
\begin{equation}\label{unexceptional} \# \{ \rho \in \Xi_Q(\textstyle\frac{1}{2}\displaystyle) : \Re \rho \geq \alpha\} \leq \lambda_2Q^{\lambda_3(1 - \alpha)}.\end{equation} \end{proposition}
\begin{proof} This is the first statement in \cite[Th\'eor\`eme 14]{bombieri}, where it is noted that the argument leads to explicit constants if desired.  \end{proof}

\begin{proposition}\label{prop75}
Let $Q \geq 10$. If there is an exceptional zero \textup{(}$\rho_*$\textup{)} at scale $Q$ then we have the stronger bound 
\begin{equation}\label{exceptional}  \# \{ \rho \in \Xi_Q(\textstyle\frac{1}{2}\displaystyle) : \rho \neq \rho_* , \Re \rho \geq \alpha\} \leq \lambda_2 (\sigma_* \log Q ) Q^{\lambda_3(1 - \alpha)}. \end{equation}
\end{proposition}
\begin{proof}
This is the second statement in \cite[Th\'eor\`eme 14]{bombieri}. 
\end{proof}

\emph{Remark.} Note that Proposition \ref{prop74} is a grand log-free zero-density estimate and as such is the same as Proposition \ref{jut-est}, except that in the latter we gave an explicit exponent of $8$ in place of $\lambda_3$. In this section, the implied constant $\lambda_2$ is also important, and moreover we must pair Proposition \ref{prop74} with Proposition \ref{prop75}, which is usually called \emph{exceptional zero repulsion}. 

Explicit (large) values for $\lambda_2, \lambda_3$ are obtained in the work of Thorner and Zaman \cite{thorner-zaman}, and using these they were able to show that the exponent $c$ in Theorem \ref{main-sarkozy} can be taken to be $10^{-18}$.

Let $M$ be the explicit absolute constant specified in \eqref{m-def}, that is to say $M = 7 \cdot 3^6 \cdot 2^{22993}$.

\begin{proposition}\label{convergent-zeros}
There is an absolute constant $B$ such that the following holds for any parameter $Q \geq 10$.
\begin{itemize}
\item If there is no exceptional zero at scale $Q$, 
\begin{equation}\label{prop47.1} 4M \sum_{\rho \in \Xi_Q(\frac{1}{2})} Q^{-B\sigma} \leq 1;\end{equation}
\item If there is an exceptional zero $\rho_*$ at scale $Q$,
\begin{equation}\label{prop47.2} Q^{-B\sigma_*} + 2M \sum_{\rho \in \Xi_Q(\frac{1}{2}) \setminus \{ \rho_*\}} Q^{-B\sigma}  \leq 1.\end{equation}
\end{itemize}
\end{proposition}
\begin{proof}
Let $\lambda_4 \in (0,1)$ be a constant such that \begin{equation}\label{c4-choice} e^{-x} + \lambda_4 \lambda_2 x < 1\end{equation} for all $x$ with $0 < x < \lambda_1$ (a choice which works is $\lambda_4 := \min(\frac{1}{2\lambda_2}, \frac{1}{2\lambda_1\lambda_2})$).

\emph{Unexceptional case.}
Choose $k_0$ such that 
\[ \lambda_2 \sum_{k \geq k_0} e^{-2^{k-1}} \leq \frac{\lambda_4}{4M}.\]
Set 
\[ S_k := \{ \rho\in \Xi_Q(\textstyle\frac{1}{2}\displaystyle) : 2^k \leq B \sigma \log Q < 2^{k+1} \}.\] For all $\rho$, the fact that we are in the unexceptional case means that $\sigma \log Q \geq \lambda_1$, so with $B$ chosen suitably (that is $B \lambda_1 > 2^{k_0}$), $S_k$ is only nonempty for $k \geq k_0$. From \eqref{unexceptional} we have
\[ |S_k| \leq \lambda_2 Q^{\frac{\lambda_3 2^{k+1}}{B \log Q}} = \lambda_2 e^{\frac{\lambda_3 2^{k+1}}{B}}.\] Therefore
\begin{align*} \sum_{\rho \in \Xi_Q(\frac{1}{2})} Q^{-B \sigma} &  = \sum_{k \geq k_0} \sum_{\rho \in S_k} Q^{-B\sigma}   \leq \sum_{k \geq k_0} |S_k| e^{-2^k} \\ &  \leq \lambda_2 \sum_{k \geq k_0} e^{\frac{\lambda_3}{B}2^{k+1} - 2^k}  \leq \lambda_2 \sum_{k \geq k_0} e^{-\frac{2^k}{2}},\end{align*} which is at most $\frac{1}{4M}$ by the choice of $k_0$. Here, in the penultimate step, we also assumed that $B \geq 4\lambda_3$.

\emph{Exceptional case.} We proceed essentially as before, but now defining \[ S_k := \{\rho \neq \rho_* : 2^k \leq B\sigma \log Q < 2^{k+1}\}.\]
 Once again, if $B$ is chosen suitably, $S_k$ is only nonempty for $k \geq k_0$. From \eqref{exceptional}, we have
\[ |S_k| \leq \lambda_2 (\sigma_*  \log Q) Q^{\frac{\lambda_3 2^{k+1}}{B\log Q}} = \lambda_2 (\sigma_* \log Q) e^{\frac{\lambda_3}{B}  2^{k+1}}.\]  Much as before, 
\[ \sum_{\substack{ \rho \in \Xi_Q(\frac{1}{2}) \\ \rho \neq \rho_*}} Q^{-B\sigma}  \leq \lambda_2 (\sigma_* \log Q)  \sum_{k \geq k_0} e^{-\frac{2^k}{2}} \leq \frac{1}{4M}\lambda_2 \lambda_4 (\sigma_* \log Q).\] 
However, by the choice of $\lambda_4$, and with $x := \sigma_* \log Q \leq \lambda_1$ in \eqref{c4-choice},
\[ \lambda_2 \lambda_4 (\sigma_* \log Q)  < 1 - Q^{-\sigma_*} .\]
This confirms the desired inequality.\end{proof}

Let us finish by stating the consequence of all this that we will need in what follows. In the following, the exponents $\frac{1}{120}$ and $10^{-6}$ are what turn out to be convenient later; it would be possible to specify any positive values for them, at the expense of making $c$ smaller.

\begin{proposition}\label{sec4-takeaway} Let $M := 7 \cdot 3^6 \cdot 2^{22993}$. There is an absolute constant $c > 0$ with the following property. Suppose that $\sigmax \leq \frac{1}{8}$. Let $N$ be a large integer. Then there is $T$, $N^{c} \leq T \leq N^{\frac{1}{120}}$, such that if we list the zeros of $\Xi_T(\sigmax)$ \textup{(}with multiplicity\textup{)} in order of decreasing real part as $\rho_1,\dots, \rho_J$, \textup{(}at least\textup{)} one of the following holds, where $\sigma_j := 1 - \Re \rho_j$:
\begin{enumerate}
\item We have \begin{equation}\label{unexc-n-bd}  \sum_{j = 1}^J N^{-\frac{1}{16} \sigma_j} \leq \frac{1}{2M};\end{equation}
\item $\rho_1$ is a simple, real zero of a real character $\chi_1$ of conductor $q_1 \leq T^{10^{-6}}$, and
\begin{equation}\label{exc-n-bd} N^{-\frac{1}{16}\sigma_1} + 2M \sum_{j =2}^J N^{-\frac{1}{16}\sigma_j} \leq 1.\end{equation}
\end{enumerate}
Moreover, 
\begin{equation}\label{crude-J} J \ll T.\end{equation}
\end{proposition}
\begin{proof}
Set $b := \min(\frac{1}{16B}, \frac{\lambda_1 \kappa}{32 \log (4M)})$. Apply Proposition \ref{two-scales-prop} with $X := N^{b}$ and $\kappa := 10^{-6}$ and let $T = Q = X$ or $X^{\kappa}$ as in the conclusion of that proposition. Note that
\begin{equation}\label{t-bds} T \leq X = N^b \leq N^{1/16B}.\end{equation}

Then either (first case) we have 
\begin{equation}\label{first-case-lower-sig} \sigma_1 \geq \frac{\lambda_1 \kappa}{2 \log T} \geq \frac{\lambda_1 \kappa}{2 \log X}\end{equation} or (second case) there is an exceptional zero $\rho_1$ at scale $T$, with modulus $q_1 \leq T^{\kappa}$.

Now apply Proposition \ref{convergent-zeros} with $Q = T$. Suppose we are in the first case. If there is no exceptional zero at scale $T$ we apply \eqref{prop47.1}, which gives (using \eqref{t-bds}) 
\[ \sum_{j = 1}^J N^{-\frac{1}{16} \sigma_j} \leq \sum_{j = 1}^J T^{-B \sigma_j}  \leq \frac{1}{4M}.\]  If there is an exceptional zero at scale $T$ (which must be $\rho_1$) then we must instead apply \eqref{prop47.2}. Ignoring the term with the exceptional zero, as above this gives\ \begin{equation}\label{sum-j-2} \sum_{j = 2}^J N^{-\frac{1}{16}\sigma_j} \leq \frac{1}{4M}.\end{equation}
Moreover, by the lower bound \eqref{first-case-lower-sig} on $\sigma_1$ we have
\[ N^{-\frac{1}{16} \sigma_1} \leq N^{-\frac{\lambda_1 \kappa}{32 \log X}} = e^{-\frac{\lambda_1 \kappa }{32 b}} \leq \frac{1}{4M}\] by the choice of $b$.  Adding this to \eqref{sum-j-2} gives \eqref{unexc-n-bd} in this case as well.

Now suppose we are in the second case (of Proposition \ref{two-scales-prop}, thus $\rho_1$ is exceptional with modulus $q_1 \leq T^{\kappa}$). Applying \eqref{prop47.2} gives, using \eqref{t-bds} again,
\[ N^{-\frac{1}{16}\sigma_j} + 2M \sum_{j =2}^J N^{-\frac{1}{16}\sigma_j} \leq T^{-B\sigma_1} + 2M \sum_{j = 2}^J T^{-B\sigma_j}  \leq 1.\] This is \eqref{exc-n-bd}, so we have indeed shown that one of (1), (2) holds.

Finally, \eqref{crude-J} follows immediately from \eqref{unexceptional-jutila} and the assumption that $\sigmax \leq \frac{1}{8}$. 
\end{proof}
Finally, we will require the following well-known bound.
\begin{lemma}[Page's bound] \label{page-theorem} We have $\sigma_1 \gg q_1^{-1/2} (\log q_1)^{-2} \gg q_1^{-1}$.
\end{lemma}
\begin{proof} See \cite[Corollary 11.12]{mv}. Note that the implied constant here can be made explicit if desired; one has Siegel's much stronger bound $\sigma_1 \gg_{\eps} q_1^{-\eps}$ for any $\eps > 0$ (see \cite[Theorem 11.14]{mv}), but here the implied constant is not explicit when $\eps < \frac{1}{2}$.\end{proof}

\section{Periodic functions and Ramanujan series}\label{rat-fourier-sec}

In this section we set out some basic nomenclature for periodic functions $f : \Z \rightarrow \C$ and for certain special periodic functions which we shall call Ramanujan series (following, for example, \cite{murty}, and inspired by \cite{ramanujan}), and we will prove some basic lemmas that will be needed later.

\subsection{Basic nomenclature and definitions} 

\emph{Periodic functions. Rational Fourier series.} The first key observation is that any periodic function $f : \Z \rightarrow \C$ has a (unique) \emph{rational Fourier series}, that is to say we may write
\begin{equation}\label{rat-fourier-def} f(n) = \sum_{\lambda \in \Q/\Z} f^{\wedge}(\lambda) e(\lambda n)\end{equation} for certain $f^{\wedge}(\lambda) \in \C$, which we call the rational Fourier coefficients of $f$, and with this sum being supported on just finitely many $\lambda$. To see the existence of such a series, simply observe that
\begin{equation}\label{char-progression} 1_{n \equiv b \mdsub{r}} = \frac{1}{r} \sum_{x \in \Z/r\Z} e\big(-\frac{xb}{r}\big) e\big(\frac{xn}{r}\big).\end{equation}  Any periodic function may be written as a finite linear combination of basic functions like this.\vspace*{8pt}

\emph{Norms, support and averages.} If $f$ is a periodic function with rational Fourier coefficients $f^{\wedge}(\lambda)$ then we write
\[ \Vert f \Vert^{\wedge}_{\infty} := \max_{\lambda} |f^{\wedge}(\lambda)| \quad \mbox{and} \quad \Vert f \Vert^{\wedge}_1 := \sum_{\lambda} |f^{\wedge}(\lambda)|.\]
We also define the support $\support(f)$ to be the largest denominator of any $\lambda$ for which $f^{\wedge}(\lambda) \neq 0$ (when $\lambda$ is written in lowest terms).

If $f :\Z \rightarrow \C$ is a periodic function then we write $\Av f$ for the average of $f$ (which is well-defined, either as the average of $f$ over any period, or as $\lim_{X \rightarrow \infty} X^{-1} \sum_{x = 1}^X f(x)$).\vspace*{8pt}

We have the orthogonality relation that $\Av e(\lambda n) = 1_{\lambda = 0}$, and so from \eqref{rat-fourier-def} we see that $f^{\wedge}(\lambda) = \Av f(n) e(-\lambda n)$. In particular the rational Fourier expansion of a function is unique, and also
\begin{equation}\label{fourier-avg} \Vert f \Vert^{\wedge}_{\infty} \leq \Av |f|.\end{equation}

\emph{Positivity.} We will say that a periodic function $f$ is \emph{Fourier-positive} if $f^{\wedge}$ is real and non-negative. If $f,g: \Z \rightarrow \C$ are two periodic functions then we write $g \prec f$ to mean that $|g^{\wedge}(\lambda)| \leq f^{\wedge}(\lambda)$ for all $\lambda \in \Q/\Z$. In particular, $f$ must be Fourier-positive for this notation to apply. We will use this notation fairly extensively throughout the paper.

A simple example of a Fourier-positive function is, for any fixed $r$, the function $f(n) := 1_{r | n}$, which is Fourier-positive by \eqref{char-progression} (with $b = 0$). Note that for this function we have
\begin{equation}\label{simple-ell1-div}\Vert 1_{r| n}  \Vert^{\wedge}_{\infty} = \frac{1}{r}, \quad  \Vert 1_{r | n}  \Vert^{\wedge}_1 = 1, \quad \support(1_{r | n}) = r.\end{equation}

\emph{Ramanujan series.}  We say that a periodic function $f$ in the form \eqref{rat-fourier-def} is a \emph{Ramanujan series} if $f^{\wedge}(\lambda)$ is supported where $\denom(\lambda)$ is squarefree, and if it depends only on $\denom(\lambda)$. Equivalently, we have an expansion
\begin{equation}\label{ram-definition} f(n) = \sum_{q = 1}^{\infty} \alpha(q) \sum_{a \in (\Z/q\Z)^*} e\big(\frac{an}{q}\big) = \sum_{q = 1}^{\infty} \alpha(q) c_q(n),\end{equation} supported on finitely many squarefree $q$. 

\subsection{Simple properties}

Here are some simple properties of positivity which we will use several times. 

\begin{proposition}\label{simple-pos-properties}
We have the following statements concerning periodic functions from $\Z$ to $\C$.
\begin{enumerate}
\item If $F, G$ are Fourier-positive then so are $F + G$, $FG$ and $t F$ for $t \in \R_{\geq 0}$;
\item If $(F_i)_{i \in I}$ is a finite collection of Fourier-positive functions then any non-negative linear combination $\sum_{i \in I} t_i F_i$, $t_i \in \R_{\geq 0}$, is Fourier-positive;
\item If $F \prec G$ then $\overline{F} \prec \overline{G}$ and $tF \prec t G$ for $t \in \R_{\geq 0}$;
\item If $F'_1 \prec F_1$ and $F'_2 \prec F_2$ then $F'_1 + F'_2 \prec F_1 + F_2$ and $F'_1F'_2 \prec F_1 F_2$;
\item If $F' \prec F$ and if $G$ is Fourier-positive then $F' G \prec FG$. In particular, this holds if $G(n) = e(\lambda n)$.
\item If $F$ has some integer $m$ as a period, $F^{\wedge}$ is supported on rationals $\lambda$ with $\denom(\lambda) | m$;
\item If $F', F$ have $m$ as a period, then $F' \prec F$ if and only if
\[ \bigg|\sum_{n \in \Z/m\Z} F'(n) e\big(-\frac{a n}{m}\big)\bigg| \leq \sum_{n \in \Z/m\Z} F(n) e\big(-\frac{a n}{m}\big)\] for all $a \in \Z/m\Z$.
\end{enumerate}
\end{proposition}
\begin{proof}
(1) and (2) are straightforward. 
For the first statement in (3), note that $\overline{F}^{\wedge}(\lambda) = \overline{F^{\wedge}(-\lambda)}$, so
\[ |\overline{F}^{\wedge}(\lambda)| = | F^{\wedge}(-\lambda)| \leq G^{\wedge}(-\lambda) = \overline{G^{\wedge}(-\lambda)} = \overline{G}^{\wedge}(\lambda),\] where the last step holds since $G^{\wedge}$ is real-valued (since $F \prec G$). The second statement in (3) is clear. For (4), we use the formula \begin{equation}\label{prod-bs}(b_1b_2)^{\wedge}(\lambda) = \sum_{\substack{\lambda_1, \lambda_2 \in \Q/\Z \\ \lambda_1 + \lambda_2 = \lambda \mdsub{1}}} b_1^{\wedge}(\lambda_1)b_2^{\wedge}(\lambda_2),\end{equation} valid for any pair of periodic functions $b_1, b_2 : \Z \rightarrow \C$. Applying this first with $b_i = F'_i$ and then with $b_i = F_i$, we obtain
\begin{align*}
|(F'_1 F'_2)^{\wedge}(\lambda) | & \leq \sum_{\substack{\lambda_1, \lambda_2 \in \Q/\Z \\ \lambda_1 + \lambda_2 = \lambda \mdsub{1}}} |F_1^{\prime \wedge}(\lambda_1)||F_2^{\prime \wedge}(\lambda_2)| \\ & \leq \sum_{\substack{\lambda_1, \lambda_2 \in \Q/\Z \\ \lambda_1 + \lambda_2 = \lambda \mdsub{1}}} F_1^{\wedge}(\lambda_1)F_2^{\wedge}(\lambda_2) = (F_1F_2)^{\wedge}(\lambda),
\end{align*}
which is the multiplicative part of (4). The additive part of (4) is straightforward. For (5), it suffices in view of (3) and (4) and the fact that any Fourier-positive $G$ is a non-negative linear combination of functions of the form $e(\lambda n)$ to check the case $G(n) = e(\lambda n)$. To see this, observe that $|(F'G)^{\wedge}(\xi)| = |(F')^{\wedge}(\xi - \lambda)| \leq F(\xi - \lambda) = (FG)^{\wedge}(\xi)$, for all $\xi \in \Q/\Z$. Item (6) follows from \eqref{char-progression} and the uniqueness of rational Fourier series. Finally, for (7) we note that $\sum_{n \in \Z/m\Z} F(n) e(-\frac{an}{m}) = m F^{\wedge} (\frac{a}{m})$, so under the stated assumption we have $|(F')^{\wedge}(\frac{a}{m})| \leq F^{\wedge}(\frac{a}{m})$ for all $a \in \Z/m\Z$. By item (6), $|(F')^{\wedge}(\lambda)| \leq F^{\wedge}(\lambda)$ for all $\lambda \in \Q/\Z$. 
\end{proof}

Next, some simple properties concerning norms and the support.

\begin{lemma}\label{norm-shift}
Suppose that $f : \Z \rightarrow \C$ is a periodic function. Then the norms $\Vert f \Vert^{\wedge}_1$, $\Vert f \Vert^{\wedge}_{\infty}$ are invariant under shifts: in particular $\Vert f^+ \Vert^{\wedge}_{\bullet} = \Vert f \Vert^{\wedge}_{\bullet} = \Vert f^{-} \Vert^{\wedge}_{\bullet}$ for both norms, where $f^+(n) = f(n+1)$ and $f^{-}(n) = f(n - 1)$.
\end{lemma}
\begin{proof}
This follows immediately from $(f^\pm)^{\wedge}(\lambda) = e(\pm\lambda) f^{\wedge}(\lambda)$, which follows from the definition \eqref{rat-fourier-def}.
\end{proof}

\begin{lemma}\label{period-fourier}
Suppose that $b_1, b_2 : \Z \rightarrow \C$ are periodic functions. Then 
\begin{equation}\label{b1b2-ell1} \Vert b_1 b_2 \Vert^{\wedge}_{1} \leq \Vert b_1 \Vert^{\wedge}_1 \Vert b_2 \Vert^{\wedge}_1.\end{equation}
\begin{equation}\label{b1b2-noncoprime} \Vert b_1 b_2 \Vert^{\wedge}_{\infty} \leq \Vert b_1 \Vert^{\wedge}_{\infty} \Vert b_2 \Vert^{\wedge}_{1}\end{equation} If $b_1, b_2$ have coprime periods then we have the equality
\begin{equation}\label{b1b2-coprime} \Vert b_1 b_2 \Vert^{\wedge}_{\infty} = \Vert b_1 \Vert^{\wedge}_{\infty} \Vert b_2 \Vert^{\wedge}_{\infty}.\end{equation}
Finally, we have 
\begin{equation}\label{support-prod} \support (b_1 b_2) \leq \support(b_1) \support(b_2).\end{equation}
\end{lemma}
\begin{proof}
The bounds \eqref{b1b2-ell1}, \eqref{b1b2-noncoprime} are immediate from \eqref{prod-bs}. 
For \eqref{b1b2-coprime}, suppose that the periods of $b_1, b_2$ are $m_1, m_2$ respectively, and that $m_1, m_2$ are coprime. By Proposition \ref{simple-pos-properties} (6), $b_i^{\wedge}(\lambda_i)$ is supported where $\denom(\lambda_i)$ divides $m_i$. Now if $\frac{a_1}{m_1} + \frac{a_2}{m_2} \equiv \frac{a'_1}{m_1} + \frac{a'_2}{m_2} \md{1}$ then, clearing denominators and using the coprimality of $m_1, m_2$, we see that $a_1 \equiv a'_1 \md{m_1}$ and $a_2 \equiv a'_2 \md{m_2}$, so $\frac{a_1}{m_1} \equiv \frac{a'_1}{m_1} \md{1}$ and $\frac{a_2}{m_2} \equiv \frac{a'_2}{m_2} \md{1}$. It follows that, in \eqref{prod-bs}, all the sums $\lambda_1 + \lambda_2$ with $\lambda_i \in \Supp(b_i^{\wedge})$ are distinct, and so for each $\lambda$ there is at most one choice of $\lambda_1, \lambda_2$ with $b_1^{\wedge}(\lambda_1), b_2^{\wedge}(\lambda_2) \neq 0$, and so \eqref{b1b2-coprime} follows. Finally, \eqref{support-prod} is clear from \eqref{prod-bs}. \end{proof}

We remark on the case of Dirichlet characters, which will arise several times in the paper.

\begin{lemma}\label{gauss-periodic}
Suppose that $\chi$ is a primitive Dirichlet character to modulus $q$. Then $\Vert \chi \Vert^{\wedge}_{\infty} \leq q^{-\frac{1}{2}}$ and $\Vert \chi \Vert^{\wedge}_1 \leq q^{\frac{1}{2}}$.
\end{lemma}
\begin{proof} The second statement follows from the first, because $\chi$ is $q$-periodic and so its rational Fourier series is supported on rationals with denominator $q$. By Fourier analysis on $\Z/q\Z$ we have $\chi(n) = \sum_{a \in \Z/q\Z} \gamma(a) e(an/q)$, where $\gamma(a) := \sum_{x \in \Z/q\Z} \chi(x) e(-ax/q)$. By \cite[equation (3.12)]{ik} and the remarks following it, $\gamma(a) = \overline{\chi}(a) \tau(\chi)$, where $\tau(\chi)$ is the Gauss sum, which has modulus $\sqrt{q}$ by \cite[Lemma 3.1]{ik}. The first statement of the lemma follows.
\end{proof}

Let us also record basic bounds on the Fourier norms of the functions $F_{\chi, Q}$ and $\Lambda_Q$, which are crucial in the paper.

\begin{lemma}\label{f-chi-ell1-lem}
Let $\chi$ be a primitive character to modulus $q$. We have
\begin{equation}\label{f-chi-ell1} \Vert F_{\chi, Q} \Vert^{\wedge}_1 \leq Q q^{\frac{1}{2}}.\end{equation}
In particular, 
\begin{equation}\label{lam-q-ell1} \Vert \Lambda_Q \Vert^{\wedge}_1 \leq Q.\end{equation}
\end{lemma}
\begin{proof}
From the definition \eqref{c-def} and the evaluation \eqref{gauss-square} for Gauss sums, we have the bound $|c_{\chi}(b,r)| \leq q^{\frac{1}{2}} \phi(r)^{-1}$. Recall the definition \eqref{fourier-trunc} of $F_{\chi, Q}$. Summing over the $\phi(r)$ values of $b \in (\Z/r\Z)^*$, and then summing over all $r$ such that $q \mid r$ and $\frac{r}{q} \leq Q$, the result follows.

The second statement \eqref{lam-q-ell1} is immediate upon taking $\chi = \chi_0$, the prinicipal character, and recalling \eqref{hb-link}.
\end{proof}

\subsection{$\vp , \wp$ and their properties} The following two types of periodic function will play an important role in the paper.
For $p$ a prime, write (as in \eqref{vp-definition})
 \begin{equation}\label{vw-def} \vp (n) := \left\{ \begin{array}{ll} 0 & \mbox{if $p \mid n$} \\ \frac{p}{p-1} & \mbox{if $(p,n) = 1$} \end{array}\right. \end{equation}
 and
 \begin{equation}\label{tau-sq-def} \wp(n) := \left\{ \begin{array}{ll} \frac{4p}{p+3} & \mbox{if $p \mid n$} \\ \frac{p}{p+3} & \mbox{if $(p,n) = 1$}. \end{array}\right.  \end{equation}
 The definitions are chosen so $\Av \vp  = \Av \wp = 1$. 
 
The function $\Lambda_{\Z/p\Z}$, which is somewhat standard, was already briefly discussed in Section \ref{sec6}. The notation $\tau^2_{\Z/p\Z}$ is not standard; this is supposed to remind the reader that this function behaves somewhat like a $\md{p}$ analogue of the squared divisor function $\tau^2$, as discussed in subsection \ref{subsec3.1}, albeit normalised to have average value $1$.

It is easily checked that the rational Fourier series of $\vp , \wp $ are 
\begin{equation}\label{v-Fourier} \vp (n) = 1 -  \frac{1}{p-1} \sum_{r \in (\Z/p\Z)^*} e\big(\frac{rn}{p}\big) = 1 - \frac{1}{p-1} c_p(n)\end{equation}
and
\begin{equation}\label{w-Fourier} \wp (n) = 1 + \frac{3}{p+3} \sum_{r \in (\Z/p\Z)^*} e\big(\frac{rn}{p}\big) = 1 + \frac{3}{p+3} c_p(n),\end{equation} so in particular both are Ramanujan series. Moreover, $\wp $ is Fourier-positive, though $\vp $ is not. 

Let us also note the following for future use.

\begin{lemma}\label{fourier-vw} We have 
\begin{equation}\label{vp-four} \Vert \vp  \Vert^{\wedge}_{\infty} = 1,\qquad \Vert \vp  \Vert^{\wedge}_1 = 2\end{equation}
and
\begin{equation}\label{wp-four} \Vert \wp  \Vert^{\wedge}_{\infty} = 1, \qquad \Vert \wp  \Vert^{\wedge}_1 \leq 4.\end{equation}
\end{lemma}

The following simple lemma plays a critical role in our paper. Here, recall, we use the notation $f^+(n) = f(n+1)$, $f^-(n) = f(n - 1)$.

\begin{lemma}\label{lem43}
For every prime $p$, $\vp ^+ \vp^- \wp $ is Fourier-positive. Also, we have the bound
\begin{equation}\label{vpvpwp-bd} \Vert \vp ^+ \vp ^- \wp  \Vert^{\wedge}_{\infty} \leq 1 + O(p^{-2}).\end{equation}
\end{lemma}
\begin{proof} Set $u_p := (1 - \frac{1}{p})^2 (1 + \frac{3}{p}) \vp ^+ \vp ^- \wp $. Then one may check that 
 \[ u_p(n) = \left\{ \begin{array}{ll} 0 & \mbox{if $n \equiv \pm 1 \md{p}$} \\ 4 & \mbox{if $n \equiv 0 \md{p}$} \\ 1& \mbox{otherwise} .\end{array}\right.  \]
 If $p \neq 2$ we may compute the discrete Fourier transform $\widehat{u}_p(k) := p^{-1} \sum_{n \in \Z/p\Z} u_p(n)e(-kn/p)$ as follows:
 \begin{equation}\label{u-fourier} \widehat{u}_p(k) = \left\{ \begin{array}{ll} \frac{1}{p}(3 - 2 \cos \frac{2\pi k}{p}) & k \neq 0 \\ 1 + \frac{1}{p}  & k = 0.\end{array}  \right.\end{equation}
Note that $\widehat{u}_p(k)$ is real and positive, indeed $\widehat{u}_p(0) =1 + \frac{1}{p}$ and $\widehat{u}_p(k) \geq \frac{1}{p}$.

By Fourier inversion in $\Z/p\Z$, 
 \[ u_p(n) = 1 + \frac{1}{p} + \frac{1}{p}\sum_{k \neq 0} \big(3 - 2 \cos \frac{2\pi k}{p} \big) e\big(\frac{kn}{p}\big),\] and so $u_p$ is Fourier-positive.
 When $p=2$ one may check that $\widehat{u}_2(0) = \widehat{u}_2(1) = 2$ and so these inequalities also hold. 
 
 Multiplying through by $(1 - \frac{1}{p})^{-2}(1 + \frac{3}{p})^{-1}$, we see that 
 \[ \Vert  \vp ^+ \vp ^- \wp  \Vert^{\wedge}_{\infty}  \leq \big(1 - \frac{1}{p}\big)^{-2}\big(1 + \frac{3}{p}\big)^{-1}\max\big(1 + \frac{1}{p}, \frac{5}{p}\big) = 1 + O(p^{-2}),\]
 which is \eqref{vpvpwp-bd}. 
 \end{proof}

Products of the $\vp $ and of the $\wp $ will be important in the paper. In fact, we have already seen $\prod_{p \leq R} \vp (n)$; it is the same thing as $\tilde\Lambda_R(n)$, the normalised characteristic function of the $R$-almost primes, as defined in \eqref{tilde-lam-def}. Let us record this observation here explicitly:

\begin{equation}\label{eq818} \tilde\Lambda_R(n) = \prod_{p \leq R} \vp (n).\end{equation}

The product of the $\wp $ is not something we have seen yet, but it and its truncation are equally important in the paper. 

\begin{definition} \label{h-tilde-def} Let $R$ be a parameter. Then we define
\[ \tilde H_R(n) := \prod_{p \leq R} \wp (n),\] where $\wp $ is defined as in \eqref{tau-sq-def}.
\end{definition}

By \eqref{w-Fourier} and the multiplicativity of Ramanujan sums, $\tilde H_R(n)$ is a Ramanujan series

\[   \tilde H_R(n) = \sum_{q | R!} \eta(q) \sum_{a \in (\Z/q\Z)^*} e\big(\frac{an}{q}\big).\]
Here, for $q \mid R!$ we have 

\[ \eta(q) := \left\{ \begin{array}{ll} \prod_{p | q} \frac{3}{p+3}  & \mbox{if $q$ is squarefree} \\ 0 & \mbox{otherwise}.\end{array} \right. \]

We also introduce the truncated version of $\tilde H$.

\begin{definition}\label{hq-definition}  Let $Q$ be a parameter. Then we define
\[ H_Q(n) :=  \sum_{q \leq Q} \eta(q) \sum_{a \in (\Z/q\Z)^*} e\big(\frac{an}{q}\big).\]
\end{definition}
This is an important function in the paper, as described in Subsection \ref{subsec3.1}. We have $\eta(q) \leq \frac{3^{\omega(q)}}{q} \leq q^{-1 + o(1)}$, where $\omega(q)$ is the number of distinct prime factors of $q$. Therefore we have the bound
\begin{equation}\label{hq-ell1-crude} \Vert H^{\wedge}_Q \Vert_1 \leq Q^{1 + o(1)}.\end{equation}

\subsection{A class of periodic functions}

It turns out that the functions $\Lambda_Q$ and $H_Q$ belong to a certain class of periodic functions on $\Z$ which has nice closure properties under shifts and pointwise multiplication. This fact will be important later in the paper.

\begin{definition}\label{cbx-def}
Let $X \geq 1$ be an integer and $B \geq 0$ a real parameter. Then we denote by $\mathscr{C}_B(X)$ the class of periodic $f : \Z \rightarrow \C$ such that $f^{\wedge}(\lambda)$ is supported on $\lambda$ with $\denom(\lambda)$ squarefree and at most $X$, and such that if $q = \denom(\lambda)$ then $|f^{\wedge}(\lambda)|  \leq \tau(q)^B/q$.
\end{definition}

\begin{lemma}\label{h-lam-nice-class}
We have $\Lambda_Q, H_Q \in \mathscr{C}_2(Q)$.
\end{lemma}
\begin{proof}
For $\Lambda_Q$ (defined in \eqref{lambda-q-def}), this is just a case of checking that $|\frac{\mu(q)}{\phi(q)}| \leq \frac{\tau(q)^2}{q}$, which is true by some distance since $\frac{1}{\phi(q)} = \prod_{p | q} \frac{1}{p-1} \leq \prod_{p | q} \frac{2}{p} = \frac{\tau(q)}{q}$.

For $H_Q$, we must show that $\eta(q) \leq \tau(q)^2/q$. Noting that when $q$ is squarefree we have $\prod_{p | q} 3 = \tau(q)^{\frac{\log 3}{\log 2}} < \tau(q)^2$ gives this bound immediately.
\end{proof}

It is clear from the definition that $\mathscr{C}_B(X)$ is closed under shifts, that is to say if $f(n) \in \mathscr{C}_B(X)$ then $f(n + h) \in \mathscr{C}_B(X)$. Slightly less obvious is the fact that these classes are (weakly) closed under taking products, which is the content of Lemma \ref{sum-closure} below. Here, for $\lambda \in \R_{> 0}$ we write $\lambda \mathscr{C}_B(X)$ for those series of the form $\lambda f$, $f \in \mathscr{C}_B(X)$.

\begin{lemma}\label{sum-closure}
Suppose that $f_1 \in \mathscr{C}_{B_1}(X_1)$ and $f_2 \in \mathscr{C}_{B_2}(X_2)$. Then $f_1 f_2 \in (\log^{O_{B_1, B_2}(1)} X_1)\mathscr{C}_{B_1 + 2B_2 + 3}(X_1X_2)$.
\end{lemma}

To prove this result we will need a counting result, Lemma \ref{denom-lem} below, about denominators of sums of fractions (when written in lowest terms). 

Suppose that $q, r$ are squarefree and that $(a,q) = (b, r) = 1$, and set
\[ d := \denom\big(\frac{a}{q} + \frac{b}{r}\big) = \frac{qr}{(ar + bq, qr)}.\]
Then $\frac{qr}{(q,r)^2} \mid d$ and $d  \mid [q,r]$ (and in particular $d$ is squarefree).

\begin{lemma}\label{denom-lem}
Suppose that $q, r$ are squarefree and that $b \in (\Z/r\Z)^*$. Then
\[ \# \{ a \in (\Z/q\Z)^* : \denom\big(\frac{a}{q} + \frac{b}{r}\big) = d\} \leq \frac{d(q,r)}{r}.\]
\end{lemma}
\begin{proof} If $d \nmid [q,r]$ then the left-hand side is empty and the result is trivial. Suppose, then, that $d \mid [q,r]$. The denominator condition is equivalent to $(ar + bq, qr) = qr/d$, and so is contained in the condition $ar + bq \equiv 0 \md {qr/d}$. Dividing through by $(q,r)$ and setting $r' := r/(q,r)$, $q' := q/(q,r)$ (thus $q', r'$ are coprime) and $t := qr/d(q,r) = [q,r]/d$, this condition is equivalent to \begin{equation}\label{equiv-mod} ar' + bq' \equiv 0 \md{t}.\end{equation} If this condition is satisfied, $t$ must be coprime to $r'$. Indeed if $p \mid (t, r')$ then $p \mid bq' = (ar' + bq') - ar'$. But this is impossible since both $b$ and $q'$ are coprime to $r'$. Moreover, since $t \mid [q,r] = qr'$, it follows that $t \mid q$. 

Returning to \eqref{equiv-mod}, as $a$ ranges over $\Z/q\Z$ then so does $ar'$, and so we simply need an upper bound for the number of solutions $\md{q}$ to a fixed congruence $x \equiv c \md{t}$, where $t \mid q$. Such an upper bound is $q/t$, and the lemma follows.
\end{proof}

\begin{proof}[Proof of Lemma \ref{sum-closure}]
Suppose that $\lambda = \frac{a}{q} \in \Q/\Z$ is in lowest terms. Then 
\begin{align*} & (f_1 f_2)^{\wedge}\big(\frac{a}{q}\big) = \sum_{\substack{q_1 \leq X_1 \\ q_2 \leq X_2 }} \sum_{\substack{a_1 \in (\Z/q_1 \Z)^* \\ a_2 \in (\Z/q_2\Z)^* \\ \frac{a_1}{q_1} + \frac{a_2}{q_2} \equiv \frac{a}{q} \mdsub{1}}}f_1^{\wedge}\big(\frac{a_1}{q_1}\big) f_2^{\wedge}\big(\frac{a_2}{q_2}\big) \\ & \leq \sum_{\substack{q_1 \leq X_1 \\ \mu^2(q_1) = 1} } \frac{\tau(q_1)^{B_1}}{q_1} \sum_{d | q_1 q} \frac{\tau(d)^{B_2}}{d} \# \{ a_1 \in (\Z/q_1\Z)^* : \denom\big(\frac{a_1}{q_1} - \frac{a}{q}\big) = d\}.\end{align*}
By Lemma \ref{denom-lem},  this is bounded above by
\[ \sum_{\substack{q_1 \leq X_1 \\ \mu^2(q_1) = 1} }  \frac{\tau(q_1)^{B_1}}{q_1} \sum_{d | q_1q} \tau(d)^{B_2} \frac{(q_1, q)}{q} \leq \frac{\tau(q)^{B_2+1}}{q}\!\!\!\sum_{\substack{q_1 \leq X_1 \\ \mu^2(q_1) = 1} }  \!\!\frac{\tau(q_1)^{B_1 + B_2 +1}(q_1, q)}{q_1} .\] Here, we used the rather crude bound
\[ \sum_{d | q_1 q} \tau(d)^{B_2} \leq \tau(q_1 q) \max_{d | q_1 q} \tau(d)^{B_2} \leq \tau(q_1 q)^{B_2 + 1} \leq \tau(q_1)^{B_2 + 1} \tau(q)^{B_2 + 1}.\]
The result now follows from Lemma \ref{to-rankin}.
\end{proof}

\subsection{The functions $\Phi_{r,b}$}\label{subsec85}

We now turn to a different example of a Fourier-positive function, which will feature heavily later on since they are important in constructing the damping function $D$.

Let $r$ be any positive integer, and let $b \in \Z/r^3 \Z$. Then we define
\begin{equation}\label{phi-a-q-def} \Phi_{r,b}(x) := r^{\frac{5}{6}} 1_{r | x} e\big(\frac{bx}{r^3}\big).\end{equation}
\emph{Remark.} Here, the constants $\frac{5}{6}$ and $3$ are convenient choices which work for our later arguments. Any $0 < c < 1 < C$ with $c$ sufficiently close to $1$ and $C$ sufficiently large would work, though the exponent in the main theorem would be affected.

The functions $\Phi_{r,b}$ are of course periodic. They are also Fourier-positive, since both $1_{r | x}$ and $e(\frac{bx}{r^3})$ are. Indeed, 

\begin{equation}\label{phi-r-b-exp} \Phi_{r,b}(n) = r^{-\frac{1}{6}} \sum_{t =0}^{r-1} e\big(\big(\frac{t}{r} + \frac{b}{r^3}\big) n\big).\end{equation} 
Note in particular that 
\begin{equation}\label{phi-ell1-bd} \Vert \Phi_{r,b} \Vert^{\wedge}_1 = r^{\frac{5}{6}}\end{equation}
and
\begin{equation}\label{phi-support} \support(\Phi_{r,b}) \leq r^3.\end{equation}

\section{Triple correlations of Ramanujan series}\label{comparison-sec}

We turn now to a key technical result concerning the relationship between certain products of shifted Ramanujan series (as defined in \eqref{ram-definition}) and their truncations. Here is the main result.

\begin{proposition}\label{ram-truncate} Let $h_1, h_2, h_3$ be distinct integers with $\max_i |h_i| \leq H$. 
Consider three Ramanujan series  \[ f_i(n) = \sum_{q} \alpha_i(q) c_q(n),\] $i = 1,2,3$, such that $\alpha_i(q)$ is supported on finitely many squarefrees and satisfies the bound $|\alpha_i(q)| \leq \tau^B(q)/q$, where $\tau$ is the divisor function and $B$ is some parameter.  Consider also their truncations
\[ f_{i, X_i}(n) := \sum_{q \leq X_i} \alpha_i(q) c_q(n),\] where $X_1, X_2, X_3$ are some thresholds. Set  $X := \min(X_1, X_2, X_3)$. 
Consider the shifted products
\[ F(n) := f_1(n + h_1) f_2(n+h_2) f_3(n + h_3) \] and the truncated variants \[ F_{X_1, X_2, X_3}(n) := f_{1,X_1}(n+h_1) f_{2, X_2}(n+h_2) f_{3, X_3}(n+h_3).\] 
Then
\begin{equation}\label{ram-bd-key} \Vert F - F_{X_1, X_2, X_3}\Vert^{\wedge}_{\infty} \ll_{ B, \eps } H^3 X^{\eps - 1},\end{equation} for every $\eps > 0$.
\end{proposition}
\emph{Remark.} Most likely, this proposition could be extended to products of an arbitrary number of $f_i$ (rather than just 3) but we do not need any such extension in this paper.\vspace*{8pt}

We need the following lemma in the proof.

\begin{lemma}\label{key-lemma-r1}
Let $m_1, m_2$ be distinct, nonzero integers. Let $r$ be a positive integer and suppose that $b \in (\Z/r\Z)^*$. Let $q_1, q_2$ be squarefree, and suppose that $P$ is an integer with $P \mid [q_1, q_2]r$. Then
\[ \sum_{\substack{a_1 \in (\Z/q_1\Z)^* \\ a_2 \in (\Z/q_2\Z)^* \\ a_1 q_2r + a_2 q_1r \equiv bq_1q_2 \mdsublem{P (q_1, q_2)}}} e\big(\frac{m_1a_1}{q_1} + \frac{m_2 a_2}{q_2}\big) \leq |m_1 m_2 (m_1 - m_2)| \tau(P)^2.\]
\end{lemma}
\begin{proof}
Set $q'_1 = q_1/(q_1, q_2)$, $q'_2 = q_2/(q_1, q_2)$, $d := P/(P,r)$ and $r' := r/(P,r)$. The congruence $a_1 q_2r + a_2 q_1r \equiv bq_1q_2 \md{P (q_1, q_2)}$ is equivalent to $a_1 q'_2 r + a_2 q'_1 r \equiv b[q_1, q_2] \md{P}$. Moreover, since $(P,r)$ divides both $P$ and $r$, if the congruence has a solution at all then $b[q_1,q_2] = v(P,r)$ for some $v$, and so the congruence is equivalent to $a_1 q'_2 r' + a_2 q'_1 r' \equiv v \md{d}$. Note that $r'$ is coprime to $d$, so the congruence is equivalent to $a_1 q'_2 + a_2 q'_1 \equiv u \md{d}$, where $u := v(r')^{-1} \md{d}$. Thus our task is to show that
\begin{equation} \label{new-task}  S := \!\!\!\sum_{\substack{a_1 \in (\Z/q_1\Z)^* \\ a_2 \in (\Z/q_2\Z)^* \\ a_1 q'_2 + a_2 q'_1 \equiv u \mdsublem{d}}} e\big(\frac{m_1a_1}{q_1} + \frac{m_2 a_2}{q_2}\big) \leq |m_1 m_2 (m_1 - m_2)| \tau(P)^2.\end{equation}
We claim that $d \mid [q_1, q_2]$. To see this, recall that $P \mid [q_1, q_2] r$. Therefore $d = P/(P, r)$ divides $[q_1, q_2] r/(P, r) = [q_1, q_2] r'$. However, $d$ and $r'$ are coprime, so indeed $d \mid [q_1, q_2]$. 

Detecting the congruence $a_1 q'_2 + a_2 q'_1 \equiv u \md{d}$ using additive characters mod $d$, the sum to be estimated is then
\begin{equation}\label{s-sum} S = \frac{1}{d} \!\sum_{\lambda \mdsub{d}} \! e\big(-\!\frac{\lambda u}{d}\big) \! \! \!  \sum_{\substack{a_1 \in (\Z/q_1\Z)^* \\ a_2 \in (\Z/q_2\Z)^* }} e\big(\frac{m_1 a_1}{q_1} + \frac{m_2 a_2}{q_2} + \frac{\lambda(a_1 q'_2 + a_2 q'_1)}{d} \big).\end{equation}
Set $t := [q_1,q_2]/d$ (so $t$ is an integer, by the above discussion) and observe that 
\[ \frac{q'_2}{d} = \frac{q_1 q'_2}{d q_1} = \frac{[q_1,q_2]/d}{q_1} = \frac{t}{q_1},\] and similarly $q'_1/d = t/q_2$. Using these expressions, the sum \eqref{s-sum} can be rewritten as
\begin{equation}\label{s-sum-2} S = \frac{1}{d}\! \sum_{\lambda \mdsub{d}} \! e\big(-\!\frac{\lambda u}{d}\big) \! \! \! \sum_{a_1 \in (\Z/q_1\Z)^*} \! \! \! \! \! e\big( \frac{a_1 (m_1 + \lambda t)}{q_1}\big) \! \! \! \! \! \sum_{a_2 \in (\Z/q_2\Z)^*} \! \! \! \! \! e\big(\frac{a_2 (m_2 + \lambda t)}{q_2}\big).\end{equation}
Now we use the inequality
\[ \big|\sum_{a \in (\Z/q\Z)^*} e\big(\frac{an}{q}\big) \big| \leq (q,n),\] which is a standard bound for Ramanujan sums (see \cite[equation (3.5)]{ik} or \cite[Section 16.6]{hardy-wright}). Substituting into \eqref{s-sum-2} gives
\begin{equation}\label{s-sum-3} |S| \leq  \frac{1}{d} \sum_{\lambda \mdsub{d}} (q_1, m_1 + \lambda t)(q_2, m_2 + \lambda t).\end{equation}
Suppose that $p$ is prime and that $p \mid (q_1, m_1 + \lambda t)$. Then certainly $p \mid q_1$. If $p \mid t$, then we must have $p \mid m_1$. It follows that either $p \mid \frac{q_1}{(q_1, t)}$, or else $p \mid m_1$, and so (since $q_1$ is squarefree)
\begin{equation}\label{temp5} (q_1, m_1 + \lambda t) \leq |m_1| \big(\frac{q_1}{(q_1, t)}, m_1 + \lambda t\big).\end{equation}
Note moreover  that since $t \mid [q_1, q_2] = q'_2 q_1$ we have $t \mid q'_2 (q_1, t)$, and hence $\frac{q_1}{(q_1, t)} = \frac{[q_1, q_2]}{q'_2(q_1, t)}$ divides $\frac{[q_1, q_2]}{t} = d$. Therefore \eqref{temp5} implies
\[ (q_1, m_1 + \lambda t) \leq |m_1| (d, m_1 + \lambda t).\]
Similarly, 
\[ (q_2, m_2 + \lambda t) \leq |m_2| (d, m_2 + \lambda t).\]
Substituting into \eqref{s-sum-3} gives 
\begin{equation}\label{s-sum-4} |S| \leq \frac{|m_1 m_2|}{d} \sum_{\lambda \mdsub{d}} (d, m_1 + \lambda t)(d, m_2 + \lambda t).\end{equation}
Now since $t = [q_1,q_2]/d$ and $[q_1, q_2]$ is squarefree, $t$ is coprime to $d$. Therefore it follows from \eqref{s-sum-4} that
\begin{equation}\label{s-sum-5} |S| \leq \frac{|m_1 m_2|}{d}  \sum_{s \mdsub{d}} (d, s)(d, s + h),\end{equation} where $h := m_2 - m_1 \neq 0$. Without loss of generality, $h > 0$.

Now for any $s$, $(d,s)$ and $(d, s+h)$ are both divisors of $d$, and any common factor of them must divide $h$. If $(d,s) = d_1$ and $(d,s + h) = d_2$ then $s \equiv 0 \md{d_1}$ and $s \equiv -h \md{d_2}$, which is a congruence condition (possibly not satisfiable) mod $[d_1,d_2]$. Note that $[d_1, d_2]  = d_1 d_2/(d_1, d_2) \geq d_1 d_2/h$. Therefore, for a fixed choice of $d_1, d_2$, the number of such $s$ is at most $d/[d_1, d_2] \leq hd/d_1 d_2$. It follows that 
\begin{align*} \\   \sum_{s \mdsub{d}} (d,s)(d,  & s+h) =  \\ & = \sum_{d_1, d_2 \mid d} d_1 d_2 \# \{ s \in \Z/d\Z : (d,s) = d_1, (d, s+ h) = d_2\} \\ & \leq \sum_{d_1, d_2 \mid d} d_1 d_2 \frac{hd}{d_1 d_2} = h d \tau(d)^2.\end{align*}
Substituting into \eqref{s-sum-5} gives $|S| \leq h |m_1 m_2| \tau(d)^2$, which immediately implies the lemma.\end{proof}

\begin{corollary}\label{cor11}
 Let $m_1, m_2$ be distinct nonzero integers. Let $r$ be a positive integer and suppose that $b \in (\Z/r\Z)^*$. Suppose that $q_1, q_2$ are squarefree. Then, uniformly in $q_3$,
\[ \sum_{\substack{a_1 \in (\Z/q_1\Z)^* \\ a_2 \in (\Z/q_2\Z)^* \\ \denom(\frac{a_1}{q_1} + \frac{a_2}{q_2} - \frac{b}{r}) = q_3}} e\big(\frac{m_1a_1}{q_1} + \frac{m_2 a_2}{q_2}\big) \leq |m_1 m_2 (m_1 - m_2)| \tau(r[ q_1, q_2])^3 .\]
\end{corollary}
\begin{proof}
The condition $\denom(\frac{a_1}{q_1} + \frac{a_2}{q_2} - \frac{b}{r}) = q_3$ is equivalent to $(a_1 q_2r + a_2 q_1r - bq_1q_2, q_1 q_2 r) = \frac{rq_1 q_2}{q_3}$ (note it can only be satisfied if $q_3 \mid r q_1 q_2$, which implies $q_3 \mid r [q_1, q_2]$ since $q_3$ is squarefree).  

We may detect this condition by M\"obius inversion, specifically the version in Lemma \ref{enhanced-mobius}. Taking $m = a_1 q_2r + a_2 q_1 r - b q_1 q_2$, $n = rq_1 q_2$ and $\Delta = rq_1 q_2/q_3$ in that lemma gives
\begin{align*}  & 1_{\denom(\frac{a_1}{q_1} + \frac{a_2}{q_2} - \frac{b}{r}) = q_3  }  = 1_{(a_1 q_2 r+ a_2 q_1r - bq_1 q_2, r q_1, q_2) = \frac{rq_1q_2}{q_3} } \\ &  \qquad = \sum_{\frac{rq_1q_2}{q_3} \mid d \mid rq_1q_2} \mu\big(\frac{d}{rq_1q_2/q_3}\big) 1_{a_1 q_2 r + a_2 q_1r \equiv b q_1 q_2 \mdsub{ d}} \\ &  \qquad = \sum_{\frac{r[q_1, q_2]}{q_3} \mid P \mid r[q_1, q_2]} \mu\big(\frac{P}{r[q_1, q_2]/q_3}\big) 1_{a_1 q_2 r + a_2 q_1r \equiv b q_1 q_2 \mdsub{ P(q_1, q_2)}},\end{align*} where $P := d/(q_1, q_2)$ (which is an integer).
Multiplying this by $e(\frac{m_1a_1}{q_1} + \frac{m_2 a_2}{q_2})$ and applying Lemma \ref{key-lemma-r1} for each value of $P$ (of which there are at most $\tau([q_1, q_2] r)$) gives the result. 
\end{proof}

\begin{lemma} \label{lem12} Let $B \geq 0$ and $X \geq 1$ be parameters and let $r \geq 1$ be an integer. Then
\begin{equation}\label{star-posts} \sum_{\substack{\max(q_1, q_2, q_3) \geq X \\ \mu^2(q_1) = \mu^2(q_2) = \mu^2(q_3) = 1 \\ q_1 \mid rq_2q_3 \\ q_2 \mid rq_1 q_3 \\ q_3 \mid rq_1 q_2}} \frac{\tau(q_1)^B}{q_1} \frac{\tau(q_2)^B}{q_2}\frac{\tau(q_3)^B}{q_3} \ll_{B, \eps}   r^{\eps} X^{\eps - 1}.\end{equation}
\end{lemma}
\begin{proof}
By symmetry it suffices to handle the sum over those $q_1, q_2, q_3$ for which $q_1$ is biggest (and hence $q_1 \geq X$). The condition $q_1 \mid rq_2 q_3$ implies that $\frac{q_1}{(q_1, rq_2)}  \mid q_3$, and so in particular $q_3 \geq \frac{q_1}{(q_1, rq_2)}$.

Given $q_1, q_2$, since $q_3 \mid rq_1 q_2$ it follows that there are at most $\tau(rq_1q_2) \leq \tau(r) \tau(q_1) \tau(q_2)$ choices for $q_3$. It also follows that, for each permissible $q_3$, $\tau(q_3) \leq \tau(r) \tau(q_1) \tau(q_2)$. It follows that the LHS of \eqref{star-posts} is bounded above by 
\begin{align*}  \tau(r)^{B+1}\sum_{\substack{q_1 \geq X \\ q_2 \leq q_1 \\ \mu^2(q_1) = \mu^2(q_2) = 1}} & \frac{\tau(q_1)^{2B+1}\tau(q_2)^{2B+1}}{q_1 q_2 \frac{q_1}{(q_1, rq_2)}} \\ & = \sum_{\substack{q_1 \geq X \\ \mu^2(q_1) = 1}} \frac{\tau(q_1)^{2B+1}}{q_1^2} \sum_{\substack{q_2 \leq q_1 \\ \mu^2(q_2) = 1}} \frac{\tau(q_2)^{2B+1} (q_1, rq_2)}{q_2}.\end{align*}
By Lemma \ref{to-rankin} (noting that $(q_1, rq_2) = (q_1, r)(\frac{q_1}{(q_1, r)}, q_2)$, and taking $h = \frac{q_1}{(q_1, r)}$ in that lemma), the inner sum is $\ll  (q_1, r)\tau(q_1)^{2B+2} \log^{O_B(1)} q_1$.
Therefore the LHS of \eqref{star-posts} is bounded above by
\[ \tau(r)^{B+1} \sum_{\substack{q_1 \geq X \\ \mu^2(q_1) = 1}} \frac{(q_1, r)\tau(q_1)^{4B+3} \log^{O_B(1)} q_1 }{q_1^2} \] 
By Lemma \ref{to-rankin-2} (and a further application of the divisor bound $\tau(r) \ll_{\eps} r^{\eps}$) this is indeed bounded above by $\ll_{B,\eps} r^{\eps} X^{\eps - 1}$, which is what we wanted to prove.
\end{proof}

\begin{proof}[Proof of Proposition \ref{ram-truncate}] From the definitions of $F$ and $F_{X_1, X_2, X_3}$ in the statement of the proposition one may check that 
\begin{align}\nonumber (F  - F_{X_1, X_2, X_3})^{\wedge}& (\lambda) =  \sum_{\max(\frac{q_1}{X_1}, \frac{q_2}{X_2}, \frac{q_3}{X_3}) > 1} \alpha(q_1) \alpha(q_2) \alpha(q_3) \times \\ & \times \!\!\!\!\!\! \sum_{\substack{a_i \in (\Z/q_i\Z)^* \\ \frac{a_1}{q_1} + \frac{a_2}{q_2} + \frac{a_3}{q_3} \equiv \lambda \mdsub{1}}}  e\big(\frac{h_1 a_1}{q_1}  + \frac{h_2 a_2}{q_2} + \frac{h_3 a_3}{q_3}\big) .\label{f-minus-trunc} \end{align}
Using this, we will show first that 
\begin{equation}\label{bound-claimed-rb}  |(F  - F_{X_1, X_2, X_3})^{\wedge}(\lambda)| \ll_{B,\eps} H^3 \denom(\lambda)^{\eps} X^{\eps - 1}.\end{equation} This statement is substantial progress towards Proposition \ref{ram-truncate}. To handle the case where $\denom(\lambda)$ is large, however, we need an additional argument using \eqref{bound-claimed-rb} and Lemma \ref{sum-closure}.

We turn to the proof of \eqref{bound-claimed-rb}. Suppose that $\lambda = \frac{b}{r}$ with $b, r$ coprime. If $\frac{a_1}{q_1} + \frac{a_2}{q_2} + \frac{a_3}{q_3} \equiv \lambda \md{1}$ then $q_1 \mid r q_2 q_3$, $q_2 \mid r q_1 q_3$ and $q_3 \mid r q_1 q_2$.  Therefore the RHS of \eqref{f-minus-trunc} may be rewritten as $e(h_3 b/r)$ times
\[ \sum_{\substack {\max(\frac{q_1}{X_1}, \frac{q_2}{X_2}, \frac{q_3}{X_3}) > 1\\ q_1 | rq_2q_3 \\ q_2 | rq_1 q_3 \\ q_3 | rq_1 q_2}}\!\!\!\alpha(q_1) \alpha(q_2) \alpha(q_3) \!\!\!\!\!\!\!\!\!\! \sum_{\substack{a_1 \in (\Z/q_1\Z)^* \\ a_2 \in (\Z/q_2\Z)^*  \\ \denom(\frac{a_1}{q_1} + \frac{a_2}{q_2} - \frac{b}{r}) = q_3}}
\!\!\!\!\!\!\!\! e\big(\frac{(h_1- h_3) a_1}{q_1} + \frac{(h_2 - h_3) a_2}{q_2} \big).\]
Note that the summation over $a_3$ has disappeared since $a_3$ is uniquely determined by $a_1, q_1, a_2, q_2$, and the condition $\denom(\frac{a_1}{q_1} + \frac{a_2}{q_2} - \frac{b}{r}) = q_3$ guarantees that this $a_3$ is coprime to $q_3$.
Since $h_1, h_2, h_3$ are distinct, the inner sum is exactly of the type considered in Corollary \ref{cor11}, and by that result it is bounded by $O(H^3 \tau(r [q_1, q_2])^3)$. It follows that 
\[ |(F  - F_{X_1, X_2, X_3})^{\wedge}(\lambda)| \ll H^3\tau(r)^3 \!\!\!\!\!\!\!\!\!\!\sum_{\substack {\max(\frac{q_1}{X_1}, \frac{q_2}{X_2}, \frac{q_3}{X_3}) > 1\\ q_1 | rq_2q_3 \\ q_2 | rq_1 q_3 \\ q_3 | rq_1 q_2}} \!\!\! |\alpha(q_1) \alpha(q_2) \alpha(q_3)| \tau([q_1, q_2])^3.\]
By assumption, $|\alpha(q)| \leq \tau(q)^B/q$ for all $q$, and $\alpha$ is supported on squarefrees, and so this is bounded above by
\[  O(H^3\tau(r)^3)\sum_{\substack {\max(q_1, q_2, q_3) \geq X \\ \\ \mu^2(q_1) = \mu^2(q_2) = \mu^2(q_3) = 1 \\ q_1 | rq_2q_3 \\ q_2 | rq_1 q_3 \\ q_3 | rq_1 q_2}} \frac{\tau(q_1)^{B+3}}{q_1} \frac{\tau(q_2)^{B+3}}{q_2} \frac{\tau(q_3)^{B+3}}{q_3} .\] Applying Lemma \ref{lem12} and the divisor bound gives the claimed bound \eqref{bound-claimed-rb}.

Now we supply the further arguments necessary to handle the case where $\denom(\lambda)$ is large. Fix some $\lambda$ and write $r := \denom(\lambda)$.  Then \eqref{bound-claimed-rb} gives
\begin{equation}\label{911-a} |(F - F_{r,r,r})^{\wedge}(\lambda)| \ll_{B, \eps} H^3 r^{\eps - 1}.\end{equation}
Now $F_{r,r,r}$ is the product of three functions in the class $\mathscr{C}_B(r)$, as defined in Definition \ref{cbx-def}. By (two applications of) Lemma \ref{sum-closure}, $F_{r,r,r} \in (\log^{O_B(1)} r) \mathscr{C}_{5B+ 6} (r^3)$, which means that 
\[ |F_{r,r,r}^{\wedge}(\lambda)| \ll_{B,\eps} r^{\eps - 1}.\] Together with \eqref{911-a}, this implies that
\[  |F^{\wedge}(\lambda)| \ll_{B,\eps} H^3 r^{\eps - 1}.\]
The same bound also applies to $F_{X_1, X_2, X_3}^{\wedge}(\lambda)$, simply by replacing the $\alpha_i$ in the definition of $F$ by $\alpha'_i(n) := \alpha_i(n) 1_{n \leq X_i}$. By the triangle inequality (and recalling that $r = \denom(\lambda)$) it follows that
\begin{equation}\label{911-c} |(F  - F_{X_1, X_2, X_3})^{\wedge}(\lambda)| \ll_{B,\eps} H^3\denom(\lambda)^{\eps - 1} .\end{equation}
We now have the desired bound $|(F  - F_{X_1, X_2, X_3})^{\wedge}(\lambda)| \ll_{B,\eps} H^3 X^{\eps - 1}$  in all cases: if $\denom(\lambda) \leq X$, apply \eqref{bound-claimed-rb}, whilst if $\denom(\lambda) > X$, apply \eqref{911-c}.

This completes the proof of Proposition \ref{ram-truncate}.\end{proof}

\section{Triple correlations with characters}\label{sec10}

In this section we build upon the previous section by introducing characters. The main result of the section (and the only one we will need elsewhere) is the following.

\begin{proposition}\label{character-correlation} Let $Q, R$ be parameters and suppose that $1 \leq Q \leq Q_1, Q_2, Q_3 \leq R$. Let $\chi$ be a primitive Dirichlet character of conductor $q \leq R$. Then 
we have
\[ \Vert F^+_{\chi, Q_1} \Lambda_{Q_2}^- H_{Q_3} - \tilde F^+_{\chi, R}\tilde\Lambda^-_{R} \tilde H_R \Vert^{\wedge}_{\infty} \ll R^{o(1)}Q^{-\frac{1}{4}}.\]
\end{proposition}

The reader may want to take the opportunity to recall the definitions of the various objects appearing here, which may be located as follows:

\begin{itemize}
\item $F_{\chi, Q}$ is a key object of the paper whose definition is given in Definition \ref{def43}. It has an alternative expression given in \eqref{f-chi-t-comb};
\item $\tilde F_{\chi, R}$ is a ``completed'' version of $F$, defined in Definition \ref{f-chi-definition} and with an alternative expression in \eqref{f-chi-comb};
\item $\Lambda_Q$ is a classical approximant to the von Mangoldt function, defined in \eqref{lambda-q-def};
\item $\tilde \Lambda_R$ is the completed version of this, defined in \eqref{tilde-lam-def} and with an alternative expression in \eqref{eq818}, which exhibits it as a normalised characteristic function of the $R$-rough numbers;
\item $H_Q$ is another key object in the paper, defined in Definition \ref{hq-definition}, and to be thought of as a kind of proxy for the square of the divisor function;
\item $\tilde H_R$ is a completed version of this, defined in Definition \ref{h-tilde-def}.
\end{itemize}

Finally, recall that $f^+(n) := f(n+1)$ and $f^-(n) := f(n-1)$, and that the basic properties of periodic functions and their Fourier series (and in particular the notation $\Vert \cdot \Vert_{\infty}^{\wedge}$) are laid out in Section \ref{rat-fourier-sec}.

For the rest of the section we regard $\chi$ and its conductor $q$ as fixed, with $q \leq R$.  Write (see Appendix \ref{char-app}) $\chi = \prod_{p | q} \chi_p$, where $\chi_p$ is a Dirichlet character of conductor $p^{\beta_p}$, where $q = \prod_{p | q} p^{\beta_p}$.

We divide the proof of Proposition \ref{character-correlation} into two parts, according to whether or not $q \geq Q^{\frac{1}{2}}$.

\subsection{The case of large conductor}

In this section we give some estimates which imply Proposition \ref{character-correlation} when $q \geq Q^{\frac{1}{2}}$. Recall the definition \eqref{vw-def} of $\vp $. It is also convenient to introduce the local factors $\chi_{\Z/p\Z}$ defined by $\vxp  = \frac{p}{p-1} \chi_p$ if $p \mid q$, and $\vxp  = \vp $ if $p \nmid q$. Then \eqref{f-chi-comb}, \eqref{tilde-lam-def} can be written in the form
\begin{equation}\label{f-chi-comb-abbrev} \tilde F_{\chi, R}  = \prod_{p \leq R} \vxpbar , \qquad  \tilde\Lambda_R = \prod_{p \leq R} \vp .\end{equation}

Our first estimate deals with the completed terms in Proposition \ref{character-correlation}.

\begin{lemma}\label{lem127} We have $\Vert \tilde F^+_{\chi, R}\tilde \Lambda_R^- \tilde H_R  \Vert^{\wedge}_{\infty} \ll \tau(q)^{16} q^{-\frac{1}{2}}$.
\end{lemma}
\begin{proof}
By \eqref{f-chi-comb-abbrev} and \eqref{b1b2-coprime}, we have
\begin{equation}\label{first-prod} \Vert \tilde F^+_{\chi, R}\tilde \Lambda_R^- \tilde H_R  \Vert^{\wedge}_{\infty} \leq \prod_{p \leq R} \Vert f_p \Vert^{\wedge}_{\infty}.\end{equation} where $f_p := \vxpbar ^+ \vp ^- \wp $. When $p \nmid q$ we have $\vxp  = \vp $, and so for these $p$ we have $\Vert f_p \Vert^{\wedge}_{\infty} \leq 1 + O(p^{-2})$ by \eqref{vpvpwp-bd}. The product over $p$ of all these contributions is $O(1)$.

When $p \mid q$, $f_p = \frac{p}{p-1} \overline{\chi}_p^+ \vp ^- \wp $, and so by \eqref{b1b2-noncoprime} we have
\begin{equation}\label{fps} \Vert f_p\Vert^{\wedge}_{\infty} \leq \frac{p}{p-1} \Vert \chi_p\Vert^{\wedge}_{\infty} \Vert \vp ^- \wp   \Vert^{\wedge}_1.\end{equation}
By Lemma \ref{gauss-periodic} we have
\begin{equation}\label{back-to} \Vert \chi_p \Vert^{\wedge}_{\infty} \leq p^{-\beta_p/2}.\end{equation}
By \eqref{b1b2-ell1} and \eqref{vp-four}, \eqref{wp-four} we have $\Vert \vp ^- \wp  \Vert^{\wedge}_1 \leq 8$.
Substituting this and \eqref{back-to} into \eqref{fps}, and using the crude bound $\frac{p}{p-1} \leq 2$, we obtain $\Vert f_p\Vert^{\wedge}_{\infty} \leq 16 p^{-\beta_p/2}$.

Taking products over $p$, we see that the right-hand side of \eqref{first-prod} is bounded by $\ll 16^{\omega(q)} q^{-1/2}$, where $\omega(q)$ denotes the number of distinct prime factors of $q$, which implies the claimed bound.\end{proof}

The second result of this section, which is a little harder, deals with the uncompleted terms in Proposition \ref{character-correlation}.

\begin{lemma}\label{lem127-trunc} Suppose that $1 \leq Q_1, Q_2, Q_3 \leq R$. Then \begin{equation}\label{char-cor-2} \Vert F^+_{\chi, Q_1}\Lambda_{Q_2} ^-H_{Q_3}  \Vert^{\wedge}_{\infty} \ll \tau(q)^{O(1)} (\log^{O(1)} R)  q^{-\frac{1}{2} }.\end{equation}
\end{lemma}
\begin{proof} Introduce
\begin{equation}\label{fd-def} f_{Q}(n) = \sum_{\substack{ r \leq Q  \\ (r, q) = 1}} \frac{\mu(r)}{\phi(r)} c_r(n). \end{equation}
(Of course, this definition depends on $q$, but we are regarding $q$ as fixed throughout the section and so suppress the dependence to keep the notation manageable.) By \eqref{f-chi-t-comb} we have
\begin{equation}\label{f-chi-q-fy} F_{\chi, Q} = \frac{q}{\phi(q)} \overline{\chi} f_Q.\end{equation}
Consider the definition of $\Lambda_Q$, viz
\begin{equation}\label{f-chi-t-comb-rpt}  \Lambda_Q(n) = \sum_{r \leq Q} \frac{\mu(r)}{\phi(r)} c_r(n).\end{equation}
Since the sum is supported where $r$ is squarefree, each $r$ in the sum may be written uniquely as $r = d r'$ where $d \mid q$  and $(r' , q) = 1$. By the multiplicative property $c_r(n) = c_{d}(n) c_{r'}(n)$ the contribution from a particular value of $d$ (to the sum over $r$ in \eqref{f-chi-t-comb-rpt}) is then 
\[ \frac{\mu(d)}{\phi(d)} c_d(n)  \sum_{\substack{ r' \leq Q/d \\ (r', q) = 1}} \frac{\mu(r')}{\phi(r')} c_{r'}(n)   .\] It follows that 
\begin{equation}\label{lam-q-sum}  \Lambda_Q =  \sum_{d | q} \frac{\mu(d)}{\phi(d)} c_d f_{Q/d}.\end{equation}
(Note, here and in what follows, $c_d$ is the Ramanujan sum considered as a function.)

Now we expand $F_{\chi, Q_1}^+ \Lambda_{Q_2}^-$ using \eqref{f-chi-q-fy} and \eqref{lam-q-sum}, obtaining
\begin{equation}\label{ff} F^+_{\chi, Q_1} \Lambda^-_{Q_2} = \frac{q}{\phi(q)}  \sum_{d | q} \frac{\mu(d)}{\phi(d)}  \overline{\chi}^+   c_d^- f_{Q_1}^+ f_{Q_2/d}^- .\end{equation}
 Recall from Definition \ref{hq-definition} that $H_{Q_3} = \sum_{r \leq Q_3} \eta(r) c_r$ where $\eta(r) = \prod_{p | r} \frac{3}{p+3}$ and $r$ is restricted to squarefrees.
Summing according to $e := (r, q)$, we see that $H_Q = \sum_{e | q} \eta(e) c_{e} h_{Q_3/e}$, where
\begin{equation}\label{he-def} h_{Y} := \sum_{\substack{r' \leq Y \\ (r',q) = 1}} \eta(r') c_{r'}.\end{equation}
Thus from \eqref{ff} we obtain

\begin{equation}\label{FFH} F^+_{\chi, Q_1} \Lambda^-_{Q_2} H_{Q_3} = \frac{q}{\phi(q)}  \sum_{d, e | q} \frac{\mu(d)}{\phi(d)} \eta(e) \overline{\chi}^+   c_d^- c_e f_{Q_1}^+ f_{Q_2/d}^- h_{Q_3/e}.\end{equation}

Consider first the term $f^+_{Q_1} f^-_{Q_2/d} h_{Q_3/e}$. The three functions $f_{Q_1}$, $f_{Q_2/d}$ and $h_{Q_3/e}$ all lie in the class $\mathscr{C}_2(R)$, as defined in Definition \ref{cbx-def}, by a trivial modification of the proof of Lemma \ref{h-lam-nice-class}. Since this class is invariant under shifts, it follows by two applications of Lemma \ref{sum-closure} that $f^+_{Q_1} f^-_{Q_2/d} h_{Q_3/e} \in (\log^{O(1)} R)\mathscr{C}_{16}(R)$, and so in particular
\begin{equation}\label{trip-prod} \Vert f^+_{Q_1} f^-_{Q_2/d} h_{Q_3/e} \Vert^{\wedge}_{\infty} \ll \log^{O(1)} R.\end{equation}
 Moreover, the rational Fourier expansion of this term has only frequencies which are coprime to $q$.

Now let us examine the other terms in \eqref{FFH}. For fixed $d, e$ (squarefree, else they are not relevant in \eqref{FFH}) we have the factorisation over prime powers
\begin{equation}\label{bp-prod}  \frac{\mu(d)}{\phi(d)}\eta(e) \overline{\chi}^+ c_{d}^- c_{e} = \prod_{p | q} b^{(d,e)}_p\end{equation}
where, for each $p$ dividing $q$, $b^{(d,e)}_p$ is a product of the following periodic functions:
\begin{itemize}
\item $\overline{\chi}^+_p$;
\item $-\frac{1}{p-1}c_p^-$ (if $p \mid d$) or $1$ (otherwise);
\item $\frac{3}{p+3}c_p$ (if $p \mid e$) or $1$ (otherwise).
\end{itemize}
Recall that $c_p(n) = -1$ if $p \nmid n$, and $c_p(n) = p-1$ if $p \mid n$.

Now if $z$ is any of the four possible products of the periodic functions ($-\frac{1}{p-1}c^-_p$ or $1$) and ($\frac{3}{p+3}c_p$  or $1$) other than the trivial product $z = 1 \cdot 1$ then we have $\Av |z| \ll \frac{1}{p}$ and so (by \eqref{fourier-avg}) $\Vert z \Vert^{\wedge}_{\infty} \ll \frac{1}{p}$. Since these functions $z$ all have $p$ as a period, their rational Fourier transforms are supported on rationals $\md{1}$ with denominator $p$ and so $\Vert z \Vert^{\wedge}_1 \ll p \Vert z \Vert^{\wedge}_{\infty} \ll 1$. This bound also holds when $z = 1$, when in fact $\Vert z \Vert^{\wedge}_1 = 1$.

Since, if $p \mid q$, $b^{(d,e)}_p = \overline{\chi}_p^+  z$ for one of the four possible choices of $z$, it follows from this, the bound $\Vert \overline{\chi}_p \Vert^{\wedge}_{\infty} \leq p^{-\beta_p/2}$ (cf. \eqref{back-to}) and \eqref{b1b2-noncoprime} that 
\[\Vert b_p^{(d,e)} \Vert^{\wedge}_{\infty} \leq \Vert \overline{\chi}_p^+ \Vert^{\wedge}_{\infty}  \Vert z \Vert^{\wedge}_1 \leq Cp^{-\beta_p/2}.\]
Taking products over $p | q$ and using \eqref{b1b2-coprime} gives, by \eqref{bp-prod}, 
\begin{equation}\label{mu-eta} \big\Vert \frac{\mu(d)}{\phi(d)} \eta(e) \overline{\chi}^+ c_d^- c_e  \big\Vert^{\wedge}_{\infty} \leq C^{\omega(q)} q^{-1/2} \ll \tau(q)^{O(1)} q^{-1/2}.\end{equation}

By \eqref{b1b2-coprime} (noting that the expression in \eqref{trip-prod} has period coprime to $q$, whereas the expression in \eqref{mu-eta} has period $q$), we obtain
\[ \big\Vert \frac{\mu(d)}{\phi(d)} \eta(e) \overline{\chi}^+   c_d^- c_e f_{Q_1}^+ f_{Q_2/d}^- h_{Q_3/e}\big\Vert^{\wedge}_{\infty} \ll \tau(q)^{O(1)} (\log^{O(1)} R) q^{-1/2},\] uniformly in $d, e | q$. Finally, summing over the $\tau(q)^2$ choices of $d, e$ and using \eqref{FFH} and the (very) crude bound $\frac{q}{\phi(q)} \leq \tau(q)$, we obtain \eqref{char-cor-2}.
\end{proof}

To conclude this subsection, we observe that Lemmas \ref{lem127} and \ref{lem127-trunc} and the triangle inequality immediately imply Proposition \ref{character-correlation} in the large conductor case $q \geq Q^{\frac{1}{2}}$.

\subsection{The case of small conductor}

We turn now to the small conductor case $q \leq Q^{\frac{1}{2}}$.  Let us look again at the expression \eqref{FFH} for the product $F_{\chi, Q_1}^+ \Lambda_{Q_2}^- H_{Q_3}$. 

We will compare this with a corresponding expression for $\tilde F_{\chi, R}^+ \tilde\Lambda^-_R \tilde H_R$. To obtain such an expression, we first introduce
\begin{equation}\label{tilde-f-def} \tilde f_R(n) := \sum_{\substack{r | R! \\ (r, q) = 1}} \frac{\mu(r)}{\phi(r)} c_r(n) = \prod_{\substack{p \leq R \\ p \nmid q}} \vp ,\end{equation} where the equivalence of the second and third expressions follows using multiplicativity and is explained after \eqref{caq-half}. By \eqref{caq-half}, we have 
\begin{equation}\label{f-tilde-little-f}  \tilde F_{\chi, R} = \frac{q}{\phi(q)} \overline{\chi} \tilde f_R.\end{equation} 
Write 
\begin{equation}\label{s-def} s := \prod_{p | q} \vp .\end{equation} Then, by \eqref{eq818}, 
\begin{equation}\label{lam-s-f} \tilde\Lambda_R = s \tilde f_R.\end{equation}
Next, introduce
\begin{equation}\label{tilde-h-def} \tilde h(n) = \sum_{\substack{r | R! \\ (r,q) = 1}} \eta(r)  c_r(n) = \prod_{\substack{p \leq R \\ p \nmid q}} \wp .  \end{equation}
The equivalence of the second and third terms here follows from \eqref{w-Fourier} and the multiplicativity of Ramanujan sums.  Finally, set
\begin{equation}\label{t-def} t := \prod_{p | q} \wp .\end{equation}
Then, by Definition \ref{h-tilde-def}, 
\begin{equation}\label{h-tilde-prod} \tilde H_R = t \tilde h.\end{equation}
Taking products of (shifted versions of) \eqref{f-tilde-little-f}, \eqref{lam-s-f} and \eqref{h-tilde-prod} gives
\begin{equation}\label{FFH-untruncated} \tilde F_{\chi, R}^+ \tilde\Lambda^-_R \tilde H_R =  \frac{q}{\phi(q)} \overline{\chi}^+   s^- t \tilde f_R^+ \tilde f_R^- \tilde h.\end{equation}

Now let $d, e$ be divisors of $q$ (as in \eqref{FFH}). Observe that (for any $Q_1, Q_2, Q_3, d, e$) the functions $f_{Q_1}$, $f_{Q_2/d}$ and $h_{Q_3/e}$ are Fourier-truncated versions, in the sense of Proposition \ref{ram-truncate}, of $\tilde f_R$, $\tilde f_R$ and $\tilde h$ respectively. This follows from simply inspecting the relevant definitions \eqref{fd-def}, \eqref{he-def}, \eqref{tilde-f-def} and \eqref{tilde-h-def}. Thus we may apply Proposition \ref{ram-truncate} with $X_1 = Q_1$, $X_2 = Q_2/d$, $X_3 = Q_3/e$, $B = 2$ and $(h_1, h_2, h_3) = (1, 0, -1)$ to get 
\begin{equation}\label{key-correlation-import} \big\Vert f_{Q_1}^+ f_{\frac{Q_2}{d}}^- h_{\frac{Q_3}{e}} - \tilde f_R^+ \tilde f_R^- \tilde h\big\Vert^{\wedge}_{\infty} \ll X^{-1 + o(1)} \ll Q^{-\frac{1}{2} + o(1)}.\end{equation} Here $X := \min(Q_1, \frac{Q_2}{d}, \frac{Q_3}{e})$, and $X \geq Q^{\frac{1}{2}}$ since $d, e \leq q \leq Q^{\frac{1}{2}}$.

Now (if $d, e | q$) the Fourier support of $\frac{\mu(d)}{\phi(d)} \eta(e)  \overline{\chi}^+ c_{d}^{-} c_e$ is on rationals whose denominators are composed of primes dividing $q$, whereas the Fourier support of $f_{Q_1}^+ f_{Q_2/d}^- h_{Q_3/e} - \tilde f_R^+ \tilde f_R^- \tilde h$ is on rationals whose denominators are coprime to $q$, as follows from the definitions \eqref{fd-def}, \eqref{he-def}, \eqref{tilde-f-def}, \eqref{tilde-h-def}. Therefore by \eqref{b1b2-coprime},  \eqref{mu-eta}  and \eqref{key-correlation-import} we have
\begin{align}\nonumber & \big\Vert  \frac{\mu(d)}{\phi(d)} \eta(e)  \overline{\chi}^+ c_{d}^{-} c_e (f_{Q_1}^+ f_{\frac{Q_2}{d}}^- h_{\frac{Q_3}{e}} - \tilde f_R^+ \tilde f_R^- \tilde h) \big\Vert^{\wedge}_{\infty} \\ & \leq \big\Vert \frac{\mu(d)}{\phi(d)}   \eta(e)  \overline{\chi}^+ c_{d}^{-} c_e  \big\Vert^{\wedge}_{\infty} \big\Vert f_{Q_1}^+ f_{\frac{Q_2}{d}}^- h_{\frac{Q_3}{e}} - \tilde f_R^+ \tilde f_R^- \tilde h\big\Vert^{\wedge}_{\infty} \ll Q^{-\frac{1}{2} + o(1)}. \label{1032}\end{align}
Now from the facts that $\vp  = 1 + \frac{\mu(p)}{\phi(p)} c_p$ (see \eqref{v-Fourier}) and $\wp  = 1 + \eta(p) c_p$ (see \eqref{w-Fourier}), the definitions \eqref{s-def}, \eqref{t-def} and the multiplicativity of $\mu, \phi, \eta$ and of $c_d(n)$ as a function of $d$, we have
\[ \sum_{d | q}  \frac{\mu(d)}{\phi(d)} c_d = s \quad \mbox{and} \quad \sum_{e | q} \eta(e) c_e = t.\]
Shifting the first equation by $-1$ and then summing \eqref{1032} over all divisors $d, e$ of $q$, we obtain (by the triangle inequality)
\[ \big\Vert \sum_{d, e | q} \frac{\mu(d)}{\phi(d)} \eta(e) \overline{\chi}^+c_d^- c_e f_{Q_1}^+ f_{\frac{Q_2}{d}}^- h_{\frac{Q_3}{e}} - \overline{\chi}^+ s^- t \tilde f_R^+ \tilde f_R^- \tilde h \big\Vert^{\wedge}_{\infty}  \ll \tau(q)^2 Q^{-\frac{1}{2} + o(1)}.\]
Multiplying through by $\frac{q}{\phi(q)}$ and comparing with \eqref{FFH}, \eqref{FFH-untruncated} we obtain
\[ \Vert F^+_{\chi, Q_1} \Lambda^-_{Q_2} H_{Q_3}  - \tilde F_{\chi, R}^+ \tilde\Lambda^-_R \tilde H_R  \Vert^{\wedge}_{\infty} \ll \tau(q)^3 Q^{-\frac{1}{2} + o(1)},\] which verifies Proposition \ref{character-correlation} in the small conductor case.

\section{Damping characters}\label{sec11}

\subsection{Introduction}

Before starting this section, the reader may wish to revisit the introductory discussion of Subsection \ref{subsec3.1}.

The main result of this section is Proposition \ref{key-proposition} below. It is the key input in constructing our damping function $D$. We begin by trying to motivate this result, which is not an easy task.

One of the key issues we face in this paper is that a Dirichlet character $\chi$ is not a Fourier-positive function. Indeed, if $\chi$ is primitive of conductor $q$, its rational Fourier coefficients will have magnitude $q^{-\frac{1}{2}}$ (by Lemma \ref{gauss-periodic}) but will not always be positive reals. We do, however, have $\chi(n)  \prec q^{\frac{1}{2}} 1_{q | n}$, because the function on the right \emph{is} Fourier-positive and its rational Fourier coefficients are $q^{-\frac{1}{2}}$ at at rationals $\frac{r}{q}$. The function $q^{\frac{1}{2}} 1_{q | n}$, whilst somewhat singular, is fairly small (of size $q^{-\frac{1}{2}}$) on average. 

Now in our paper we are faced with a function, namely $\Lambda_{\sharp, T,\sigmax}^+ \Lambda_Q^- H_Q$, which is a sum of objects involving shifted characters $\chi^+(n) = \chi(n+1)$. It is complicated, so for the purposes of this motivating discussion suppose we have a function $f(n) = 1 - \sum_{j \geq 2} j^{-2} \chi^+_j(n)$, which is a convergent sum of Dirichlet characters of the same general flavour. Suppose $\cond(\chi_j) = q_j$. Now as a first step we could consider $f(n) (1 + q_1^{\frac{1}{2}} 1_{q_1 | n})$, where the multiplying factor here is designed to dominate the Fourier oscillations from $\chi^+_1$. Unfortunately, though it achieves this aim, it introduces new cross-terms such as $-\frac{1}{4} \chi^+_2(n) (q_1^{\frac{1}{2}} 1_{q_1 | n})$, which still have large (but not necessarily positive) Fourier coefficients. 

To deal with this issue, we need a family of Fourier-positive functions like the $q^{\frac{1}{2}} 1_{q | n}$, but with the property that $\chi^+$ times any such function is Fourier-dominated by something like another function of the same type (where here $\cond(\chi)$ and $q$ might be different). Then, we can define $D$ as a suitable sum of such functions. 

A family of functions with something like this property is the set of convex combinations of the functions $\Phi_{q,a}$, which we introduced in Subsection \ref{subsec85}. Recall that $\Phi_{q,a}(n) := q^{\frac{5}{6}} 1_{q | n} e(\frac{an}{q^3})$. The reader will note that this is somewhat akin to the function $q^{\frac{1}{2}} 1_{q | n}$ appearing in the heuristic discussion above, only it is even more singular (but still with a small $\ell^1$ norm, of size $q^{-\frac{1}{6}}$) and now with an additional linear phase term. The need for the linear phase term arises from a resonance phenomenon which occurs when (for example), in the above heuristic discussion we have something like $q_1 = p$ and $\cond(\chi_2) = p^2$. Restricted to $n$ for which $p \mid n$, $\chi_2^+$ behaves like a linear phase; this sort of phenomenon ultimately comes from the Postnikov character formula (see Appendix \ref{appA}). 

Before turning to the precise statement of our main proposition, let us give some definitions. 

With $N$ the main global parameter in the paper, set
\begin{equation}\label{pi-def} \Pi := \prod_{p \leq N} p^N.\end{equation} 

\begin{definition}\label{def11.1}
Denote by $\mathscr{F}$ the class of all convex combinations $\sum_i \alpha_i \Phi_{q_i, a_i}$, where $q_i | \Pi$, $0 \leq \alpha_i \leq 1$, $\sum_i \alpha_i \leq 1$ and $a_i \in \Z/q_i^3 \Z$. Now let $\mathcal{P}$ be some finite set of primes. Write $\mathscr{F}_{\mathcal{P}}$ for the subclass spanned (as convex combinations) by those $\Phi_{q_i, a_i}$ where each $q_i$ is divisible by all the primes in $\mathcal{P}$. 
\end{definition}
\emph{Remarks.} Every element of $\mathscr{F}$ is Fourier-positive, since all the $\Phi_{q_i, a_i}$ are (by \eqref{phi-r-b-exp}) and by Proposition \ref{simple-pos-properties} (1). The role of $\Pi$ (which could be replaced by any fixed sufficiently composite large number) is to ensure that there are only finitely many functions $\Phi_{q,a}$ under consideration, and so the convex combinations may be assumed to be finite ones, and so every function in $\mathscr{F}$ is periodic. This is a very helpful technical convenience later on, specifically in Section \ref{sec13}, where we want to make sense of a convergent infinite sum of such functions.\vspace*{6pt}

Define $\nu(p)$, $p$ a prime, by
\begin{equation}\label{omega-p-def}\nu(p) := \left \{ \begin{array}{ll} 7 & \mbox{$p = 2$}; \\ 3 & \mbox{$3 \leq p \leq 17$}; \\ 2 & \mbox{$19 \leq p \leq 2^{18}$}; \\ 1 & \mbox{$p > 2^{18}$}. \end{array} \right.  \end{equation} (The exact values are unimportant but ones which work are given for definiteness; the key feature is that $\nu(p) = 1$ for sufficiently large $p$.)
Set
\begin{equation}\label{m-def} M := \prod_p \nu(p).\end{equation}

Here is the main result of the section.

\begin{proposition}\label{key-proposition}
Let $\chi$ be a primitive Dirichlet character with conductor $q$, and suppose that $q \mid \Pi$, where $\Pi$ is defined in \eqref{pi-def}. Let $\mathcal{P} = \mathcal{P}(\chi)$ be the set of primes dividing $q$. Let $F \in \mathscr{F}$. Then there is some $\tilde F = \tilde F(F, \chi) \in \mathscr{F}_{\mathcal{P}}$ such that $\chi^+  1_{(n-1,q) = 1} F \prec M \tilde F$, where $M$ is the absolute constant defined in \eqref{m-def}.
\end{proposition}
\emph{Remark.} The $1_{(n-1,q) = 1}$ term, which does not influence the analysis very much, is necessary later on. It did not feature in the heuristic discussion above.

The first main idea in the proof of Proposition \ref{key-proposition} is to work one prime at a time, and then take a product over primes at the end. Most of the section is devoted to the case $q = p^t$ a prime power.

\subsection{Prime power moduli -- setup}

\begin{definition}
For a given prime $p$ and exponent $\beta \geq 0$, denote by $\mathscr{F}^{\beta}_p$ the class of all finite convex combinations $\sum_i \alpha_i \Phi_{p^{\beta_i}, a_i}$, where $0 \leq \alpha_i \leq 1$, $\sum_i \alpha_i \leq 1$, the $\beta_i$ are positive integers with $\beta \leq \beta_i \leq N$, and $a_i \in \Z/p^{3\beta_i} \Z$. 
\end{definition}
\emph{Remark.} The upper bound $\beta_i \leq N$ ensures that $p^{\beta_i} \mid \Pi$, where $\Pi$ is defined in \eqref{pi-def}, and so $\mathscr{F}^{\beta}_p \subseteq \mathscr{F}$; once again, the role of this upper bound is to ensure that the spaces under consideration are finite-dimensional.

Note that each $\mathscr{F}_p^{\beta}$ is closed under convex combinations, and that we have the nesting \begin{equation}\label{nesting} \mathscr{F}_p^0 \supseteq \mathscr{F}_p^1 \supseteq \cdots.\end{equation} Every function in $\mathscr{F}_p^0$ is periodic and Fourier-positive (by Proposition \ref{simple-pos-properties} (1) and \eqref{phi-r-b-exp}). Here is the main result for prime-power moduli.

\begin{proposition}\label{main-pp-dirich} 
Let $p$ be a prime, and let $\chi$ be a primitive Dirichlet character to modulus $p^t$, $1 \leq t \leq N$. Let $F \in \mathscr{F}^{\beta}_p$. Then there is some $\tilde F = \tilde F(F, \chi) \in \mathscr{F}^{\max(\beta, 1)}_p$ such that $\chi^+ 1_{(p, n - 1) = 1} F \prec \nu(p)\tilde F$, where $\nu(p)$ is as defined in \eqref{omega-p-def}.\end{proposition}

We make some initial reductions. First, it suffices to handle the case $F = \Phi_{p^{\beta}, a}$ (for arbitrary $\beta$, not necessarily the same $\beta$). Indeed, suppose we have this case and write $F = \sum_i \alpha_i \Phi_{p^{\beta_i}, a_i}$. Suppose that $\chi^+ 1_{(p, n-1) = 1} \Phi_{p^{\beta_i}, a_i} \prec \nu(p) \tilde F_i$ where $\tilde F_i \in \mathscr{F}^{\max(\beta_i, 1)}_p$. By the nesting \eqref{nesting}, and the fact that $\beta_i \geq \beta$, we have $\tilde F_i \in \mathscr{F}_p^{\max(\beta, 1)}$. Then (using Proposition \ref{simple-pos-properties} (3) and (4)) we have $\chi^+ 1_{(p, n-1) = 1} F \prec \nu(p) \sum_i \alpha_i \tilde F_i $. Taking $\tilde F = \sum_i \alpha_i \tilde F_i$ completes the proof, noting that $\tilde F \in \mathscr{F}_p^{\max(\beta, 1)}$ because of the closure of this class under convex combinations.

We may further reduce to the case $a = 0$, that is to say $F(n) = p^{5\beta/6} 1_{p^{\beta} | n}$. Indeed, if for this $F$, $\chi^+ 1_{(n-1, p) = 1} F \prec \nu(p) \tilde F$ for some $\tilde{F} \in \mathscr{F}_p^{\max(\beta, 1)}$ then (by Proposition \ref{simple-pos-properties} (5))
\[ \chi^+(n) 1_{(n-1, p) = 1} F(n) e\big(\frac{an}{p^{3\beta}}\big) \prec \nu(p)\tilde F(n) e\big(\frac{an}{p^{3\beta}}\big) .\] The function $\tilde F(n) e(\frac{an}{p^{3 \beta}})$ can be written as a convex combination of functions of the form $\Phi_{p^{\tilde \beta}, \tilde a}(n) e(\frac{an}{p^{3\beta}}) = \Phi_{p^{\tilde \beta}, b}(n)$ where $b := \tilde a + p^{3(\tilde \beta - \beta)} a$, where $\tilde\beta \geq \max(\beta, 1)$.

These reductions reduce the task of proving Proposition \ref{main-pp-dirich} to the following.

\begin{proposition}\label{main-pp-dirich-reduce}
Let $\chi$ be a primitive Dirichlet character modulo $p^t$, $t \geq 1$. Let $F(n) = p^{5 \beta/6} 1_{p^{\beta} | n}$, where $\beta \geq 0$. Then there is some $\tilde F = \tilde F(F, \chi) \in \mathscr{F}_p^{\max(\beta, 1)}$ such that $\chi^+ 1_{(n-1, p) = 1} F \prec \nu(p) \tilde F$.
\end{proposition}

The proof of this falls into two somewhat distinct parts, according to the size of $\beta$. When $\beta \leq t/3$ we can use Gauss sum estimates: see subsection \ref{small-beta-case-p}. When $\beta > t/3$, if $p$ is odd we use the fact that the values $\chi(1 \pm p^{\beta} k)$ are polynomial phases in $k$ of degree at most $2$ by a special case of a well-known result of Postnikov. The case when $t/3 < \beta < t/2$, where the aforementioned phases are quadratic, is the most delicate part of the argument. For the details, see subsection \ref{large-beta-case-p}.

\subsection{The case of small $\beta$}\label{small-beta-case-p}

The case $\beta \leq t/3$ of Proposition \ref{main-pp-dirich-reduce} can be handled with Gauss sum estimates. 

\begin{proof}[Proof of Proposition \ref{main-pp-dirich-reduce}, $\beta \leq t/3$]
We use the identities
\[  1_{p^{\beta} | n} = p^{-\beta} \sum_{a \in \Z/p^{\beta} \Z} e\big(\frac{an}{p^{\beta}}\big) \quad \mbox{and} \quad 1_{(p, n-1) = 1} = 1 - \frac{1}{p} \sum_{a \in \Z/p\Z} e\big(\frac{a(n-1)}{p}\big)\] and the Gauss sum estimate
\begin{equation}\label{gauss-sum} \bigg| \sum_{x \in \Z/p^t\Z} \chi(x) e\big(-\frac{\lambda x}{p^t}\big) \bigg| \leq p^{t/2},\end{equation} which holds for all $\lambda \in \Z/p^t \Z$, even those divisible by $p$ (when the left-hand side is zero) by \cite[equation (3.12)]{ik} and the remarks following it. It follows that for all $\lambda \in \Z/p^t \Z$ we have

\begin{equation}\label{chi-f-lam} \bigg|\sum_{n \in \Z/p^t \Z} \chi(n+1) 1_{(p, n-1) = 1} F(n) e\big(-\frac{\lambda n}{p^t}\big)\bigg| \leq 2 p^{\frac{1}{2}t + \frac{5}{6}\beta},\end{equation}
where here $F(n) = p^{\frac{5\beta}{6}}1_{p^{\beta} | n}$ as in the statement of Proposition \ref{main-pp-dirich-reduce}. Since 
\begin{equation}\label{phi-t-56} \sum_{n \in \Z/p^t \Z} \Phi_{p^t, 0}(n) e\big(-\frac{\lambda n}{p^t}\big) = p^{\frac{5}{6}t},\end{equation} it follows from Proposition \ref{simple-pos-properties} (7) that 
\begin{equation} \label{chi-p-om} \chi(n+1) 1_{(p, n-1) = 1} F(n) \prec \nu(p)\Phi_{p^t,0}.\end{equation} Here we made use of the facts that $\frac{1}{2}t  + \frac{5}{6}\beta \leq (\frac{5}{6} - \frac{1}{18})t$ (which follows from the assumption that $\beta \leq t/3$) and $2p^{-t/18} \leq \nu(p)$ (which is a consequence of the definition \eqref{omega-p-def} of $\nu(p)$). Thus we may take $\tilde F = \Phi_{p^t, 0}$, noting that this lies in $\mathscr{F}_p^t$, which is contained in $\mathscr{F}_p^{\max(\beta, 1)}$ since $t \geq \max(3\beta, 1)$.
\end{proof}

It is convenient to note that the same analysis also handles the case $p = 2$ and $t \in \{1,2\}$ and any $\beta \leq t$, since $\nu(2) > 4$. Indeed, from \eqref{chi-f-lam}, \eqref{phi-t-56} and the fact that $2\cdot 2^{\frac{1}{2}t + \frac{5}{6}\beta} \leq 4 \cdot 2^{\frac{5}{6} t}$ when $\beta \leq t \leq 2$, we still have \eqref{chi-p-om}. This is convenient, as these small cases will be mildly exceptional in the analysis of the next subsection.

\subsection{The case of large $\beta$}\label{large-beta-case-p}

We turn now to the case $\beta > t/3$ of Proposition \ref{main-pp-dirich-reduce}. We assume throughout the section that $p^t \neq \{2,4\}$, because this case can be handled as described at the end of the last subsection. Note that since $t \geq 1$, in this case we automatically have $\beta \geq 1$ and so $\max(\beta, 1) = \beta$.
Let us first of all note that the case $\beta \geq t$ is trivial, because then if $p^{\beta} \mid n$ then $\chi(n+1) 1_{(p, n-1) = 1} = 1$ and so $\chi(n+1) 1_{(p, n-1) = 1} F(n) = F(n)$ for all $n$, so we may simply take $\tilde F = F$. Suppose, then, for the remainder of this section that $t/3 < \beta < t$, and write, for $\lambda \in \Z/p^t\Z$,
\begin{align*}  S(\lambda) & := \sum_{n \in \Z/p^t \Z} \chi(n+1) 1_{(p, n-1) = 1} 1_{p^{\beta} | n} e\big(-\frac{\lambda n}{p^t}\big) \\ & = \sum_{k \in \Z/p^{t - \beta} \Z} \chi(1 + p^{\beta} k)  e\big(-\frac{\lambda k}{p^{t - \beta}}\big).
\end{align*}
We now need a consequence of a result which is usually known as the Postnikov character formula. For a self-contained treatment of this topic, see Appendix \ref{appA}. 

It follows from Proposition \ref{post-formula} and the comments immediately following it that, since $p^t \notin \{2, 4\}$, there are $a_1, a_2 \in \Z$ such that 
\begin{equation}\label{s-lam-poly} S(\lambda) =  \sum_{k \in \Z/p^{t - \beta} \Z} e\bigg(\frac{(a_1 - \lambda) k + a_2 k^2 }{p^{t - \beta}}\bigg).\end{equation}

We now divide into two cases (it is important to note that the division into cases depends only on $\chi, \beta$, and \emph{not} on $\lambda$). 

\emph{Case 1 \textup{(}Purely linear case\textup{)}:}  we have $a_2 \equiv 0 \md{p^{t - \beta}}$. In this case we have simply

\[  S(\lambda)  = \sum_{k \in \Z/p^{t - \beta} \Z} e\bigg(\frac{(a_1 - \lambda) k}{p^{t - \beta}}\bigg) = \sum_{n \in \Z/p^t \Z} 1_{p^{\beta}| n} e\big(\frac{a_1 n}{p^t}\big) e\big(-\frac{\lambda n}{p^t}\big).\]
Therefore if $F(n) = p^{\frac{5}{6}\beta} 1_{p^{\beta} | n}$ then, by Proposition \ref{simple-pos-properties} (7) we have $\chi^+ 1_{(p, n-1) = 1} F \prec \tilde F$, where $\tilde F(n) := p^{\frac{5}{6}\beta} 1_{p^{\beta} | n} e(\frac{a_1 n}{p^t})$. Note that $\tilde F(n)$ is of the form $\Phi_{p^{\beta}, b}$, since we may write $\frac{a_1}{p^t} = \frac{b}{p^{3\beta}}$ since $\beta > t/3$. This confirms Proposition \ref{main-pp-dirich-reduce} in Case 1, the purely linear case.

\emph{Case 2 \textup{(}Not purely linear case\textup{)}:} $a_2 \not\equiv 0 \md{p^{t - \beta}}$. Suppose first that $p$ is odd; minor modifications, which we will indicate at the end, are required when $p = 2$. 

Given $\lambda$, define $u(\lambda)$ by
\begin{equation}u(\lambda) := \left\{ \begin{array}{ll} 0 & \mbox{if $p^{t - \beta} | a_1 - \lambda$} \\ t - \beta - v_p(a_1 - \lambda)  & \mbox{otherwise}.\end{array}  \right. \end{equation}
Thus $0 \leq u(\lambda) \leq t - \beta$; note that the definition has been made so that it is well-defined for $\lambda \in \Z/p^t\Z$. Lemma \ref{quadratic-bd} (and \eqref{s-lam-poly}) tells us that 
\begin{equation}\label{linear-est} |S(\lambda)| \leq p^{t - \beta - \frac{1}{2}u(\lambda)}.\end{equation}
When $u(\lambda) = 0$ we can do better: in this case we have $v_p(a_2)  \leq t - \beta - 1$ by the assumption that we are in the not purely linear case, and so Lemma \ref{quadratic-bd} gives
\begin{equation}\label{nonlinear-est} |S(\lambda)| \leq p^{t - \beta - \frac{1}{2}}.\end{equation}
Let $0 \leq v \leq t - \beta$. Define
\[ F_v(n) := p^{\frac{5}{6}(v + \beta)}1_{p^{v + \beta} | n} e\big(\frac{a_1 n}{p^t}\big).\]We note that $F_v$ is of the form $\Phi_{p^{\tilde\beta}, \tilde a}$ with $\tilde \beta \geq \beta$, so in particular lies in $\mathscr{F}_p^{\beta}$. Indeed, we can take $\tilde\beta = \beta + v$ and then write $\frac{a_1}{p^t} = \frac{\tilde a}{p^{3\tilde \beta}}$ since $3 \tilde \beta = 3(\beta + v) \geq 3 \beta > t$. Now (since $v \leq t - \beta$) $F_v(n)$ is a function on $\Z/p^t\Z$. We may compute its $\md{p^t}$ Fourier coefficient at $\lambda$ as follows:

\begin{align}\nonumber \sum_{n \in \Z/p^t \Z} F_v(n) e\big(-\frac{\lambda n}{p^t}\big) & =  p^{\frac{5}{6}(v + \beta)}  \sum_{k  \in \Z/p^{t - v - \beta} \Z} e\bigg(\frac{(a_1 - \lambda) k}{p^{t - v - \beta}}\bigg) \\ & = \left\{  \begin{array}{ll} 0 & \mbox{if $v < u(\lambda)$} \\ p^{\frac{5}{6}(v + \beta) + t - v - \beta} & \mbox{if $v \geq u(\lambda)$} .\end{array} \right.\label{fu-fourier} \end{align}

In particular, comparing with \eqref{linear-est} one sees that (writing $u = u(\lambda)$ for brevity)

\begin{align*}
\bigg| \sum_{n \in \Z/p^t \Z} & \chi(n+1) 1_{(n-1, p) = 1} F(n) e\big(-\frac{\lambda n}{p^t}\big)\bigg|  = p^{\frac{5}{6}\beta}|S(\lambda) |  \leq p^{\frac{5}{6}\beta + t - \beta - \frac{1}{2}u} \\  & = p^{-\frac{1}{3} u} p^{\frac{5}{6}(u + \beta) + t - u - \beta}  = p^{-\frac{1}{3} u} \sum_{n \in \Z/p^t \Z} F_u(n) e\big(-\frac{\lambda n}{p^t}\big) .
\end{align*}
When $u(\lambda) = 0$, we may use \eqref{nonlinear-est} instead of \eqref{linear-est}, obtaining the bound
\[ \bigg|\sum_{n \in \Z/p^t \Z} \chi(n+1) 1_{(n-1, p) = 1} F(n) e\big(-\frac{\lambda n}{p^t}\big) \bigg| \leq p^{-\frac{1}{3}} \sum_{n \in \Z/p^t \Z} F_0(n) e\big(-\frac{\lambda n}{p^t}\big) .\]

Since the Fourier coefficients of the $F_v$ are always non-negative (by \eqref{fu-fourier}) it follows that for \emph{all} $\lambda$ (no matter the value of $u(\lambda)$) we have

\[ \bigg|\sum_{n \in \Z/p^t \Z} \!\!\chi(n+1) 1_{(n-1, p) = 1}F(n) e\big(-\frac{\lambda n}{p^t}\big)\bigg| \leq  \sum_{u = 0}^{t - \beta} \alpha_u(p) \!\!\! \sum_{n \in \Z/p^t \Z} \!\! F_u(n) e\big(-\frac{\lambda n}{p^t}\big) ,\]

where $\alpha_0(p) = p^{-\frac{1}{3}}$ and $\alpha_u(p) = p^{-\frac{1}{3}u}$ for $u = 1,2,\dots$.

Thus, by another application of Proposition \eqref{simple-pos-properties} (7) we have  \[ \chi(n+1) 1_{(n-1, p) = 1} F(n) \prec \sum_{u = 0}^{t - \beta} \alpha_u(p) F_u(n).\]
Recalling that each $F_u$ lies in $\mathscr{F}_p^{\beta}$, we obtain the desired conclusion of Proposition \ref{main-pp-dirich-reduce} provided that 
\begin{equation}\label{om-check} \sum_{u = 0}^{t - \beta} \alpha_u(p) \leq \nu(p).\end{equation}
However, we have \[ \sum_{u = 0}^{t - \beta} \alpha_u(p) \leq p^{-\frac{1}{3}} + \sum_{i = 1}^{\infty} p^{-\frac{1}{3} i},\] and a calculation shows that this is at most $1$ for $p \geq 19$, and  less than $3$ for all odd $p$. Therefore \eqref{om-check} does indeed hold.

When $p = 2$, the bounds on $|S(\lambda)|$ coming from Lemma \ref{quadratic-bd} are weaker by a factor of $2^{\frac{1}{2}}$. This means we must take the $\alpha_u$ larger by this factor, and now the required numerical check is that $ 2^{\frac{1}{2}} (2^{-\frac{1}{3}} + \sum_{i = 1}^{\infty} 2^{-\frac{1}{3} i}) < 7 = \nu(2)$.

This completes the proof of Proposition \ref{main-pp-dirich-reduce} in Case 2, and hence of the whole proposition.

\subsection{General moduli}

We turn now to the proof of the main result of the section, Proposition \ref{key-proposition}. This essentially amounts to taking products of Proposition \ref{main-pp-dirich} over different primes $p$, but one must be slightly careful.

\begin{lemma}\label{crt}
Suppose that $q = q_1 q_2$ with $(q_1, q_2) = 1$. Then every $\Phi_{q_1, a_1} \Phi_{q_2, a_2}$ is of the form $\Phi_{q,a}$, and vice versa. 
\end{lemma}
\begin{proof}
This follows fairly immediately from the definition \eqref{phi-a-q-def} of the $\Phi_{q,a}$ and the Chinese remainder theorem.
\end{proof}

Finally, we are ready for the proof of the main result of the section.

\begin{proof}[Proof of Proposition \ref{key-proposition}]
It suffices to consider the case $F = \Phi_{r,b}$, from which the general case follows by taking appropriate convex combinations (using Proposition \ref{simple-pos-properties}). Suppose that $r = r' q$ where $q = \prod_{p \in \mathcal{P}} p^{\beta_p}$, and $r'$ is not divisible by any prime in $\mathcal{P}$, so that $(r', q) = 1$. We have $F = \Phi_{r,b} = \Phi_{r', b'}\Phi_{q,a}$ where 
$\Phi_{q,a} = \prod_{p \in \mathcal{P}} \Phi_{p^{\beta_p}, a_p}$, and here $a \in \Z/q^3 \Z$,  $a_p \in \Z/p^3 \Z$ and $b' \in \Z/(r')^3 \Z$. We also (see Appendix \ref{char-app}) have a factorisation $\chi = \prod_{p \in \mathcal{P}} \chi_p$, where each $\chi_p$ is a primitive Dirichlet character to modulus $p^{t_p}$, where $\prod_p p^{t_p} = q$ (and so, in particular, $t_p \geq 1$ for every $p \in \mathcal{P}$). 

By Proposition \ref{main-pp-dirich} (taking $\beta = 0$), for every $p$ there is some $\tilde F_p \in \mathscr{F}_p^1$ such that 
\[ \chi_p^+ 1_{(p, n-1) = 1} \Phi_{p^{\beta_p}, a_p} \prec \nu(p) \tilde F_p.\]
Taking products over $p \in \mathcal{P}$ (and using Proposition \ref{simple-pos-properties} (4)) gives
\[ \chi^+ 1_{(q, n-1) = 1} \Phi_{q, a} \prec M \prod_p \tilde F_p,\] where recall that $M := \prod_p \nu(p)$ with $\omega$ as defined in \eqref{omega-p-def}.
Finally, we may multiply both sides by $\Phi_{r', b'}$, which is Fourier-positive, so by Proposition \ref{simple-pos-properties} (3) we have
 \[ \chi^+ 1_{(q, n-1) = 1} F \prec M \tilde F,\] where
 $\tilde F := \Phi_{r', b'}\prod_{p \in \mathcal{P}} \tilde F_p$. Finally, we observe using Lemma \ref{crt} that $\tilde F$ is a convex sum over functions $\Phi_{q', a'}$, where $r' \prod_{p \in \mathcal{P}} p$ divides $q'$, the key here being that $r'$ is coprime to all the $p$. In particular, $\tilde F$ lies in $\mathscr{F}_{\mathcal{P}}$, which is what we wanted to prove.
\end{proof}

\section{An archimedean construction}\label{sec12}

\subsection{Introduction and statement of results}

In this section we describe the function $w$ in the definition \eqref{psi-prod} of our key function $\Psi$. We need $w$ to have a fairly long list of properties, which are detailed in Proposition \ref{prop15.2}, but the most important is that the exponential sum of $w(n)$ twisted by $n^{\rho - 1}$ can be controlled in terms of the exponential sum of $w$ alone (see items (4) and (5) of Proposition \ref{prop15.2}). The requirement for this property means that the most standard kind of cutoff functions such as the Fej\'er kernel are not sufficient for our purposes, and we instead make a construction modelled on the function $W(x) = x^{-\frac{1}{2}} e^{-x}$. Some further remarks on the particular choice of $W$ may be found at the end of Lemma \ref{gamma-t}.

\begin{proposition}\label{prop15.2} There is $w : \N \rightarrow [0, \infty)$ such that we have the following, for all $\rho = 1 - \sigma + i \gamma$ with $\sigma \leq \frac{1}{48}$ and $|\gamma| \leq N^{\frac{1}{8}}$:
\begin{enumerate}
\item $\sum_{n =1}^N n^{-\sigma} w(n)  \ll N^{1 - \sigma}$;
\item $\sum_{n = 1}^N w(n) \gg \frac{N}{\log N}$;\item uniformly for $\eta \in \R/\Z$ with $\Vert \eta\Vert \geq N^{-\frac{1}{2}}$ we have
\[ \sum_{n=1}^N n^{\rho - 1} w(n) e(\eta n) \ll N^{1 - \frac{1}{12}}.\]
\item uniformly for $\eta \in \R/\Z$ 
\begin{equation}\label{prop15.2.3} |\sum_{n = 1}^N n^{\rho - 1} w(n)  e(\eta n) | \leq 11 N^{-\frac{1}{4}\sigma} \sum_{n =1}^N w(n) \cos(2\pi \eta n) + O(N^{1 - \frac{1}{12}}).\end{equation}
\item If $\rho = 1 -\sigma$ is real we have, uniformly for $\eta \in \R/\Z$,
\[ \sum_{n = 1}^N n^{\rho - 1} w(n)  \cos(2\pi \eta n) \leq N^{-\frac{1}{4}\sigma} \sum_{n = 1}^N w(n) \cos(2\pi \eta n) + O(N^{1 - \frac{1}{12}}).\]
\item For any $X$, $1 \leq X \leq N$, we may write $w = w_0 + w_1$ where 
\begin{equation}\label{w0-bds} \sum_{n} w_0(n) \ll NX^{-1}\end{equation}
and $w_1(n) = \int_{\R} a(\xi) e( \frac{\xi n}{Y}) d\xi$ where $Y := \frac{N}{2 \log N}$ and
\begin{equation}\label{a-ell-1} \int_{\R} |a(\xi)| d\xi \ll N^{o(1)} X.\end{equation}
\end{enumerate}
\end{proposition}

For the rest of the section, take $Y :=  \frac{N}{2\log N}$ (as in the statement of (6) above). The idea is to set 
\begin{equation}\label{W-def} W(x) := x^{-\frac{1}{2}} e^{-x}\end{equation} for $x > 0$ and $W = 0$ elsewhere, and then define
\begin{equation}\label{w-def} w(n) := W\big(\frac{n}{Y}\big).\end{equation}

\subsection{Basic bounds for sums}

In this section we provide some bounds for $\sum_{n = 1}^N n^{\rho - 1} w(n) e(\eta n)$. In the first lemma, we compare this sum to an integral when $\eta$ is not too large. 

Here, and throughout the rest of the section, we write $T_{\rho} f(x) := x^{\rho - 1} f(x)$ for $x > 0$.

If $f : \R \rightarrow \R$ is a (suitable nice) function, we normalise the Fourier transform by $\widehat{f}(\xi) := \int_{\R} f(x) e(-x \xi) dx$, so that the inversion formula reads $f(x) = \int_{\R} \widehat{f}(\xi) e(\xi x) d\xi$.

\begin{lemma}\label{sum-integral-comp}
Suppose that $\frac{4}{5} \leq \beta \leq 1$ and $|\gamma| \leq N^{\frac{1}{8}}$. Then, uniformly for $|\eta| \leq N^{-\frac{1}{8}}$,
\[ \sum_{n = 1}^{N}n^{\rho - 1} w(n) e(\eta n) = Y^{\rho - 1} \widehat{T_{\rho} W}(-\eta Y) + O(N^{\frac{7}{8}}).\]
\end{lemma}
\begin{proof}
Set $F(x) := T_{\rho} W(x) e(\eta Y x) = x^{\rho - 1} W(x) e(\eta Y x)$ for $x > 0$.  
Writing $\rho = 1 - \sigma + i\gamma$ and 
\[ F(x) = x^{-\sigma} \cdot W(x)  \cdot e\big(\frac{\gamma}{2\pi} \log x + \eta Y x\big)\] and applying the product rule for derivatives, we see that 
\[ |F'(x)| \ll x^{-\sigma - 1} W(x)+ x^{-\sigma} W'(x) + x^{-\sigma} W(x) \big( \frac{|\gamma|}{x} + |\eta Y| \big).\]
For the first term we use the bound $W(x) \ll x^{-\frac{1}{2}}$, and for the second $W'(x) \ll x^{-\frac{3}{2}}$ (which may be checked by looking at $x < 1$ and $x \geq 1$ separately). We split the third term in two. The part with $\frac{|\gamma|}{x}$ can be bounded as for the first term, but now with an extra factor of $|\gamma| \leq N^{\frac{1}{8}}$. We keep the remaining part as it is. Putting all this together gives
\begin{equation}\label{f-prime-ptwise} |F'(x)| \ll x^{-\sigma - \frac{3}{2}} N^{\frac{1}{8}} + x^{-\sigma - \frac{1}{2}} e^{-x} |\eta| Y.\end{equation}

We have
\begin{equation}\label{sum-w-F} \sum_{n = 1}^N n^{\rho - 1} w(n) e(\eta n)  = Y^{\rho - 1} \sum_{n = 1}^{N} F\big(\frac{n}{Y}\big)\end{equation}
and
\begin{equation} \label{int-f-fourier}  \widehat{T_{\rho} W}(-\eta Y) = \int^{\infty}_0 F(x) dx .\end{equation}

Now we have
\begin{align}\nonumber   \big|  \frac{1}{Y} \sum_{n = 1}^N  &F\big(\frac{n}{Y}\big) - \int^{\infty}_0 F \big| \\ &   \leq  \int^{\frac{1}{Y}}_0 |F| +  \big|\frac{1}{Y}\sum_{n = 1}^N F\big(\frac{n}{Y}\big) - \int^{\frac{N+1}{Y}}_{\frac{1}{Y}} F\big| + \int^{\infty}_{\frac{N+1}{Y}} |F|.\label{sum-int}\end{align}
To bound the third term, we use the bound $|F(x)| \leq |W(x)| \leq e^{-x}$, valid for $x \geq 1$, to bound this term by $e^{-N/Y} = \frac{1}{N^2}$ (which is tiny).

For the second term in \eqref{sum-int}, we can use the mean value theorem to bound this by
\[ \frac{1}{Y^2} \sum_{n = 1}^N \sup_{x \in [\frac{n}{Y}, \frac{n+1}{Y}]} |F'(x)|. \]
To bound the sum appearing here, we use the derivative estimate \eqref{f-prime-ptwise}. The contribution from the first term there is
\[ \ll \frac{1}{Y^2} \sum_{n = 1}^N \big|\frac{n}{Y}\big|^{-\frac{3}{2} - \sigma} N^{\frac{1}{8}}  \ll Y^{\sigma - \frac{1}{2}} N^{\frac{1}{8}}.\]
The contribution from the second is
\[ \ll \frac{1}{Y^2} |\eta| Y \sum_{n = 1}^N \big|\frac{n}{Y}\big|^{-\frac{1}{2} - \sigma} e^{-n/Y} \ll |\eta|,\] where here we used the bound $\sum_{n = 1}^{\infty} n^{-s} e^{-n/Y} \ll_s Y^{1 - s}$ uniformly for $s \leq s_0 < 1$. 

Finally, to bound the first term in \eqref{sum-int}, we use $F(x) \ll x^{-\frac{1}{2} - \sigma}$, obtaining that this term is $\ll Y^{\sigma - \frac{1}{2}}$.

Putting all these bounds into \eqref{sum-int} gives
\[ \bigg|\frac{1}{Y} \sum_{n = 1}^N F\big(\frac{n}{Y}\big) - \int^{\infty}_0 F(x) dx\bigg| \ll Y^{\sigma-\frac{1}{2}} N^{\frac{1}{8}} + |\eta|.\] Multiplying by $Y^{\rho}$ and combining this with \eqref{sum-w-F} and \eqref{int-f-fourier} gives the desired result. 
\end{proof}

\begin{lemma}\label{w-decomp}
Let $0 < \eps < 1$. Then we may split $W = W_0 + W_1$, where $0 \leq W_0, W_1 \leq W$ pointwise, $\Supp (W_0) \subset [0, 2\eps]$ and $\Supp(W_1) \subset [\eps, \infty)$ and we have \[W_1(x) = \int_{\R} \widehat{W}_1 (\xi) e(\xi x) d\xi,\] where 
\begin{equation}\label{w1-fourier}| \widehat{W}_1(\xi)| \ll \min \big(1, \eps^{-\frac{1}{2}} |\xi|^{-1}, \eps^{-\frac{3}{2}} |\xi|^{-2}\big).\end{equation}
In particular, \[ \Vert \widehat{W}_1 \Vert_1 \ll \eps^{-\frac{1}{2}} \log\big(\frac{1}{\eps}\big).\]
\end{lemma}
\begin{proof} Let $\Psi : \R \rightarrow [0, \infty)$ be a smooth function with $\Psi(x) = 0$ for $x \leq 1$, $0 \leq \Psi(x) \leq 1$ for $x \in [1,2]$, $\Psi(x) = 1$ for $x \geq 2$ and $\Vert \Psi^{(m)} \Vert_{\infty} \ll_m 1$ for all $m$, where $\Psi^{(m)}$ denotes the $m$th derivative. The existence of such functions is standard.  Set $\Psi_{\eps}(x) := \Psi(x/\eps)$, and then define $W_1 := W \Psi_{\eps}$ and $W_0:= W (1 - \Psi_{\eps})$. The support properties and pointwise bounds in (1) are then clear.

For (2), we use the well-known bound (which follows by integration by parts)
\begin{equation}\label{fourier-second-der} |\widehat{W}_1(\xi)| \ll \Vert W_1^{(m)} \Vert_1 |\xi|^{-m},\end{equation} for $m = 0,1,2$.
We have
\begin{equation}\label{zero-der} \Vert W_1 \Vert_1 \leq \Vert W \Vert_1 \ll 1.\end{equation}
For the first derivative, note first that
\[ \Vert W'_1 \Vert_1 \leq \Vert W' \Psi_{\eps} \Vert_1 + \Vert W \Psi'_{\eps} \Vert_1.\]
Now $\Vert W' \Psi_{\eps} \Vert_1 \leq \int_{x > \eps} |W'(x)| dx  \ll \eps^{-\frac{1}{2}}$, whilst \[ \Vert W \Psi'_{\eps}\Vert_1 \ll \eps^{-1} \int^{2\eps}_{\eps} W(x) dx \ll \eps^{-\frac{1}{2}}.\] Therefore
\begin{equation}\label{first-der} \Vert W'_1 \Vert_1 \ll \eps^{-\frac{1}{2}}.\end{equation}
For the second derivative, we have
\[ \Vert W''_1 \Vert_1 \leq \Vert W'' \Psi_{\eps} \Vert_1 + 2 \Vert W' \Psi'_{\eps} \Vert_1 + \Vert W \Psi''_{\eps}\Vert_1.\]
Now \[ \Vert W'' \Psi_{\eps} \Vert_1 \leq \int_{x > \eps} |W''(x)| dx \ll \eps^{-\frac{3}{2}},\] \[ \Vert W' \Psi'_{\eps} \Vert_1 \ll \frac{1}{\eps} \int^{2\eps}_{\eps} |W'(x)| dx \ll \eps^{-\frac{3}{2}},\] and \[ \Vert W \Psi''_{\eps}\Vert_1 \ll \frac{1}{\eps^2} \int^{2\eps}_{\eps} W(x) dx \ll \eps^{-\frac{3}{2}}.\] 
Thus
\begin{equation}\label{second-der} \Vert W''_1 \Vert_1 \ll \eps^{-\frac{3}{2}}.\end{equation}
Combining \eqref{zero-der}, \eqref{first-der} and \eqref{second-der} with \eqref{fourier-second-der} gives the stated bound \eqref{w1-fourier}.

The bound on $\Vert \widehat{W}_1 \Vert_1$ follows by splitting into the ranges $|\xi| \leq \eps^{-\frac{1}{2}}$, $\eps^{-\frac{1}{2}} \leq |\xi| \leq \eps^{-1}$ and $|\xi| \geq \eps^{-1}$ and using \eqref{zero-der}, \eqref{first-der} and \eqref{second-der} respectively on these ranges. 

Finally, the fact that $\Vert \widehat{W}_1 \Vert_1$ is finite guarantees that $W_1$ is represented pointwise by its Fourier integral, by well-known theory. 
\end{proof}

\begin{proposition}\label{prop16.6}
Uniformly for $0 \leq \sigma \leq \frac{1}{4}$ , $|\gamma| \leq N^{\frac{1}{4}}$ and for $N^{-\frac{1}{2}} < \Vert \eta \Vert  \leq \frac{1}{2}$ we have
\[ \sum_{n = 1}^N n^{\rho-1} w(n) e(\eta n) \ll N^{1-\frac{1}{12}}.\]
\end{proposition}
\begin{proof}   Set $\alpha := \frac{1}{6}$ and $\eps := N^{-\alpha}$. (It is helpful, to see the structure of the argument, to work with $\alpha$ rather than the explicit fraction $\frac{1}{6}$ during the proof.) We use the decomposition of $W$ from Lemma \ref{w-decomp} with this value of $\eps$, and write the sum to be bounded as $S_0 + S_1$, where
\[ S_j := \sum_{n =1}^N  n^{\rho-1} W_j\big(\frac{n}{Y}\big) e(\eta n).\]
We bound $S_0$ in a fairly trivial manner, using the pointwise bound $W_0(x) \leq W(x) \leq x^{-\frac{1}{2}}$ to conclude that 
\[ S_0 \ll \sum_{n = 1}^{2N^{1 - \alpha}} n^{-\sigma} \big(\frac{n}{Y}\big)^{-\frac{1}{2}} \ll Y^{\frac{1}{2}} (N^{1 - \alpha})^{\frac{1}{2} - \sigma} \ll N^{1 - \frac{1}{2}\alpha} = N^{1 - \frac{1}{12}}.\] 
For $S_1$, we use the Fourier expansion of $W_1$ as described in Lemma \ref{w-decomp}, obtaining
\[ S_1 = \int_{\R} \widehat{W}_1(\xi) d\xi \sum_{n = 1}^N n^{\rho-1} e\big((\frac{\xi}{Y} + \eta) n\big) = S'_1 + S''_1,\] where $S'_1$ is the integral over $|\xi| \leq N^{\frac{1}{3}}$ and $S''_1$ is the remaining integral. Using Lemma \ref{w-decomp} again, we have
\[ S'_1 \ll N^{\frac{1}{2}\alpha + o(1)} \sup_{|\xi| \leq N^{\frac{1}{3}}} \bigg|\sum_{n = 1}^N n^{\rho-1} e\big((\frac{\xi}{Y} + \eta) n\big)\bigg|.\] To estimate this, we use the middle bound of Proposition \ref{prop43}, which implies (since $|\gamma| \leq N^{\frac{1}{4}}$ and $\beta \leq 1$) that 
\[ S'_1 \ll N^{\frac{1}{2}\alpha + o(1) + \frac{1}{4}} \sup_{|\xi| \leq N^{\frac{1}{3}}} \big\Vert \frac{\xi}{Y} + \eta \big\Vert^{-1} \ll N^{\frac{1}{2}\alpha + o(1) + \frac{1}{4} + \frac{1}{2}} \ll N^{1 - \frac{1}{12}},\] where in the penultimate step we used the assumption that $\Vert \eta \Vert \geq N^{-\frac{1}{2}}$ (and recall $Y = \frac{N}{2\log N}$).

Finally, for $S''_1$ we use the bound $|\widehat{W}_1(\xi)| \ll \eps^{-\frac{3}{2}} |\xi|^{-2}$ and the first (trivial) bound in Proposition \ref{prop43}, obtaining
\[ S''_1 \ll N \int_{|\xi| \geq N^{\frac{1}{3}}} |\widehat{W}_1(\xi)| d\xi \ll N^{1 + \frac{3}{2}\alpha - \frac{1}{3}} = N^{1 - \frac{1}{12}}.\]
Putting all this together concludes the proof.\end{proof}

\subsection{Bounds for integrals}

Throughout this section, suppose that $\rho = \beta + i\gamma = 1 - \sigma + i \gamma$.

Obviously, to make any use of Lemma \ref{sum-integral-comp}, we need some understanding of the Fourier transform of $T_{\rho}W$. We have
\begin{equation}\label{t-hat-rho} \widehat{T_{\rho} W}(\theta) = \int_0^{\infty} x^{\rho -1} W(x) e^{-2\pi i x \theta} dx = \int_0^{\infty} x^{-\sigma -\frac{1}{2} + i \gamma} e^{-x(1 + 2\pi i \theta)} dx. \end{equation}
To handle this, we make use of the fact that the formula
\begin{equation}\label{gamm-int} \int^{\infty}_0 x^{s - 1} e^{-ax} dx = \frac{\Gamma(s)}{a^s} \end{equation} holds for all $\Re s, \Re a > 0$, where $a^s = e^{s \log a}$ with $\log a$ being the principal branch of the logarithm on $\Re a > 0$. This fact is ``well-known'' but it is somewhat difficult to find a satisfactory reference; it appears as \cite[equation (1.5.1)]{lebedev}, which states (quite correctly, but with no further details) that it follows from the well-known case where $a$ is real, by the identity principle for holomorphic functions. It is also possible to give a proof using contour integration.

From \eqref{t-hat-rho}, \eqref{gamm-int} we have
\begin{equation}\label{t-rho-gam} \widehat{T_{\rho} W}(\theta) =  \frac{\Gamma(\frac{1}{2} - \sigma + i \gamma)}{(1 + 2\pi i \theta)^{\frac{1}{2} - \sigma + i \gamma}}.\end{equation}
Specialising to $\rho = 1$ (i.e. $\sigma = \gamma = 0$) gives a formula for $\widehat{W}$, namely
\[ \widehat{W}(\theta) = \Gamma(\textstyle\frac{1}{2}\displaystyle) (1 + 2\pi i \theta)^{-\frac{1}{2}}.\]
The bounds we need can now be read from the existing literature on the $\Gamma$-function, and are summarised in the following lemma.

\begin{lemma}\label{gamma-t}
We have\begin{equation}\label{compare-first}
|\widehat{T_{\rho} W}(\theta)| \leq 10 (1 + \theta^2)^{\frac{1}{2}\sigma} \Re \widehat{W}(\theta),
\end{equation}
uniformly for $0 \leq \sigma \leq \frac{1}{4}$ and all $\gamma, \theta \in \R$.
 If $\rho = 1 - \sigma$ is real with $\sigma$ small then, uniformly for $0 \leq \sigma \leq \frac{1}{4}$, we have 
\begin{equation}\label{real-small} \Re \widehat{T_{\rho} W}(\theta) \leq (1 + O(\sigma)) (1 + 4\pi^2 \theta^2)^{\frac{1}{2}\sigma} \Re\widehat{W}(\theta).\end{equation}
\end{lemma}
\begin{proof}
We first find a lower bound for $\Re \widehat{W}(\theta)$. We have \[  \Re (1 + 2\pi i \theta)^{-\frac{1}{2}} = (1 + 4\pi^2\theta^2)^{-\frac{1}{4}} \cos (\textstyle\frac{1}{2}\displaystyle \tan^{-1} (2\pi \theta)).\] Since $| \tan^{-1} (2\pi \theta)| \leq \frac{\pi}{2}$, we have $\cos (\frac{1}{2} \tan^{-1} (2\pi \theta)) \geq \cos (\frac{\pi}{4}) = \frac{1}{\sqrt{2}}$, and so
\begin{equation}\label{w-lower} \Re \widehat{W}(\theta) \geq 2^{-\frac{1}{2}}(1 + 4 \pi^2\theta^2)^{-\frac{1}{4}}. \end{equation}
We turn now to finding an upper bound for $|\widehat{T_{\rho} W}(\theta)|$. We will use the inequality
\begin{equation}\label{gamma-upper} |\Gamma(x + iy)| \leq 7 e^{-\pi |y|/2},\end{equation} which holds uniformly for $\frac{1}{4} \leq x \leq \frac{1}{2}$ and for all $y$. For the proof, see Lemma \ref{explicit-gam}. Also,
\[ |(1 + 2\pi i \theta)^{x + iy}| = (1 + 4\pi^2 \theta^2)^{x/2} e^{- y \tan^{-1}(2\pi  \theta)} \geq (1 + 4\pi^2\theta^2)^{x/2} e^{-\pi |y|/2}.\]
It follows immediately from this, \eqref{gamma-upper} and \eqref{t-rho-gam} that 
\begin{equation}\label{key-fourier} |\widehat{T_{\rho} W}(\theta)| \leq 7 (1 + 4\pi^2\theta^2)^{-\frac{1}{2}(\frac{1}{2} - \sigma)}.\end{equation}
Comparing \eqref{w-lower} and \eqref{key-fourier}, the desired bound \eqref{compare-first} follows immediately since $7 \sqrt{2} < 10$.

Now we turn to \eqref{real-small}.  By \eqref{t-rho-gam} we have 
\begin{align}\nonumber \Re & \widehat{T_{1 - \sigma} W}(\theta) \\ & =  \Gamma(\textstyle\frac{1}{2}\displaystyle  -  \sigma) (1 + 4\pi^2\theta^2)^{-\frac{1}{4} + \frac{1}{2}\sigma} \cos \big((\textstyle\frac{1}{2}\displaystyle -  \sigma) \tan^{-1} (2\pi \theta)\big).\label{t-real}\end{align}
We want to compare this with the same expression when $\sigma = 0$. Since $\Gamma$ is differentiable (and never zero) on $(0,1)$ we have $\Gamma(\frac{1}{2} - \sigma) = \Gamma(\frac{1}{2}) (1 + O(\sigma))$. Since $-\frac{\pi}{2} \leq \tan^{-1}(2 \pi \theta) \leq \frac{\pi}{2}$, and since the derivative of $\cos(tx)$ with respect to $x$ is uniformly bounded for $(t,x) \in [-\frac{\pi}{2}, \frac{\pi}{2}]\ \times \R$, we have \[ \cos \big((\textstyle\frac{1}{2}\displaystyle - \sigma) \tan^{-1} (2\pi \theta)\big) = (1 + O(\sigma)) \cos(\textstyle\frac{1}{2}\displaystyle \tan^{-1} (2\pi \theta))\] uniformly in $\theta$. 

Putting this information together with \eqref{t-real} leads to \eqref{real-small}.
\end{proof}

\emph{Remarks.} If we tried to run a similar analysis with $W(x) = e^{-x}$ (or any $x^{-t} e^{-x}$ with $t < \frac{1}{2}$) it would break down at the point where we asserted \eqref{gamma-upper}. Here, it is critical that we are only considering $\Gamma(z)$ for $\Re z \leq \frac{1}{2}$. Indeed, one may compute that for large positive $\gamma$, if we instead take $W(x) = e^{-x}$ then $\widehat{T_{1 + i \gamma} W}(\gamma) \sim \gamma^{-\frac{1}{2}}$ whilst $\widehat{W}(\gamma) \sim \gamma^{-1}$, so \eqref{compare-first} actually fails. What is going on here is that in $\widehat{T_{1 + i \gamma} W}(\gamma)$ we have a phase $x^{i \gamma}e(-\gamma x) = e (\gamma (\frac{1}{2\pi} \log x - x))$ which is stationary at $x = \frac{1}{2\pi}$, leading to less cancellation than one might na\"{\i}vely expect. The role of the factor $x^{-\frac{1}{2}}$ in the definition of $W$ is to reduce the relative effect of this stationary phase term by making the Fourier transform of $W$ decay more slowly, so that \eqref{compare-first} has a chance of holding.

\subsection{Proof of the main result}

We are now in a position to establish the main result of the section, Proposition \ref{prop15.2}.

\begin{proof}[Proof of Proposition \ref{prop15.2}]
(1) We have
\[ \sum_{n = 1}^N n^{-\sigma} w(n) \ll \sum_{n = 1}^N n^{-\sigma} \big(\frac{n}{Y}\big)^{-\frac{1}{2}} \ll Y^{\frac{1}{2}}N^{\frac{1}{2} - \sigma} \ll N^{1 - \sigma},\] with the implied constants being uniform for $\sigma \leq \frac{1}{4}$ (say). 

(2) We have
\[ \sum_{n = 1}^N w(n) \geq \sum_{n = 1}^Y w(n) \gg  \sum_{n = 1}^Y \big(\frac{n}{Y}\big)^{-\frac{1}{2}} \gg Y.\]

(3) This has already been established in Proposition \ref{prop16.6}.

(4) We first handle the case $|\eta| \leq N^{-\frac{1}{2}}$. By Lemma \ref{sum-integral-comp},
\begin{equation}\label{first-si} \sum_{n = 1}^{N}n^{\rho - 1} w(n) e(\eta n) = Y^{\rho - 1} \widehat{T_{\rho} W}(-\eta Y) + O(N^{\frac{7}{8}}).\end{equation}
Taking $\rho = 1$, 
\begin{equation}\label{second-si} \sum_{n = 1}^N w(n) \cos(2 \pi \eta n) = \Re \widehat{W}(-\eta Y) + O(N^{\frac{7}{8}}).\end{equation}
However, by \eqref{compare-first} we have
\begin{align*} |Y^{\rho - 1} \widehat{T_{\rho} W}(-\eta Y) | & \leq 10 Y^{-\sigma}(1 + |\eta Y|^2)^{\frac{1}{2}\sigma} \Re \widehat{W}(-\eta Y) \\ & \leq 11 N^{-\frac{1}{4}\sigma} \Re \widehat{W}(-\eta Y).\end{align*} Combining these results gives the conclusion when $|\eta| \leq N^{-\frac{1}{2}}$.

In the case $N^{-\frac{1}{2}} < | \eta |\leq \frac{1}{2}$, Proposition \ref{prop16.6} (applied to both sides of \eqref{prop15.2.3}) gives the result immediately. 

(5) The case when $N^{-\frac{1}{2}} < \Vert \eta \Vert \leq \frac{1}{2}$ may be handled exactly as in (4). For the case $|\eta| \leq N^{-\frac{1}{2}}$, we again start with \eqref{first-si} and \eqref{second-si}. In fact, we need to take real parts of \eqref{first-si} which, since $\rho = 1 - \sigma$ is real, gives
\begin{equation}\label{real-first-xi} \sum_{n = 1}^N n^{-\sigma} w(n) \cos (\eta n) = Y^{-\sigma} \Re \widehat{T_{1 - \sigma} W}(-\eta Y) + O(N^{\frac{7}{8}}).\end{equation}
By \eqref{real-small} we have 
\begin{align*} Y^{-\sigma} \Re \widehat{T_{1 - \sigma} W}(-\eta Y) & \leq (1 + O(\sigma)) Y^{-\sigma} (1 + |\eta Y|^2)^{\frac{1}{2}\sigma} \Re \widehat{W}(-\eta Y) \\ & \leq (1 + O(\sigma)) Y^{-\sigma}N^{\frac{1}{2}\sigma} \Re \widehat{W}(-\eta Y) \\ & \leq N^{-\frac{1}{4}\sigma} \Re \widehat{W}(-\eta Y).\end{align*}
Here, in the second step we used $1 + |\eta Y|^2 < 1 + (N^{-\frac{1}{2}} \frac{N}{\log N})^2 < N$, and in the last step we used $Y = \frac{N}{2\log N} > N^{\frac{7}{8}}$ and $1 + C \sigma \leq N^{\frac{1}{8}\sigma}$, which is valid uniformly for all real $\sigma \geq 0$ and for $N > e^{8C}$.

The desired result follows by combining this with \eqref{real-first-xi} and \eqref{second-si}.

(6) We use the decomposition $W = W_0 + W_1$ from Lemma \ref{w-decomp}, taking $\eps := X^{-2}$, and set $w_0(n) := W_0(\frac{n}{Y})$, $w_1(n) := W_1(\frac{n}{Y})$. From Lemma \ref{w-decomp} we have $w_1(n) = \int \widehat{W}_1(\xi) e(\frac{\xi}{Y} n) d\xi$, so we take $a(\xi) := \widehat{W}_1(\xi)$. The stated bound \eqref{a-ell-1} on $\int_{\R} |a(\xi)| d\xi$ is then immediate from Lemma \ref{w-decomp}. 

Also, 
\[ \sum_{n = 1}^N w_0(n) \leq \sum_{n \leq NX^{-2}} W\big(\frac{n}{Y}\big) \ll Y^{\frac{1}{2}} (NX^{-2})^{\frac{1}{2}} \ll NX^{-1},\] which is the claimed bound \eqref{w0-bds}.
\end{proof}

\part{The main construction} \label{p-adic-part}

We turn now to the main construction of a function $\Psi$ and to the proof of Theorem \ref{mainthm-3} (and hence, as explained shortly before the statement of that result, Theorems \ref{main-sarkozy} and \ref{vdc}). The reader may find it helpful to reread the broad overview of the construction, as given in Section \ref{subsec3.1}.

For the remainder of the paper, we fix a scale $T$ satisfying the conclusions of Proposition \ref{sec4-takeaway}, with $\sigmax = \frac{1}{48}$, this choice being with future applications of Proposition \ref{prop15.2} in mind. From now on, $\rho_1, \rho_2,\dots, \rho_J$ are the zeros in $\Xi_T := \Xi_T(\frac{1}{48})$ listed in order of decreasing real part, or equivalently in increasing order of $\sigma_j := 1 - \Re \rho_j$. For each $j$, write $\chi_j$ for the primitive character for which $\rho_j$ is a zero of $L(s, \chi_j)$, and set $q_j := \cond(\chi_j)$. 

At least one of the options (1) and (2) of Proposition \ref{sec4-takeaway} holds (indeed that is the content of the proposition). We refer to the two cases as the ``unexceptional'' and ``exceptional'' cases respectively. 

\emph{Remark.} As previously remarked, it is technically possible to be in case (1) yet still have an exceptional zero at scale $T$ in the sense of Proposition \ref{prop42}. However, there should be little danger of confusion since we will have no further need for the definitions in Proposition \ref{prop42}.

\section{Construction of the damping term}\label{sec13}

To ease notation a little, in this section only we write \[ \varphi_j := \tilde F_{\chi_j, T} \quad j = 1,2,\dots, J \quad \mbox{and} \quad \varphi_0 := \tilde\Lambda_T\] for the completed approximants as defined in Section \ref{sec6}. Thus, from \eqref{f-chi-comb}, \eqref{tilde-lam-def-2} we have
\begin{equation}\label{varphi-j-def} \varphi_j = \frac{q_j}{\phi(q_j)} \overline{\chi}_j \prod_{\substack{p \leq T \\ p \nmid q_j}} \vp , \quad j \geq 1 \quad \mbox{and}  \quad \varphi_0 = \prod_{p \leq T} \vp .\end{equation}
We are now in a position to actually construct the damping term $D$. Recall the definition of $\Phi_{r,b}$ (see \eqref{phi-a-q-def}), namely $\Phi_{r,b}(x) := r^{\frac{5}{6}} 1_{r | x} e(\frac{bx}{r^3})$.  Here are the key properties our damping term $D$ will have.

\begin{proposition}\label{key-D-prop}
There is a function $D : \Z \rightarrow \R$ with the following properties. $D(n)$ is a finite convex combination \begin{equation}\label{D-linear} \sum_{q} \sum_{a \in \Z/q^3\Z} \alpha_{q,a} \Phi_{q,a}(n),\end{equation} where
\begin{equation}\label{coeffs} 0 \leq \alpha_{q,a} \leq 1,\quad \sum_{q,a} \alpha_{q,a} = 1 \quad \mbox{and} \quad\alpha_{1,0} \geq \frac{3}{4},\end{equation}
and such that
\begin{equation}\label{9prop1-reform} N^{-\frac{1}{8}\sigma_j} \varphi_j^+ \varphi_0^- \tilde H_T D \prec \varphi_0^+ \varphi_0^- \tilde H_T D  \end{equation}
 for every $j \geq 1$ \textup{(}in the unexceptional case\textup{)} and for every $j \geq 2$ \textup{(}in the exceptional case\textup{)}.
\end{proposition}
\begin{proof} In this proof we write for positive integer $q$,
\[ \Lambda_{\Z/q\Z} := \prod_{p | q} \Lambda_{\Z/p\Z}, \quad \tau^2_{\Z/q\Z} := \prod_{p | q} \tau^2_{\Z/p\Z}.\] 

We first show that it is enough to construct a function $D : \Z \rightarrow \C$ with the required properties, that is to say without the condition that $D$ be real-valued. The key claim here is that, for every $j$, $\overline{\varphi}_j$ also appears as one of the functions $\varphi_{j'}$.

To prove the claim, consider the zero $\rho_j$ and the character $\chi_j$ associated to the index $j$. As explained in the proof of Lemma \ref{lem45}, $\overline{\rho}_j$ is a zero of $L(s, \overline{\chi}_j) = 0$, so for some $j'$ (possibly $j' = j$) we have $\rho_{j'} = \overline{\rho}_j$ and $\chi_{j'} = \overline{\chi}_j$. From \eqref{varphi-j-def} we see that $\varphi_{j'} = \overline{\varphi}_j$, which completes the proof of the claim.

Now suppose we are given $D : \Z \rightarrow \C$ satisfying Proposition \ref{key-D-prop}. We assert that $\Re D = \frac{1}{2}(D + \overline{D})$ has the same properties (and of course it \emph{is} real-valued). First note it is a finite convex combination of the form \eqref{D-linear} because $\overline{\Phi_{q,a}} = \Phi_{q, -a}$. 

Now, by the claim, if we have \eqref{9prop1-reform} for all $j$ then we also have 
\[ N^{-\frac{1}{8}\sigma_j} \overline{\varphi}_j^+ \varphi_0^- \tilde H_T D \prec \varphi_0^+ \varphi_0^- \tilde H_T D \] for all $j$. Taking complex conjugates and using the fact that both $\varphi_0 = \tilde\Lambda_T = \prod_{p \leq T} \vp $ and $\tilde H_T = \prod_{p \leq T} \wp $ are real-valued, and using Proposition \ref{simple-pos-properties} (3), we have
\[ N^{-\frac{1}{8}\sigma_j} \varphi_j^+ \varphi_0^- \tilde H_T \overline{D} \prec \varphi_0^+ \varphi_0^- \tilde H_T \overline{D} .\] Adding this to \eqref{9prop1-reform} and dividing by $2$ (and using Proposition \ref{simple-pos-properties} (4)) confirms our assertion about $\Re D$.

Now we turn to the construction of a $D : \Z \rightarrow \C$ satisfying the properties claimed. Recall (Definition \ref{def11.1}) the definitions of the classes $\mathscr{F}$ and $\mathscr{F}_{\mathcal{P}}$, where $\mathcal{P}$ is a set of primes. For each $j$, write $\mathcal{P}(j)$ for the set of primes dividing $q_j$. 

For the moment suppose that we are in the unexceptional case in which option (1) of Proposition \ref{sec4-takeaway} holds. For every finite (possibly empty) sequence $(j_1, j_2,j_3,\dots, j_m)$ of integers from $\{1,\dots, J\}$ we define $\Psi_{j_1,\dots, j_m} \in \mathscr{F}$ as follows. Set $\Psi_{\emptyset} := \Phi_{1,0}$, that is to say the constant function $1$. Assuming that $\Psi_{j_1,\dots, j_{m-1}}$ has already been defined, take $\Psi_{j_1,\dots, j_m} \in \mathscr{F}_{\mathcal{P}(j_m)}$ to be such that 
\begin{equation}\label{us-psi}  \chi^+_{j_m} 1_{(q_m, n - 1) = 1} \Psi_{j_1,\dots, j_{m-1}} \prec M \Psi_{j_1,\dots, j_{m}} .\end{equation}
The existence of $\Psi_{j_1,\dots, j_m}$ is guaranteed by Proposition \ref{key-proposition}.

Since $\Psi_{j_1,\dots, j_m}$ lies in $\mathscr{F}_{\mathcal{P}(j_m)}$, it is supported on $n$ for which $p | n$ for all $p \in \mathcal{P}(j_m)$. For such $n$ and for $p \in \mathcal{P}(j_m)$, we have $\vp ^+(n) = \vp ^-(n) = \frac{p}{p-1}$, and so $\Lambda_{\Z/q_m\Z}^+(n) = \Lambda_{\Z/q_m\Z}^-(n) = \frac{q_m}{\phi(q_m)}$. Therefore $\Lambda_{\Z/q_m\Z}^+ \Lambda_{\Z/q_m\Z}^- \Psi_{j_1,\dots, j_m} = (\frac{q_m}{\phi(q_m)})^2 \Psi_{j_1,\dots, j_m}$. Since also $\Lambda_{\Z/q_m \Z}^-(n) = \frac{q_m}{\phi(q_m)} 1_{(q_m, n - 1) = 1}$, it follows from \eqref{us-psi} that
\[  \frac{q_m}{\phi(q_m)} \chi^+_{j_m} \Lambda_{\Z/q_m\Z}^- \Psi_{j_1,\dots, j_{m-1}} \prec M \Lambda_{\Z/q_m\Z}^+ \Lambda_{\Z/q_m\Z}^- \Psi_{j_1,\dots, j_{m}} .\]
We now multiply both sides of this by $\tau^2_{\Z/q_m \Z} = \prod_{p | q_m} \wp$ and then by $\prod_{p \leq T : p \nmid q_m} \vp ^+ \vp ^- \wp $. By \eqref{w-Fourier} and Lemma \ref{lem43} (and Proposition \ref{simple-pos-properties} (1)) both of these are Fourier-positive, so by Proposition \ref{simple-pos-properties} (5) we have
\begin{align*}  \frac{q_m}{\phi(q_m)} & \chi^+_{j_m} \Lambda_{\Z/q_m\Z}^- \tau^2_{\Z/q_m\Z} \big( \prod_{p \leq T : p \nmid q_m} \vp ^+ \vp ^- \wp \big) \Psi_{j_1,\dots, j_{m-1}} \\ & \prec M \Lambda_{\Z/q_m\Z}^+ \Lambda_{\Z/q_m\Z}^- \tau^2_{\Z/q_m\Z} \big( \prod_{p \leq T : p \nmid q_m} \vp ^+ \vp ^- \wp \big) \Psi_{j_1,\dots, j_{m}} .\end{align*}
By \eqref{varphi-j-def} and the fact (Definition \ref{h-tilde-def}) that $\tilde H_T = \prod_{p \leq T} \wp $, this may equivalently be written as
\begin{equation}\label{without-D} \varphi_{j_m}^+ \varphi_0^- \tilde H_T \Psi_{j_1,\dots, j_{m-1}} \prec M \varphi_0^+ \varphi_0^- \tilde H_T \Psi_{j_1,\dots, j_m}.\end{equation}
Set
\[ \mu :=\sum_{m = 1}^{\infty}  \sum_{(j_1,j_2,\dots, j_m)} N^{-\frac{1}{16}(\sigma_{j_1} + \dots + \sigma_{j_m})} .\] Here, the tuples $(j_1,\dots, j_m)$ range over all $J^m$ tuples with $1 \leq j_i \leq J$ for each $i$. We claim that this is a convergent sum and that in fact $\mu \leq \frac{1}{3}$. For this, note that the $m$th term in the sum for $\mu$ is $\big(\sum_{j =1}^J N^{-\frac{1}{16}\sigma_{j}} \big)^m$, so by summing the geometric series it is enough to show that $\sum_{j = 1}^J N^{-\frac{1}{16}\sigma_{j}} \leq \frac{1}{4}$, which is (massively) true by \eqref{unexc-n-bd}.

Now we define \begin{equation}\label{d-series} D(n) := (1 + \mu)^{-1} \sum_{m = 0}^{\infty}  \sum_{(j_1,\dots, j_m)} N^{-\frac{1}{16}(\sigma_{j_1} + \dots + \sigma_{j_m})} \Psi_{j_1,\dots, j_m}(n),\end{equation} where again the tuples $(j_1,\dots, j_m)$ range over all $J^m$ tuples with $1 \leq j_i \leq J$ for each $i$. 
The functions $\Psi_{j_1,\dots, j_m}$ all lie in $\mathscr{F}$ (as defined in Definition \ref{def11.1}), and $D$ is a convex combination of them. Therefore $D$ also lies in $\mathscr{F}$; the key point here is that $\mathscr{F}$ is the convex span of a \emph{finite} set of functions $\Phi_{q,a}$, so the fact that $D$ is defined by an infinite sum rather than a finite one does not matter.

$D$ satisfies the first two conditions of \eqref{coeffs}. It also satisfies the third condition of \eqref{coeffs}: note that the term with $m = 0$ in $D(n)$ is $(1 + \mu)^{-1} \Psi_{\emptyset} = (1 + \mu)^{-1} \Phi_{1,0}$, so this follows from the fact that $\mu \leq \frac{1}{3}$.

Turning now to property \eqref{9prop1-reform}, we have the following deduction of this statement, where we will explain the steps below:
\begin{align*}
(1 + \mu) & N^{-\frac{1}{8}\sigma_j}\varphi_j^+ \varphi_{0}^- \tilde H_T D \\ & = \sum_{m=0}^{\infty} \sum_{(j_1,\dots, j_m)} N^{-\frac{1}{16}(\sigma_{j_1} + \dots + \sigma_{j_m})} N^{-\frac{1}{8}\sigma_{j}}\varphi_j^+ \varphi_0^- \tilde H_T \Psi_{j_1,\dots, j_m} \\ & \prec \sum_{m = 0}^{\infty} \sum_{(j_1,\dots, j_m)} N^{-\frac{1}{16}(\sigma_{j_1} + \dots + \sigma_{j_m})} N^{-\frac{1}{8} \sigma_{j}} M \varphi^+_0 \varphi^-_0 \tilde H_T \Psi_{j_1,\dots, j_m,j} \\ & \prec \sum_{m = 0}^{\infty} \sum_{(j_1,\dots, j_m)} N^{-\frac{1}{16}(\sigma_{j_1} + \dots + \sigma_{j_m} + \sigma_{j})}  \varphi_0^+ \varphi_0^- \tilde H_T \Psi_{j_1,\dots, j_m,j}  \\ & \prec (1 + \mu) \varphi_0^+ \varphi_0^- \tilde H_T D.
\end{align*}

Here, the second step uses \eqref{without-D}. In the third, penultimate, step we used the fact that for $j \in \{1,\dots, J\}$ we have
\begin{equation}\label{second-to-verify} N^{-\frac{1}{16}\sigma_j} \leq \frac{1}{M},\end{equation} by \eqref{unexc-n-bd}.

Finally, in the last step we added in the terms \[ N^{-\frac{1}{16}(\sigma_{j_1} + \dots + \sigma_{j_m} + \sigma_{j_{m+1}})}\varphi_0^+ \varphi_0^- \tilde H_T \Psi_{j_1,\dots, j_m,j_{m+1}}, \quad  j_{m+1} \neq j,\] and an extra copy of $\varphi_0^+ \varphi_0^- \tilde H_T$, noting that these are all Fourier-positive as we have previously seen (specifically, $\varphi_0^+ \varphi_0^- \tilde H_T$ is Fourier-positive by Lemma \ref{lem43} and Proposition  \ref{simple-pos-properties} (1), whilst $\Psi_{j_1,\dots, j_m,j_{m+1}}$ is Fourier-positive  because every function in $\mathscr{F}$ is, as remarked after Definition \ref{def11.1}.)

This completes the proof of property \eqref{9prop1-reform}, and hence of Proposition \ref{key-D-prop}, in the unexceptional case, that is to say when option (1) of Proposition \ref{sec4-takeaway} holds.

The argument in the exceptional case (option (2) of Proposition \ref{sec4-takeaway}) is essentially identical, except now the indices $j_i$ in the above construction only range over $\{2,\dots, J\}$. Instead of \eqref{second-to-verify}, we need the analogous result in which $j$ ranges over $\{2,\dots, J\}$. This is supplied by \eqref{exc-n-bd} (ignoring the contribution of $N^{-\frac{1}{16}\sigma_1}$ in that inequality). 

This completes the proof of Proposition \ref{key-D-prop}.\end{proof}

\section{Further bounds on correlations}\label{sec14}

The following section, which is a little technical, further enhances the main result of Section \ref{sec10} (Proposition \ref{character-correlation}), with an additional twist by a damping function such as the $D$ constructed in Proposition \ref{key-D-prop}. For the purposes of this section, the precise construction of $D$ is relatively unimportant and we consider general linear combinations
\[ D(n) = \sum_r \sum_{b \in (\Z/r^3\Z)^*} \alpha_{r,b} \Phi_{r,b}(n),\] where $\sum_{r,b} |\alpha_{r,b}| \leq 1$. The example we really care about, of course, is the one in Proposition \ref{key-D-prop}.
Given such a $D$, define also the truncated sum
\begin{equation}\label{d-trunc} D_X(n) := \sum_{r \leq X} \sum_{b \in (\Z/r^3\Z)^*} \alpha_{r,b} \Phi_{r,b}(n).\end{equation}

We pause to record a simple lemma which will be used a couple of times later on.

\begin{lemma}\label{d-trun-lem}
Let $D$ and its truncation $D_X$ be as above. Then 
\begin{equation}\label{d-trun-fourier} \Vert D_X \Vert^{\wedge}_1 \leq X^{\frac{5}{6}} \quad \mbox{and} \quad \support(D_X) \leq X^3.\end{equation}
\end{lemma}
\begin{proof}
These follow from \eqref{phi-ell1-bd}, \eqref{phi-support} and the definition of $D_X$.
\end{proof} 

Now we turn to the main result of the section. For the notation, see the comments after the statement of Proposition \ref{character-correlation}.

\begin{proposition}\label{final-correlation}
Let $1 \leq Q_1, Q_2, Q_3, Q_4 \leq R$ be parameters with $Q := \min(Q_1, Q_2, Q_3)$. Then, for any $q$, 
\[ \big\Vert F_{\chi, Q_1}^+ \Lambda_{Q_2}^- H_{Q_3} D_{Q_4}1_{q | n}  - \tilde F_{\chi, R}^+ \tilde \Lambda_R^- \tilde H_R D 1_{q | n} \big\Vert^{\wedge}_{\infty} \ll R^{o(1)} \big( Q_4^{-\frac{1}{6}} + Q_4^{\frac{5}{6}} Q^{-\frac{1}{4}} \big).
\]
\end{proposition}

We isolate a lemma from the proof.

\begin{lemma}\label{lem142}
We have $\Av | \tilde H_R \Phi_{r,b} | \ll r^{-\frac{1}{6} + o(1)}\log^{3} R$.
\end{lemma}
\begin{proof}
Recall from the definition, Definition \ref{h-tilde-def}, that $\tilde H_R = \prod_{p \leq R} \wp $. From the definition of $\wp $ (see \eqref{tau-sq-def}) we have the pointwise bound 
\begin{equation}\label{hr-tils}  \tilde H_R(n) \leq 4^{\omega_R(n)},\end{equation} where $\omega_R(n)$ denotes the number of distinct prime factors $p$ of $n$ with $p \leq R$. This upper bound can occasionally be rather large, so we must be careful. So that we can quote existing number-theoretic estimates, it is convenient to compute the average of $\tilde H_R(n) \Phi_{r,b}(n)$ by averaging over $n \leq X$ then letting $X$ tend to infinity.

By \eqref{hr-tils} This average is bounded above by
\[ r^{\frac{5}{6}} X^{-1}\sum_{\substack{n \leq X \\ r | n}} 4^{\omega_R(n)} \leq r^{\frac{5}{6}} X^{-1}4^{\omega_R(r)} \sum_{k \leq \frac{X}{r}} 4^{\omega_R(k)},\] where in the second step we used the inequality $\omega_R(rk) \leq \omega_R(k) + \omega_R(r)$. We have $4^{\omega_R(r)} \leq \tau(r)^2 \ll r^{o(1)}$ by the divisor bound. Applying Lemma \ref{power-res-div} (and letting $X \rightarrow \infty$) completes the proof.
\end{proof}

\begin{proof}[Proof of Proposition \ref{final-correlation}]
First we claim that 
\begin{equation}\label{first-baw-baw-claim} \big\Vert \tilde F_{\chi, R}^+ \tilde \Lambda_R^- \tilde H_R (D - D_{Q_4}) \big\Vert^{\wedge}_{\infty} \ll Q_4^{-\frac{1}{6}} R^{o(1)}.\end{equation}
We have
\[ D(n) - D_{Q_4}(n) = \sum_{r > Q_4} \sum_{b \mdsub{r^3}} \alpha_{b,r} \Phi_{r,b}(n),\] where $\sum |\alpha_{b,r}| \leq 1$, and therefore by Lemma \ref{lem142} 
\begin{equation}\label{eq1444}\Av |\tilde H_R (D - D_{Q_4})| \leq \max_{\substack{r > Q_4 \\ b \in \Z/r^3 \Z}} \Av | \tilde H_R \Phi_{r,b}| \ll Q_4^{-\frac{1}{6}} R^{o(1)}.\end{equation}
 Therefore by \eqref{fourier-avg} 
\begin{align*} \Vert \tilde F_{\chi, R}^+ & \tilde \Lambda_R^- \tilde H_R (D - D_{Q_4})  \Vert^{\wedge}_{\infty}  \leq \Av |\tilde F_{\chi, R}^+ \tilde \Lambda_R^- \tilde H_R (D - D_{Q_4}) | \\ & \leq \Vert \tilde F_{\chi, R} \Vert_{\infty} \Vert \tilde \Lambda_R \Vert_{\infty}  \Av | \tilde H_R (D - D_{Q_4})| \ll Q_4^{-\frac{1}{6}} R^{o(1)}, \end{align*}
where in the final step here we used \eqref{eq1444} and the pointwise bounds $\Vert \tilde F_{\chi_R} \Vert_{\infty}, \Vert \tilde \Lambda_R \Vert_{\infty} \ll \log R$, which follow from \eqref{f-chi-comb-abbrev} and the fact that $\prod_{p \leq R} \frac{p}{p-1} \ll \log R$. This confirms \eqref{first-baw-baw-claim}.

Now by Lemma \ref{d-trun-lem} and \eqref{b1b2-noncoprime} we have
\begin{equation}\label{per-f-dx} \Vert f D_{Q_4} \Vert^{\wedge}_{\infty} \leq Q^{\frac{5}{6}}_4 \Vert f \Vert^{\wedge}_{\infty}.\end{equation}
for any periodic function $f$.

Now we are ready to prove Proposition \ref{final-correlation} itself. From Proposition \ref{character-correlation} and \eqref{per-f-dx} we have
\[ \big\Vert F^+_{\chi, Q_1} \Lambda_{Q_2}^- H_{Q_3}D_{Q_4} - \tilde F^+_{\chi, R}\tilde \Lambda_R^- \tilde H_R D_{Q_4}  \big\Vert^{\wedge}_{\infty} \ll R^{o(1)} Q_4^{\frac{5}{6}}\min(Q_1, Q_2, Q_3)^{-\frac{1}{4}} .\]
From this, \eqref{first-baw-baw-claim} and the triangle inequality, we obtain Proposition \ref{final-correlation}, but without the $1_{q | n}$ terms. 

The same bound holds \emph{with} the $1_{q | n}$ terms included, using \eqref{b1b2-noncoprime} and the fact (see \eqref{simple-ell1-div}) that $\Vert (1_{q | n})^{\wedge} \Vert_1 = 1$.
\end{proof}

\section{The main construction and completion of the proof}

\subsection{Introduction}\label{subsec151}

We turn now to the proof of Theorem \ref{mainthm-3} itself. A brief overview of the construction of a suitable function $\Psi$ was given in Section \ref{section3} of the paper, and the reader might find it helpful to review that section.

In this section, we write $c_i := (\frac{1}{100})^i$. Here, $\frac{1}{100}$ is a convenient (but by no means optimised) small exponent which works for our argument. Any smaller exponent would also work but would give even weaker quantitative results. As elsewhere, we will use the $\lessapprox$, $\gtrapprox$ and $\tilde O$ notations to mean up to a factor of $N^{o(1)}$.

Let us now give an actual definition of the function $\Psi$. As it has been throughout the paper, let $N$ be some large integer parameter. Let $T$ be a scale satisfying the conclusions of Proposition \ref{sec4-takeaway}. 

As in Proposition \ref{sec4-takeaway}, we write $\rho_1,\rho_2,\dots, \rho_J$ for the zeros in $\Xi_T = \Xi_T(\frac{1}{48})$ in decreasing order of their real parts $\sigma_j$. Recall from  \eqref{crude-J} that we have the crude bound $J \ll T$.

\begin{definition}\label{psi-definition}
With parameters as above, define
\begin{equation}\label{psi-def} \Psi(n) := \Lambda'(n+1) \Lambda_{T^{c_1}}(n-1) H_{T^{c_1}}(n) D_{T^{c_2}}(n) E(n) w(n) 1_{n \leq N},\end{equation}
for $n \geq 0$, where
\begin{itemize}
\item $\Lambda'$ is the von Mangoldt function restricted to primes, that is to say $\Lambda'(p) = \log p$ for $p$ a prime;
\item $\Lambda_{Q}$ is the classical approximant to the von Mangoldt function, defined in \eqref{lambda-q-def};
\item $H_{Q}$ is a certain weight function which behaves like a kind of truncated version of $\tau^2$, and whose definition is given in Definition \ref{hq-definition};
\item $D_{Q}$ is the truncated version (as described in \eqref{d-trunc}) of $D$, the damping term whose existence was shown in Proposition \ref{key-D-prop};
\item $E$ is the constant function $1$ in the unexceptional case, and is $1_{q_1 | n}$ in the exceptional case, where $q_1$ is the conductor of the exceptional character $q_1$;
\item $w$ is the function in Proposition \ref{prop15.2}.
\end{itemize}
\end{definition}
Recall that the terms ``unexceptional'' and ``exceptional'' here refer to options (1) and (2) respectively in the conclusion of Proposition \ref{sec4-takeaway}.

Henceforth, it is convenient to adopt the convention that $q_1 := 1$ in the unexceptional case, thus $E(n) = 1_{q_1 | n}$ in all cases. This avoids a certain amount of rather tedious division into cases.

Our claim is that this function $\Psi$ satisfies the conclusions of Theorem \ref{mainthm-3}, with $\kappa_1 = \frac{1}{6} c_2\tau_0$ and $\kappa_2 = 2c_3\tau_0$, where $T = N^{\tau_0}$. Recall that, by Proposition \ref{sec4-takeaway}, $c \leq \tau_0 \leq \frac{1}{120}$ (where the constant $c$ comes from Proposition \ref{sec4-takeaway}). Let us first note that $\Psi$ is a real-valued function because all seven of the functions making up the product in \eqref{psi-def} are. It is also clear that $\Psi$ satisfies the support property (1) of Theorem \ref{mainthm-3}, because $\Lambda'$ does. It therefore remains to confirm items (2) and (3) of Theorem \ref{mainthm-3} with the values we have just stated, that is to say the estimates
\begin{equation}\label{psi-to-prove-1} \sum_{n = 1}^N \Psi(n) \cos(2 \pi \theta n) \gtrapprox -N T^{-\frac{1}{6} c_2} \end{equation} for all $\theta \in \R/\Z$ and
\begin{equation}\label{psi-to-prove-2} \sum_{n = 1}^N \Psi(n) \gtrapprox  N T^{- 2c_3}.\end{equation}

\subsection{Passing to a sharp approximant for $\Lambda$}\label{sec152}

The first step in establishing \eqref{psi-to-prove-1}, \eqref{psi-to-prove-2} is to reduce matters to the corresponding statements where $\Psi$ is replaced by an auxilliary function $\Psi'$ in which $\Lambda'$ is replaced by $\Lambda_{\sharp}$, where
\begin{equation}\label{lam-sharp-def-14} \Lambda_{\sharp} := \Lambda_{\sharp, T,\frac{1}{48}},\end{equation} the approximant to the primes constructed in Part \ref{partii}. Thus, we define
\begin{equation}\label{psi-dash} \Psi'(n) := \Lambda_{\sharp}(n+1) \Lambda_{T^{c_1}}(n-1) H_{T^{c_1}}(n) D_{T^{c_2}}(n) E(n) w(n) 1_{n \leq N}.\end{equation} for $n \geq 0$. Note that $\Psi'$ is also real-valued since, in addition to the functions mentioned above, $\Lambda_{\sharp}$ is real-valued by Lemma \ref{lem45}.

\begin{lemma}\label{psi-psidash} We have
\[ \sup_{\theta} \big|\sum_{n=1}^N (\Psi - \Psi')(n) \cos(2\pi \theta n)\big| \ll N T^{-\frac{1}{20}}.\]
\end{lemma}
\begin{proof}
Consider the function $F := \Lambda_{T^{c_1}}^-H_{T^{c_1}} D_{T^{c_2}} E$. By \eqref{b1b2-ell1} we have 
\begin{equation} \label{beta-sum-2} \Vert F \Vert^{\wedge}_1 \leq \Vert  \Lambda_{T^{c_1}} \Vert^{\wedge}_{1} \Vert H_{T^{c_1}} \Vert^{\wedge}_{1} \Vert D_{T^{c_2}} \Vert^{\wedge}_{1}  \Vert E \Vert^{\wedge}_1 \lessapprox T^{2c_1 + c_2},\end{equation} where here we have used \eqref{lam-q-ell1}, \eqref{hq-ell1-crude}, Lemma \ref{d-trun-lem} and \eqref{simple-ell1-div} respectively for the four terms in the product.

Let us write $w = w_0 + w_1$ as in Proposition \ref{prop15.2} (6), with $X := T^{\frac{1}{12}}$. Then
\begin{align} \nonumber \sum_{n = 1}^N (\Psi - \Psi')(n) \cos (2\pi \theta n) & = \sum_{i = 0, 1} \Re \sum_{n = 1}^N (\Lambda' (n) - \Lambda_{\sharp}(n)) F(n)w_i(n) e(n\theta) \\ & = E_0 + E_1,\label{i01}\end{align} say.
For the term $E_0$ we use the crude pointwise bound $|F(n)| \lessapprox T^{2c_1+ c_2}$, which follows from \eqref{beta-sum-2} and the rational Fourier expansion $F(n) = \sum_{\lambda \in \Q/\Z} F^{\wedge}(\lambda) e(\lambda n)$, obtaining a bound of 
\[ \lessapprox T^{2c_1 + c_2} \big(  \sum_{n = 1}^N \Lambda'(n) w_0(n) + \sum_{n = 1}^N |\Lambda_{\sharp}(n)| w_0(n) \big).\]
The first sum, over $\Lambda'$, is $\ll (\log N)\sum_{n = 1}^N w_0(n) \lessapprox N T^{-\frac{1}{12}}$ by \eqref{w0-bds} and the choice of $X$.

The second sum, over $\Lambda_{\sharp}$ is a little more delicate since this function is somewhat large for small $n$. To handle this we split the sum over $n$ into the ranges $1 \leq n \leq N^{\frac{1}{2}}$ and $N^{\frac{1}{2}} < n \leq N$. On both ranges we use the pointwise estimate \eqref{lam-sharp-ptwise} which, in the notation introduced at the start of the section, implies that 
\begin{equation}\label{lam-s-pt-rpt} |\Lambda_{\sharp}(n)| \lessapprox \sum_{j = 1}^{J} n^{-\sigma_j}.\end{equation}
For $n \leq N^{\frac{1}{2}}$ we just bound this crudely by $\lessapprox J \lessapprox T$ (this being \eqref{crude-J}).

Therefore 
\[ \sum_{n \leq N^{\frac{1}{2}}} |\Lambda_{\sharp}(n)| w_0(n)   \lessapprox  T \sum_{n \leq N^{\frac{1}{2}}} w(n)  \lessapprox T \sum_{n \leq N^{\frac{1}{2}}} \big(\frac{n}{Y}\big)^{-\frac{1}{2}} \lessapprox N^{\frac{3}{4}} T.\] (Recall here that $Y := \frac{N}{2\log N}$, this quantity being involved in the definition \eqref{w-def} of $w$.)
When $n > N^{\frac{1}{2}}$, \eqref{lam-s-pt-rpt} gives the stronger bound $|\Lambda_{\sharp}(n)| \lessapprox 1$, since we have the estimate $\sum_{j=1}^J N^{-\frac{1}{2}\sigma_j} \ll 1$ by \eqref{unexc-n-bd} (in the unexceptional case) or \eqref{exc-n-bd} (in the exceptional case). Therefore, by another application of \eqref{w0-bds}, 
\[ \sum_{N^{\frac{1}{2}} < n \leq N} |\Lambda_{\sharp}(n)| w_0(n) \lessapprox \sum_{n \leq N} w_0(n) \lessapprox N T^{-\frac{1}{12}}.\]
Putting all this together, we obtain
\begin{equation}\label{e0-bd} E_0 \lessapprox N T^{2c_1+ c_2 -\frac{1}{12}}.\end{equation} 

Turning now to the term $E_1$ in \eqref{i01} we use the expansion $w_1(n) = \int_{\R} a(\xi) e(\frac{\xi n}{Y}) d \xi$ from Proposition \ref{prop15.2} (6) and write the rational Fourier expansion of $F$ as $F(n) = \sum_{\lambda \in \Q/\Z} F^{\wedge}(\lambda) e(\lambda n)$, obtaining 
\begin{align*} & E_1  = \Re \sum_{\lambda \in \Q/\Z} F^{\wedge}(\lambda) \int_{\R} a(\xi) \sum_{n = 1}^N (\Lambda'(n) - \Lambda_{\sharp}(n)) e\big((\lambda + \theta + \frac{\xi}{Y}) n\big)d\xi \\ & \leq \sum_{\lambda \in \Q/\Z} |F^{\wedge}(\lambda)| \int_{\R} |a(\xi)| \sup_{\lambda,\xi} \bigg|\sum_{n = 1}^N (\Lambda'(n) - \Lambda_{\sharp}(n)) e\big((\lambda + \theta + \frac{\xi}{Y}) n\big)\bigg| d \xi.\end{align*}
By Proposition \ref{approx-fourier} with $Q = T$, \eqref{a-ell-1} and \eqref{beta-sum-2}, 
\[ E_1 \lessapprox N T^{-\frac{1}{6}} \cdot T^{2c_1+ c_2} \cdot  T^{\frac{1}{12}} .\]   Adding this to \eqref{e0-bd} concludes the proof of Lemma \ref{psi-psidash}, using the fact that $2c_1 + c_2 - \frac{1}{12} < -\frac{1}{20}$.
\end{proof}

Let us now formally state the remaining task, which is to establish the analogues of \eqref{psi-to-prove-1}, \eqref{psi-to-prove-2} with $\Psi'$ in place of $\Psi$.

\begin{proposition}\label{main-task-psidash}
The function $\Psi' : \Z \rightarrow \R$ satisfies
\begin{equation}\label{first-psi-dash}  \sum_{n = 1}^N \Psi'(n) \cos(2 \pi \theta n) \gtrapprox - N T^{-\frac{1}{6} c_2}\end{equation} and \begin{equation}\label{second-psi-dash}    \sum_{n = 1}^N \Psi'(n) \gtrapprox N T^{-2c_3}.\end{equation}
\end{proposition}

Note that \eqref{psi-to-prove-1} and \eqref{psi-to-prove-2} follow from these two statements by Lemma \ref{psi-psidash} (with $\theta = 0$ in the case of \eqref{psi-to-prove-2}) and the fact that $\frac{1}{6} c_2 , 2c_3 < \frac{1}{20}$. As explained at the end of Subsection \ref{subsec151}, this completes the proof of Theorem \ref{mainthm-3} and hence of all our main results, conditional upon Proposition \ref{main-task-psidash}.

\subsection{Fourier coefficients $\beta_j, \beta^{\trunc}_j$ and their properties}

For the remainder of the section, for each $j \in \{0,\dots, J\}$ write 
\[  \beta_j := (\tilde F_{\chi_j, T}^+\tilde\Lambda_T^- \tilde H_{T} D E)^{\wedge} \] and \[ \beta_j^{\trunc} := (F_{\chi_j, T}^+ \Lambda_{T^{c_1}}^- H_{T^{c_1}} D_{T^{c_2}} E)^{\wedge}.\]
To be clear, the $\wedge$ here indicates rational Fourier coefficients in the sense of Section \ref{rat-fourier-sec}, thus both $\beta_j$ and $\beta_j^{\trunc}$ are defined on $\Q/\Z$.

Let us collect some basic observations about these quantities.

\begin{lemma}\label{basic-beta}
We have the following:
\begin{enumerate}
\item $\Vert \beta_0\Vert_{\infty} \lessapprox 1$;
\item $\beta_0$ is real, non-negative and symmetric, i.e. $\beta_0(\lambda) = \beta_0(-\lambda)$;
\item $\beta_j^{\trunc}(\lambda)$ is supported where $\denom(\lambda) \leq T^3$;
\item $\sum_{\lambda \in \Q/\Z} | \beta^{\trunc}_j(\lambda)|  \ll T^2$;
\item $\Vert \beta^{\trunc}_{j} - \beta_{j}\Vert_{\infty} \lessapprox T^{-\frac{1}{6} c_2}$ uniformly for all $j$;
\item $|\beta_{j}(\lambda)| \leq N^{\frac{1}{8}\sigma_j} \beta_{0}(\lambda)$ for all $j \geq 1$ \textup{(}in the unexceptional case\textup{)} and all $j \geq 2$ \textup{(}in the exceptional case\textup{)};
\item in the exceptional case, $\beta_0 = \beta_1$;
\item we have $\beta_0(0) \gg q_1^{-1}$.
\end{enumerate} 
\end{lemma}
\begin{proof}
Recall throughout that $\tilde F_{\chi_0, T} = \tilde \Lambda_T$ and $F_{\chi_0, T} = \Lambda_T$. In particular, $\beta_0 = (\tilde\Lambda_T^+\tilde\Lambda_T^- \tilde H_{T} D E)^{\wedge}$.

(1)  By \eqref{fourier-avg} we have
\begin{align}\nonumber \Vert \beta_0 \Vert_{\infty}  = \Vert \tilde \Lambda_T^+ \tilde \Lambda_T^- \tilde H_T D E \Vert^{\wedge}_{\infty} & \leq \Av |\tilde \Lambda_T^+ \tilde \Lambda_T^- \tilde H_T D E| \\ & \leq \Vert \tilde \Lambda_T\Vert_{\infty}^2 \Vert E \Vert_{\infty} \Av | \tilde H_T D|.\label{bet0-max} \end{align}
By Lemma \ref{lem142} we have $\Av |\tilde H_T \Phi_{r,b} | \lessapprox 1$ uniformly for all $r, b$ and so, since $D$ is a convex combination of functions $\Phi_{r,b}$, $\Av |\tilde H_T D| \lessapprox 1$. We also have $\Vert E \Vert_{\infty} = 1$ (clear from the definition of $E$) and $\Vert \tilde \Lambda_T \Vert_{\infty} = \prod_{p \leq T} (1 - \frac{1}{p})^{-1} \ll \log T$. Substituting these bounds into \eqref{bet0-max} gives the result.

(2)  This is already implicit in the statement and proof of Proposition \ref{key-D-prop}, but to avoid potential confusion we give a self-contained explanation. The statement is equivalent to the assertion that $\tilde \Lambda_T^+ \tilde\Lambda_T^- \tilde H_T D E$ is Fourier-positive as defined in Section \ref{rat-fourier-sec}. By Proposition \ref{simple-pos-properties} (1), it is enough to show that each of the three functions $\tilde \Lambda_T^+ \tilde\Lambda_T^- \tilde H_T$, $D$ and $E$ are Fourier-positive. For $E$, this follows from the Fourier expansion $1_{q_1 | n} = \frac{1}{q_1} \sum_{a \in (\Z/q_1\Z)^*} e(\frac{an}{q_1})$. For $D$, this follows from the fact that $D$ is a convex combination of functions $\Phi_{r,b}$, all of which are Fourier-positive by \eqref{phi-r-b-exp}. Finally, we have (see \eqref{tilde-lam-def-2}) $\tilde \Lambda_T = \prod_{p \leq T} \vp $ and $\tilde H_T = \prod_{p \leq T} \wp $ (Definition \ref{h-tilde-def}), so the Fourier positivity of $\tilde \Lambda_T^+ \tilde\Lambda_T^- \tilde H_T$ is a consequence of Lemma \ref{lem43}, which asserts that each $\vp ^+ \vp ^- \wp $ is Fourier-positive.

Finally, the symmetry property $\beta_0(\lambda) = \beta_0(-\lambda)$ is a consequence of the fact that $\tilde \Lambda_T^+ \tilde \Lambda_T^- \tilde H_T D E$ is real-valued and the uniqueness of rational Fourier expansion. 

(3) By \eqref{support-prod} we have
\begin{equation}\label{crude-supports} \support(F_{\chi_j, T}^+ \Lambda_{T^{c_1}}^- H_{T^{c_1}} D_{T^{c_2}} E) \leq \support(F_{\chi_j, T})\support( \Lambda_{T^{c_1}})\support( H_{T^{c_1}}) \support(D_{T^{c_2}}) \support(E).\end{equation}
Now from the definition \eqref{fourier-trunc} we see that $\support(F_{\chi_j, T}) \leq T \cond(\chi_j) \leq T^2$. From the definition \eqref{lambda-q-def} we have $\support(\Lambda_{T^{c_1}}) \leq T^{c_1}$, and from Definition \ref{hq-definition} we have $\support(H_{T^{c_1}}) \leq T^{c_1}$. From \eqref{d-trun-fourier} we have $\support(D_{T^{c_2}}) \leq T^{3c_2}$, and finally from \eqref{simple-ell1-div} we have $\support(E) \leq q_1 \leq T^{c_3}$ (where the bound $q_1 \leq T^{c_3}$ comes from Proposition \ref{sec4-takeaway} (2)). Substituting all this into \eqref{crude-supports} (and using the crude bound $2c_1 + 3c_2 + c_3 < 1$), we obtain the required statement.

(4) By \eqref{b1b2-ell1} and the definition of $\beta_j^{\trunc}$, we have
\[ \sum_{\lambda \in \Q/\Z} |\beta_j^{\trunc} (\lambda) | \leq \Vert  F_{\chi_j, T}^{\wedge} \Vert_1 \Vert \Lambda^{\wedge}_{T^{c_1}} \Vert_1 \Vert H^{\wedge}_{T^{c_1}} \Vert_1  \Vert D^{\wedge}_{T^{c_2}}\Vert_1 \Vert  E^{\wedge} \Vert_1.\]
We have the bounds $\Vert F_{\chi_j, T}^{\wedge} \Vert_1 \leq T^{\frac{3}{2}}$ \eqref{f-chi-ell1},  $\Vert\Lambda_{T^{c_1}} \Vert^{\wedge}_1 \ll T^{c_1}$ \eqref{lam-q-ell1}, $\Vert H_{T^{c_1}}\Vert^{\wedge}_1 \ll T^{(1 + o(1))c_1}$ \eqref{hq-ell1-crude}, $\Vert D_{T^{c_2}} \Vert^{\wedge}_1 \ll T^{\frac{5}{6}c_2}$ \eqref{d-trun-fourier}, and $\Vert E \Vert^{\wedge}_1 = 1$ \eqref{simple-ell1-div}, and so the result follows, again using a crude bound $2c_1 + \frac{5}{6}c_2 < \frac{1}{2}$.

(5) This is little more than a restatement of an instance of Proposition \ref{final-correlation} in the language of this section. In that proposition, we have taken $Q_1 = R = T$, $Q_2 = Q_3 = T^{c_1}$, $Q_4  = T^{c_2}$ and $q = q_1$. Since $c_2$ is tiny compared to $\frac{1}{4}c_1$, the term $Q_4^{-\frac{1}{6}}$ in Proposition \ref{final-correlation} is dominant.

(6) This is \eqref{9prop1-reform} in which both sides have been multiplied by $E$ (which only makes a difference in the exceptional case). This preserves the $\prec$ relation, by Proposition \ref{simple-pos-properties} (5).

(7) This is a consequence of the fact that $\tilde F^+_{\chi_1, T} E = \tilde \Lambda^+_T E$, which is clear from the expressions (see \eqref{varphi-j-def})
\[ \tilde \Lambda_T = \prod_{p \leq T} \vp , \quad \tilde F_{\chi_1, T} = \frac{q_1}{\phi(q_1)} \overline{\chi}_1 \prod_{\substack{p \leq T \\ p \nmid q_1}} \vp ,\] since $\chi_1(n+1) = 1$ when $q_1 \mid n$. Here, $\chi_1$ is the (real) exceptional character to modulus $q_1$.

(8) From Proposition \ref{key-D-prop} and the fact that all the $\Phi_{q,a}$ are Fourier-positive, we have $\frac{3}{4} \Phi_{1,0} \prec D$. Note that $\Phi_{1,0}$ is just the constant function $1$. Since $E$ is Fourier-positive, we have $\frac{3}{4} E \prec DE$. Since $E(n) = \frac{1}{q_1} \sum_{a \in \Z/q_1 \Z} e(\frac{a n}{q_1})$, we also have $\frac{1}{q_1} \prec E$, and therefore $\frac{3}{4q_1} \prec DE$.

Using this and the fact that $\tilde \Lambda_T^+ \tilde\Lambda_T^- \tilde H_T$ is Fourier-positive (as remarked in the proof of (2) above) it follows from Proposition \ref{simple-pos-properties} that 
\begin{equation}\label{D-remove} \tilde \Lambda_T^+ \tilde\Lambda_T^- \tilde H_T DE \succ \frac{3}{4q_1} \tilde \Lambda_T^+ \tilde\Lambda_T^- \tilde H_T.\end{equation}
Now it follows from the proof of Lemma \ref{lem43} that the constant term in the rational Fourier expansion of $\vp ^+ \vp ^- \wp $ is $(1 + \frac{1}{p})(1 - \frac{1}{p})^{-2}(1 + \frac{3}{p})^{-1}$ for $p \geq 3$, and $\frac{16}{5}$ for $p = 2$. Multiplying over all $p \leq T$ (and using the fact that all frequencies of $\vp ^+ \vp ^- \wp $ have denominator $p$) it follows that the constant term in the rational Fourier expansion of $\tilde \Lambda_T^+ \tilde\Lambda_T^- \tilde H_T$ is 
\[  \frac{16}{5} \prod_{3 \leq p \leq T} \big(1 - \frac{1}{p}\big)^{-2} \big(1 + \frac{3}{p}\big)^{-1} \big(1 + \frac{1}{p}\big) \gg 1. \] Combining this with \eqref{D-remove} gives the required statement.
\end{proof}

\subsection{Preliminary analysis}

We now begin the proof of Proposition \ref{main-task-psidash}. The initial steps will be the same in the unexceptional and exceptional cases. 

Recall the definitions of $\Psi'$ (see \eqref{psi-dash}) and of $\Lambda_{\sharp} = \Lambda_{\sharp, T,\frac{1}{48}}$ (Definition \ref{sharp-approx-def}, \eqref{lam-sharp-conv}, \eqref{lam-sharp-def-14}). Substituting the latter definition into the former and expanding, and writing (just for the next few lines) $\varphi_j := F_{\chi_j, T}^+ \Lambda_{T^{c_1}}^- H_{T^{c_1}} D_{T^{c_2}} E$ for brevity, we have
\begin{equation}\label{first-expand}
 \sum_{n = 1}^N \Psi'(n) e(n \theta)  = \sum_{j = 0}^J \eps_j  \sum_{n= 1}^N \varphi_j(n) (n+1)^{\rho_j - 1} w(n) e(n \theta) ,\end{equation} where $\eps_0 = 1$ and $\eps_j = -1$ for $j \geq 1$.

The first, rather minor, task is to replace $(n+1)^{\rho_j - 1}$ here with $n^{\rho_j - 1}$, so that we can access Proposition \ref{prop15.2}.  First of all, we claim that if $\rho = \beta + i \gamma$ with $0 \leq \beta \leq 1$ and $|\gamma| \leq T$ then
\[ \big| (n+1)^{\rho - 1} - n^{\rho - 1}\big| = |n^{\rho - 1}| \big| (1 + \frac{1}{n})^{\rho - 1} - 1\big]| \ll \frac{T}{n},\] uniformly for positive integers $n$. 
To prove this claim, we first use the trivial bound $|n^{\rho - 1}| \leq 1$. Then, to estimate $| (1 + \frac{1}{n})^{\rho - 1} - 1|$, we split into the cases $n \geq T$ and $n < T$. In the case $n \geq T$, we use the inequality $|e^z - 1| \ll |z|$, valid for $|z| \leq 10$ (say), with $z = (\rho - 1) \log(1 + \frac{1}{n})$, which has $|z| \ll (1 + |\gamma|) \frac{1}{n} \ll \frac{T}{n}$. In the case $n < T$, we use the bound $|(1 + \frac{1}{n})^{\rho - 1}| = (1 + \frac{1}{n})^{\beta - 1} \ll 1$.

Using this claim, we see that the error in replacing $(n+1)^{\rho_j - 1}$ by $n^{\rho_j - 1}$ in \eqref{first-expand} is bounded above by
\begin{equation} \ll J T \sup_j \sum_{n = 1}^N |\varphi_j(n)| \frac{w(n)}{n}  \ll 
 T^{2 + 2c_1 + c_2}N^{o(1)}  \sum_{n = 1}^N  \frac{w(n)}{n}. \label{eq477}
\end{equation}
Here, the second bound follows from the first by \eqref{crude-J} and crude pointwise estimates for the constituent functions of $\varphi_j$: for $F_{\chi_j, T}$ and $\Lambda_{T^{c_1}}$ we have the pointwise bound of $N^{o(1)}$, by Corollary \ref{cor6.4}; for $H_{T^{c_1}}$ we have the (very) crude pointwise bound $|H_{T^{c_1}}(n)| \ll T^{2c_1}$, which follows from Definition \ref{hq-definition} and the trivial bound $|\eta(q)| \leq 1$; for $D_{T^{c_2}}$ we have the pointwise bound $|D_{T^{c_2}}(n)| \ll T^{\frac{5}{6}c_2} < T^{c_2}$, which is immediate from \eqref{D-linear} and the definition \eqref{phi-a-q-def} of $\Phi_{q,a}$; and finally for $E$ we have the trivial bound $E(n) \leq 1$.

Now by the definition \eqref{W-def}, \eqref{w-def} of $w(n)$ we have
\[ \sum_{n = 1}^N \frac{w(n)}{n} \leq \sum_{n = 1}^N \frac{1}{n} \big(\frac{n}{Y}\big)^{-\frac{1}{2}} \ll Y^{\frac{1}{2}} \ll N^{\frac{1}{2}}.\] Substituting into \eqref{eq477} shows that the error in replacing $(n+1)^{\rho_j - 1}$ by $n^{\rho_j - 1}$ in \eqref{first-expand} is $O(N^{\frac{2}{3}})$, that is to say
\[
 \sum_{n = 1}^N \Psi'(n) e(n \theta) = \sum_{j = 0}^J \eps_j  \sum_{n= 1}^N \varphi_j(n)  n^{\rho_j - 1} w(n) e(n \theta) + O(N^{\frac{2}{3}}).
\]

Now we expand $\varphi_j = F_{\chi_j, T}^+ \Lambda_{T^{c_1}}^- H_{T^{c_1}} D_{T^{c_2}}E$ as its rational Fourier series, with (by definition) coefficients $\beta^{\trunc}_j(\lambda)$. This gives
\begin{equation}\label{third-expand} 
 \sum_{n = 1}^N \Psi'(n) e(n \theta)  = \sum_{j = 0}^J \eps_j \!\!\! \sum_{\lambda \in \Q/\Z}
  \!\!\! \beta_j^{\trunc}(\lambda) \sum_{n =1}^N n^{\rho_j - 1} w(n)  e\big((\lambda+ \theta) n\big) + O(N^{\frac{2}{3}}).
\end{equation}

From now on we think of $\theta \in \R/\Z$ as fixed; some choices will depend on it, but all bounds will be uniform in $\theta$. Note that, by Lemma \ref{basic-beta} (3), all the $\lambda$ which are relevant in the above sum have $\denom(\lambda) \leq T^3$. Write $\lambda_0$ for the rational closest to $-\theta$ (mod 1) with $\denom(\lambda_0) \leq T^3$. We claim that the contribution to \eqref{third-expand} from all $\lambda \neq \lambda_0$ is small. To see this, note that for such $\lambda$ we have both $\Vert \lambda + \theta\Vert \geq \Vert \lambda_0+ \theta\Vert$ (by the assumption that $\lambda_0$ is closest to $-\theta$) and $\Vert \lambda - \lambda_0\Vert \geq T^{-6}$ (since $\denom(\lambda), \denom(\lambda_0) \leq T^3$). By the triangle inequality, we see that $\Vert \lambda + \theta \Vert \geq \frac{1}{2} T^{-6} > N^{-\frac{1}{2}}$, recalling here that we chose $T < N^{\frac{1}{120}}$ in Proposition \ref{sec4-takeaway}. Therefore the contribution to the sum in \eqref{third-expand} from $\lambda \neq  \lambda_0$ is bounded above by
\begin{align*} & \ll J T^6 \sup_{j} \Vert \beta^{\trunc}_j\Vert_{\infty} \sup_{\substack{ \Vert \eta \Vert \geq \frac{1}{2} N^{-\frac{1}{2}} \\ \rho \in \Xi^*_T(\frac{1}{48})}}  \big|\sum_{n = 1}^N n^{\rho - 1} w(n) e(\eta n)\big| \\ & \lessapprox J T^6 \cdot T^2 \cdot N^{1 - \frac{1}{12}} \ll N^{1 - \frac{1}{24}}.
\end{align*}
In the last step, we used $J \ll T$ (by \eqref{crude-J}), Lemma \ref{basic-beta} (4), Proposition \ref{prop15.2} (3) and the fact that $T < N^{\frac{1}{120}}$ (Proposition \ref{sec4-takeaway}). 

It follows from this analysis and \eqref{third-expand} that $\sum_{n = 1}^N \Psi'(n) e(n \theta)$ is
\begin{equation}\label{fourth-expand} 
  \sum_{j = 0}^J \eps_j \beta_j^{\trunc}(\lambda_0) \sum_{n =1}^N w(n) n^{\rho_j - 1} e\big((\lambda_0 + \theta) n\big) + O(N^{1 - \frac{1}{24}}).
\end{equation}

Next, we wish to replace the $\beta^{\trunc}$ terms by their untruncated variants $\beta$. The main tool here is Lemma \ref{basic-beta} (5) (which, recall, is a substantial result, requiring all of Sections \ref{comparison-sec}, \ref{sec10} and \ref{sec14} for its proof). By that result, the error introduced by replacing $\beta_j^{\trunc}$ by $\beta_j$ in \eqref{fourth-expand} is bounded above by
\begin{align*}
 \sum_{j = 0}^J |\beta^{\trunc}_j(\lambda_0) - \beta_j(\lambda_0)| & \big|\sum_{n = 1}^N n^{\rho_j - 1} w(n) e\big((\lambda_0 + \theta) n\big) \big| \\
& \lessapprox T^{-\frac{1}{6} c_2} \sum_{j = 0}^J \big|\sum_{n = 1}^N n^{\rho_j - 1} w(n) e(\eta n) \big| .
\end{align*}
Here, we have written $\eta := \lambda_0 + \theta$. By Proposition \ref{prop15.2} (4), this in turn is bounded by
\begin{equation}\label{to-bound-moment} \lessapprox T^{-\frac{1}{6} c_2}\big(\sum_{n = 1}^N w(n) \cos (2\pi \eta n) \big) \sum_{j = 0}^J N^{-\frac{1}{4}\sigma_j}  + O(J N^{1 - \frac{1}{12}}).
\end{equation}
By \eqref{unexc-n-bd} or \eqref{exc-n-bd} (depending on whether we are in the unexceptional case or the exceptional case) and the bound $\sum_{n = 1}^N w(n) \ll N$ (which is Proposition \ref{prop15.2} in the case $\sigma = 0$), \eqref{to-bound-moment} is $\lessapprox N T^{-\frac{1}{6} c_2}$. Note that the error term of $O(J N^{1 - \frac{1}{12}})$ is negligible since $J \ll T \leq N^{\frac{1}{120}}$ by Proposition \ref{sec4-takeaway}.

It follows from this analysis and \eqref{fourth-expand} that
\[
 \sum_{n = 1}^N \Psi'(n) e(n \theta)  = \sum_{j = 0}^J \eps_j \beta_j(\lambda_0) \sum_{n =1}^N w(n) n^{\rho_j - 1} e(\eta n) + \tilde O(N T^{-\frac{1}{6} c_2}).
\]
Recall here that $\eta := \lambda_0 + \theta$; observe also that the error term of $O(N^{1 - \frac{1}{24}})$ in \eqref{fourth-expand} is much smaller than the error term here, so was absorbed into it.

We may run an essentially identical analysis with $-\theta$ in place of $\theta$. Since $-\lambda_0$ is the closest rational to $-\theta$ (mod 1) with $\denom(\lambda_0) \leq T^3$, we obtain
\[ \sum_{n = 1}^N \Psi'(n) e(-n\theta) = \sum_{j = 0}^J \eps_j \beta_j(-\lambda_0) \sum_{n =1}^N w(n) n^{\rho_j - 1} e(-\eta n) + \tilde{O}(N  T^{-\frac{1}{6} c_2}).\]
Adding these together, and using the symmetry of $\beta_0$ (Lemma \ref{basic-beta} (2)) we have
\begin{align}\nonumber 2\sum_{n = 1}^N&\Psi'(n) \cos(2 \pi \theta n)  = 2\beta_0(\lambda_0) \sum_{n = 1}^N w(n) \cos(2 \pi \eta n) \\  & - \sum_{\pm} \sum_{j \geq 1}\beta_j(\pm \lambda_0)  \sum_{n =1}^N w(n) n^{\rho_j - 1}  e(\pm\eta n) + \tilde{O}(N T^{-\frac{1}{6} c_2}).\label{to-quote-exc} \end{align} (The $\sum_{\pm}$ notation means that we take the sum of two terms, one for each choice of sign.) 

From this point on, the treatments in the unexceptional and exceptional cases are slightly different.

\subsection{The unexceptional case}

We start with \eqref{to-quote-exc}. Taking absolute values of the terms with $j \geq 1$ and using Lemma \ref{basic-beta} (6) gives
\begin{align}\nonumber  \sum_{n = 1}^N & \Psi'(n) \cos(2 \pi \theta n )  \geq \beta_0(\lambda_0) \bigg( 2\sum_n w(n) \cos(2 \pi \eta n) -  \\ & - \sum_{\pm}\sum_{j =1}^J N^{\frac{1}{8}\sigma_j}\big|\sum_{n = 1}^N n^{\rho_j - 1}  w(n)  e(\pm \eta n)\big| \bigg)+ \tilde O(N T^{-\frac{1}{6} c_2}). \label{shortly}\end{align} 

Proposition \ref{prop15.2} (4) gives that
\begin{equation}\label{1214}  \big|\sum_{n = 1}^N n^{\rho_j - 1} w(n) e(\pm\eta n)\big| \leq 11 N^{-\frac{1}{4} \sigma_j} \sum_{n = 1}^N w(n) \cos(2 \pi \eta n) + O(N^{1 - \frac{1}{12}}).\end{equation}

When substituted into \eqref{shortly}, the contribution of the error terms $O(N^{1 - \frac{1}{12}})$ is $\ll |\beta_0(\lambda_0)| J \sup_j N^{\frac{1}{8} \sigma_j} N^{1 - \frac{1}{12}} \lessapprox N T^{-\frac{1}{6} c_2}$. Here, we used the fact that $\sigma_j \leq \frac{1}{48}$ for all $j$; the bound $|\beta_0(\lambda_0)| \lessapprox 1$ (Lemma \ref{basic-beta} (1)); the bound $J \ll T$ \eqref{crude-J}; and finally the fact that $T \leq N^{\frac{1}{120}}$ (Proposition \ref{sec4-takeaway}).

Therefore, substituting \eqref{1214} into \eqref{shortly} gives
\begin{align}\nonumber
&  \sum_{n = 1}^N \Psi'(n) \cos(2 \pi \theta n)  \\ & \geq 2\beta_0(\lambda_0) \big(\sum_n  w(n) \cos(2 \pi \eta n)\big) \big(1 - 11\sum_{j =1}^J N^{-\frac{1}{8} \sigma_j}\big)  - \tilde O(N T^{-\frac{1}{6} c_2}).\label{almost-unexc}\end{align}
 Now by \eqref{unexc-n-bd} (and since $M$ is vastly bigger than $11$) we have $1 - 11\sum_{j = 1}^J N^{-\frac{1}{8} \sigma_j} \geq \frac{1}{2}$. Also, taking $\rho = 1$ in Proposition \ref{prop15.2} (4) gives
 \begin{equation}\label{w-cos-lower} \sum_{n = 1}^N w(n) \cos(2 \pi \eta n) \geq - O(N^{1 - \frac{1}{12}}).\end{equation} Since $|\beta_0(\lambda_0)| \lessapprox 1$ (Lemma \ref{basic-beta} (1)) it follows from these observations and \eqref{almost-unexc} that 
 \[  \sum_{n = 1}^N \Psi'(n) \cos(2 \pi \theta n)  \geq - \tilde O(N T^{-\frac{1}{6} c_2}).\]
 This completes the proof of \eqref{first-psi-dash} (in the unexceptional case).

We turn now to \eqref{second-psi-dash}. We set $\theta = 0$ in the above analysis, thus $\lambda_0 = \eta = 0$ and \eqref{almost-unexc} gives
\[  \sum_{n = 1}^N \Psi'(n) \geq 2 \beta_0(0) \big(\sum_{n = 1}^N w(n)\big)\big(1 - 11 \sum_{j = 1}^J N^{-\frac{1}{8}\sigma_j}\big) - \tilde O(N T^{-\frac{1}{6} c_2}).\]
By \eqref{unexc-n-bd}, Proposition \ref{prop15.2} (2) and Lemma \ref{basic-beta} (8), it follows that 

\[ \sum_{n = 1}^N \Psi'(n) \geq c \frac{N}{\log N} - \tilde O(N T^{-\frac{1}{6} c_2}).\] for some $c > 0$. This is considerably stronger than \eqref{second-psi-dash}, and so this completes the proof of Proposition \ref{main-task-psidash} in the unexceptional case.

\subsection{The exceptional case} We now look at what happens in the exceptional case. Thus $E(n) = 1_{q_1 | n}$, where $q_1 = \cond(\chi_1)$ is the conductor of the exceptional character. 

As in the unexceptional case, our point of departure is \eqref{to-quote-exc}. Now, however, we separate off the term $j = 1$ and use Lemma \ref{basic-beta} (7), obtaining (recalling here that $\rho_1 = 1 - \sigma_1$ is real)

\begin{align*} \sum_{n = 1}^N  \Psi'(n) & \cos(2 \pi \theta n) = 2\beta_0(\lambda_0) \sum_{n =1}^N w(n) (1 - n^{-\sigma_1})\cos(2 \pi \eta n) - \\ & - \sum_{\pm}\sum_{j \geq 2} \beta_j(\pm \lambda_0) \sum_{n = 1}^N w(n) n^{\rho_j - 1}  e(\pm \eta n) - \tilde O(N T^{-\frac{1}{6} c_2}) . \end{align*}
Taking absolute values and using Lemma \ref{basic-beta} (6) gives
\begin{align*} \sum_{n = 1}^N  \Psi'(n) &\cos(2 \pi \theta n)  \geq \beta_0(\lambda_0)\bigg( 2 \sum_{n =1}^N w(n) (1 - n^{-\sigma_1})\cos(2 \pi \eta n) -  \\ & - \sum_{\pm}\sum_{j \geq 2} N^{\frac{1}{8}\sigma_j} \big|\sum_{n = 1}^N w(n) n^{\rho_j - 1}  e(\pm\eta n)\big|\bigg)  - \tilde O(N T^{-\frac{1}{6} c_2}) . \end{align*}

Applying Proposition \ref{prop15.2} (4) and handling the error term exactly as in the unexceptional case, we obtain 
\begin{align}\nonumber
\sum_{n = 1}^N & \Psi'(n) \cos(2 \pi \theta n) \geq \beta_0(\lambda_0) \bigg( 2 \sum_{n =1}^N w(n) (1 - n^{-\sigma_1})\cos(2 \pi \eta n) - \\ & - 2M\big(\sum_{n = 1}^N w(n) \cos (2 \pi \eta n)\big) \sum_{j \geq 1} N^{-\frac{1}{8}\sigma_j} \bigg)- \tilde O(N T^{-\frac{1}{6} c_2}).\label{1529}
\end{align}

Now we additionally apply Proposition \ref{prop15.2} (5), which tells us that  
 \[ \sum_{n = 1}^N n^{-\sigma_1} w(n)  \cos(2 \pi \eta n) \leq N^{-\frac{1}{4}\sigma_1} \sum_{n = 1}^N w(n) \cos(2 \pi \eta n) + O(N^{1 - \frac{1}{12}}).\]
 Substituting into \eqref{1529} and using $|\beta_0(\lambda_0)| \lessapprox 1$ (that is, Lemma \ref{basic-beta} (1)) to control the error term, we obtain
 \begin{equation}\label{9821} \sum_{n = 1}^N  \Psi'(n) \cos(2 \pi \theta n)  \geq  2\tau \beta_0(\lambda_0) \sum_{n =1}^N w(n)  \cos(2 \pi \eta n)  - \tilde O(N T^{-\frac{1}{6} c_2}). \end{equation}
 where
\[ \tau :=  1 - N^{-\frac{1}{4}\sigma_1} - M\sum_{j \geq 2} N^{-\frac{1}{8}\sigma_j} .\]
Now \eqref{exc-n-bd} implies that $\tau \geq 0$; in fact, it implies
\begin{equation}\label{tau-lower} \tau \geq \frac{1}{2}(1 - N^{-\frac{1}{8} \sigma_1}),\end{equation} a fact we shall need shortly.
 We again have \eqref{w-cos-lower} and $|\beta_0(\lambda_0)| \lessapprox 1$, and therefore \eqref{9821} implies that
\[ \sum_{n = 1}^N \Psi'(n) \cos (2 \pi \theta n) \geq - O(N T^{-\frac{1}{6} c_2}).\]
This confirms \eqref{first-psi-dash} in the exceptional case.

Finally, we turn to the proof of \eqref{second-psi-dash} in the exceptional case. Setting $\theta = 0$ in \eqref{9821}, we obtain
\begin{equation}\label{to-use-exc} \sum_{n = 1}^N \Psi'(n) \geq 2 \tau \beta_0(0) \sum_{n = 1}^N w(n) - \tilde O(N T^{-\frac{1}{6} c_2}).
\end{equation}
To estimate this, we use the inequalities $\beta_0(0) \gg q_1^{-1}$ (Lemma \ref{basic-beta} (8)), $\sum_{n = 1}^N w(n) \gg \frac{N}{\log N}$ (Proposition \ref{prop15.2} (2)) and $\tau \geq \frac{1}{2}(1 - N^{-\frac{1}{8}\sigma_1})$ (which is \eqref{tau-lower}).
Substituting into \eqref{to-use-exc} gives 
\begin{equation}\label{paged}\sum_{n = 1}^N \Psi'(n) \gg q_1^{-1} \frac{N}{\log N} \big( 1 - N^{-\frac{1}{8}\sigma_1} ) - \tilde O(N T^{-\frac{1}{6} c_2}).\end{equation} Now we have $1 - N^{-\frac{1}{8} \sigma_1} \geq 1 - e^{-\frac{1}{8} \sigma_1} \gg \sigma_1$, and by Page's theorem (Lemma \ref{page-theorem}) $\sigma_1 \gg q_1^{-1}$. Therefore \eqref{paged} implies that
\[\sum_{n = 1}^N \Psi'(n) \gg q_1^{-2} \frac{N}{\log N}  - \tilde O(N T^{-\frac{1}{6} c_2}).\]

Finally, recall from Proposition \ref{sec4-takeaway} that $q_1 \leq T^{c_3}$. Therefore
\[ \sum_{n = 1}^N \Psi'(n) \gtrapprox N T^{-2 c_3} - \tilde O(N T^{-\frac{1}{6} c_2}) \gtrapprox  N T^{- 2c_3}\] since $c_3 < \frac{1}{12} c_2$.

This concludes the proof of \eqref{second-psi-dash} in the exceptional case, and hence the proof of Proposition \ref{main-task-psidash} itself. \vspace*{8pt}

As remarked at the end of Subsection \ref{sec152}, this completes the proof of all the main results in the paper.

\part{Appendices}

\appendix

\section{Postnikov character formula}\label{appA}

In this section we assemble material that is well-known to experts but for which I could not locate a convenient reference covering all that is needed for this paper in a quotable form. In preparing this we used \cite{gallagher-postnikov} ,  \cite[Section 12.6]{ik} and \cite[Lemma 13]{milicevic}.

Define $L_N \in \Q[x]$ by
\begin{equation}\label{LN-def} L_N(x) := - \sum_{i = 1}^N \frac{(-x)^i}{i} = x - \frac{x^2}{2} + \frac{x^3}{3} + \dots \pm \frac{x^N}{N}.\end{equation}

\begin{lemma}\label{gal-lem}
\[ L_N(x+ y + xy) = L_N(x) + L_N(y) + \sum_{N < i + j \leq 2N} c_{ij} x^i y^j,\] with the $(c_{ij})_{i, j \geq 0}$ being rationals with denominator $\leq N$. 
\end{lemma}
\begin{proof} This is \cite[Lemma 1]{gallagher-postnikov}.\end{proof} 
\emph{Remarks.} For the purposes of this paper, we only need the case $N = 2$, which is an easy direct check, but this sits most naturally within the more general theory. Of course, $L_N$ is a truncated version of the series for $\log(1 + x)$, and then Lemma \ref{gal-lem} becomes a truncated version of $\log \big((1 + x)(1 + y)\big) = \log (1 + x) + \log(1 + y)$.

Denote by $\Z_{(p)} \subseteq \Q$ the ring of rationals whose denominator is coprime to $p$ (which, abstractly, is the localisation of $\Z$ at $p$). Since $v_p(i) \leq i$ for all $i, p$, one sees that 
\begin{equation}\label{lnpx} L_N(px) = \sum_{i=1}^N \frac{(-px)^i}{i} \in \Z_{(p)}[x].\end{equation}
Let $\pi_n : \Z_{(p)} \rightarrow \Z/p^n \Z$ be the natural ring homomorphism. Let $\eta : \Z/p^n \Z \rightarrow \C^{\times}$ be the additive character induced by the character $e(\frac{x}{p^n})$ on $\Z$, and let $\psi_n := \eta \circ \pi_n$, thus $\psi_n$ is an additive character on $\Z_{(p)}$. Note that $\ker \psi_n = \{ x \in \Z_{(p)} : v_p(x) \geq n\}$.

\begin{lemma}
Let $p$ be a prime. Let $m,n$ be positive integers with $m \leq n$, and suppose that $p^m \neq 2$. Let $N$ be any integer such that
\begin{equation}\label{somewhat-annoying} m(N+1) - \lfloor \log_p N\rfloor \geq n.\end{equation}
For each $a \in \Z$, consider the map $\phi_a : 1 + p^m \Z \rightarrow \C^{\times}$ defined by
\begin{equation}\label{phi-psi}  \phi_a(1 + p^m t) := \psi_n(aL_N(p^m t)).\end{equation} Then 
\begin{equation}\label{phi-hom} \phi_a((1 + p^m t)(1 + p^m t')) = \phi_a(1 + p^m t) \phi_a(1 + p^m t')\end{equation}
and
\begin{equation}\label{phi-m-n} \phi_a(1 + p^m t) = \phi_a(1 + p^m t') \quad \mbox{if} \quad t \equiv t' \md{p^{n - m}}.\end{equation}
Finally, the maps $\phi_a$ and $\phi_{a'}$ coincide if and only if $a \equiv a' \md{p^{n - m}}$.
\end{lemma}
\begin{proof}
We start with \eqref{phi-hom}. Substituting in the definitions and using the homomorphism property of $\psi_n$, we must show that 
\[ \psi_n \big( a ( L_N (p^m t + p^m t' + p^{2m} tt') - L_N(p^m t) - L_N(p^m t') ) \big) = 1,\] to which end it is enough (and in fact, if we want a result valid for all $a$, necessary and sufficient) to show that 
\[ v_p (L_N (p^m t + p^m t' + p^{2m} tt') - L_N(p^m t) - L_N(p^m t')) \geq n.\]
By Lemma \ref{gal-lem} with $x = p^m t$, $y = p^m t'$, it suffices to show that $v_p(c_{ij} (p^m t)^i (p^m t')^j) \geq n$ for every pair $i, j$ featuring in that lemma. Note that $i + j \geq N + 1$ for all these pairs, and also that $v_p(c_{ij}) \geq - \lfloor \log_p N\rfloor$ since $c_{ij}$ has denominator at most $N$. The required statement therefore follows from the assumption \eqref{somewhat-annoying}.

We turn now to \eqref{phi-m-n}. Here, it is enough to show that 
\[ v_p\big(L_N(p^m t) - L_N(p^m t')\big) \geq n\] when $t \equiv t' \md{p^{n - m}}$.
Recalling the definition \eqref{LN-def} of $L_N$, it is enough to proceed termwise and show that 
\begin{equation}\label{enough-minute} v_p \bigg( \frac{(p^m t)^i}{i} - \frac{(p^m t')^i}{i} \bigg) \geq n\end{equation} for $i = 1,2,\dots$.
By factoring $t^i - (t')^i = (t - t') (t^{i - 1} + \dots)$ we have $v_p(t^i - (t')^i) \geq v_p(t - t') \geq n - m$, and so to establish \eqref{enough-minute} it is sufficient to show $im + n - m - v_p(i) \geq n$, that is to say $im \geq m + v_p(i)$. This is a true statement: it follows from $v_p(i) \leq i-1$ and $(i - 1)(m-1) \geq 0$. This completes the proof of \eqref{phi-m-n}.

Finally, we turn to the last statement. Suppose that $\phi_a$ and $\phi_{a'}$ are the same; then $\phi_{a - a'}$ is trivial. In particular, $\phi_{a - a'} (1 + p^m) = 1$, which means that $v_p((a - a') L_N(p^m) ) \geq n$. It is therefore enough to show that $v_p(L_N(p^m)) = m$. From the definition \eqref{LN-def} we have $L_N(p^m) = p^m - \sum_{i = 2}^N(-1)^i \frac{p^{im}}{i}$. Since the $p$-adic valuation of the first term is $m$, it is enough to show that all subsequent terms have strictly greater valuation, that is to say 
\begin{equation}\label{to-ver} im - v_p(i) \geq m + 1 \end{equation} for all $i \geq 2$. If $p$ is odd, we have $v_p(i) \leq i - 2$ and the inequality \eqref{to-ver} follows from $(i - 1)(m - 1) \geq 0$, which is of course true for all $i \geq 2$ and $m \geq 1$. If $p = 2$, the same applies unless $i = 2$. In that case, it may be checked directly that \eqref{to-ver} holds unless $m = 1$. 
\end{proof}

\begin{corollary}\label{cora3}
Let $p$ be a prime. Let $m,n$ be positive integers with $m \leq n$, and suppose that $p^m \neq 2$. Let $N$ be any integer such that \eqref{somewhat-annoying} is satisfied. Let $\chi$ be a Dirichlet character $\mdlem{p^n}$. Then there is some $a$ such that $\chi = \phi_a$ on $1 + p^m \Z$. 
\end{corollary}
\begin{proof}
$\chi$ restricts to give a character on $\Gamma_{m,n}(p) := (1 + p^m \Z )/p^n \Z$, which is a group of size $p^{n - m}$. By \eqref{phi-hom}, \eqref{phi-m-n}, every $\phi_a$ also restricts to such a character. Since the number of distinct $\phi_a$ is $p^{n - m}$, every character on $\Gamma_{m,n}(p)$ arises in this way. The result follows.
\end{proof}

Now consider \eqref{lnpx} again. Since $v_p(i) \leq i -1$ for all primes $p$ and all $i \geq 1$, we have $L_N(px) = \sum_{i = 1}^N a_i x^i$ with $v_p(a_i) \geq 1$, and so $L_N(p^m x) = \sum_{i = 1}^N a'_i x^i$ with $v_p(a'_i) \geq m$. From the definition \eqref{phi-psi} and the fact that $\psi_n$ is an additive character, we see that $\phi_a(1 + p^m t) =  \psi_n (\sum_{i = 1}^N a a'_i t^i)$. Let $b_i$, $i = 1,\dots, N$, be integers such that $\pi_n(aa'_i) = \pi_n(b_i p^m)$. Then we see that $\phi_a(1 + p^m t) = e(\frac{f_a(t)}{p^{n - m}})$, where $f_a(t) := \sum_{i = 1}^n b_i t^i$. 

Combining these observations with Corollary \ref{cora3}, which is the result we call the \emph{Postnikov character formula} in this paper.

\begin{proposition}\label{post-formula}
Let $p$ be a prime. Let $m, n$ be positive integers with $m \leq n$, and suppose that $p^m \neq 2$. Let $\chi$ be a Dirichlet character with conductor $p^n$. Then there is a polynomial $f \in \Z[x]$ such that $\chi(1 + p^m x) = e(\frac{f(x)}{p^{n - m}})$, and such that $f$ has no constant term and has degree at most $N$, where $N$ is the smallest positive integer satisfying $m (N+1) - \lfloor \log_p N \rfloor \geq n$.
\end{proposition}

Finally, we remark on the particular regime of interest in the main body of the paper, specifically in Section \ref{sec11}, which is where $m > \frac{1}{3}n$. Thus $3m \geq n + 1$, and we see that $N$ can be taken to be $2$. We must also ensure that $p^m \neq 2$, which will always be the case unless $\cond(\chi) = p^n \in \{ 2,4\}$.

\section{Various estimates for arithmetic functions}\label{appB}

In this appendix we collect various estimates, mainly for sums involving divisor functions, which are of fairly standard type. Throughout, $B \geq 0$ is some real parameter.

\begin{lemma}\label{a6-standard}
We have 
\[ \sum_{n \leq X} \tau(n)^B  \ll_B X (\log X)^{O_B(1)}.\]\end{lemma}
\begin{proof} This is standard; see for instance \cite[equation (1.80)]{ik}.\end{proof}

\begin{lemma}\label{to-rankin}
Let $h \in \N$. We have
\[ \sum_{\substack{n \leq X \\ \mu^2(n) = 1}} \frac{\tau(n)^B (n,h)}{n} \ll_B \tau(h)^{B+1} \log ^{2^B} X\]
\end{lemma}
\begin{proof}
We use ``Rankin's trick'', noting that
\begin{align*}  \sum_{\substack{n \leq X \\ \mu^2(n) = 1}} \frac{\tau(n)^B (n,h)}{n} & \ll \sum_n \mu^2(n)\tau(n)^B (n,h) n^{-1 - 1/\log X} \\ & = \prod_{p | h} \big(1 +2^B p^{-1/\log X}\big) \prod_{p \nmid h} \big(1 + 2^B p^{-1 - 1/\log X}\big) \\ & \ll \tau(h)^{B+1} \prod_p \big(1 + 2^B p^{-1 - 1/\log X}\big).  \end{align*}
Here, for the first term we used the crude inequality $1 + 2^B p^{-1/\log X} < 2^{B+1}$. By Taylor expansion we have, for some $C = C(B)$, $(1 + 2^B x) (1 - x)^{2^B} \leq 1 + Cx^2$ for all $x \leq \frac{1}{2}$. Applying this to the product over primes above, and noting that 
\[ \prod_{p} (1 - p^{-1 -1/\log X})^{-1}  = \zeta(1+ \frac{1}{\log X}) = (1 + o(1)) \log X\] and $\prod_p (1 + Cp^{-2}) = O_B(1)$, the result follows.
\end{proof}

\begin{lemma}\label{to-rankin-2}
Let $h \in \N$. We have
\[ \sum_{\substack{n \geq X \\ \mu^2(n) = 1}} \frac{(n,h) \tau(n)^B \log^B n}{n^2} \ll_{\eps, B}  h^{\eps} X^{\eps - 1}.\]
\end{lemma}
\begin{proof}
By Lemma \ref{to-rankin} we have
\[ \sum_{\substack{X \leq n < 2X \\ \mu^2(n) = 1}} \frac{(n,h) \tau(n)^B \log^B n}{n^2} \ll_{B,\eps} \tau(h)^{B+1} X^{\eps - 1}.\] Summing dyadically over $X$ and applying the divisor bound to $\tau(h)$ gives the result.
\end{proof}

\begin{lemma}\label{power-res-div}
Write $\omega_R(n)$ for the number of distinct primes $p \leq R$ for which $p | n$. Then 
\[ X^{-1}\sum_{n \leq X} 4^{\omega_R(n)} \ll \log^{3} R.\]
\end{lemma}
\begin{proof}
We have
\[ 2^{\omega_R(n)} = \sum_{\substack{d | n \\ d | P_R}} 1,\] where $P_R$ is the product of the primes less than or equal to $R$.  Squaring, summing over $X$ and interchanging the order of summation gives
\[ \sum_{n \leq X} 4^{\omega_R(n)} \leq \sum_{n \leq X} \sum_{\substack{d, e | n \\ d, e | P_R}} 1 = \sum_{d, e | P_R} \sum_{\substack{n \leq X \\ d, e | n}} 1 \leq X \sum_{d, e | P_R} \frac{1}{[d,e]} \leq X\!\!\!\sum_{d', e', u | P_R} \frac{1}{d' e' u},\] where in the last step we wrote $d = ud'$, $e = ue'$ with $u := (d,e)$.
The expression on the right is 
\[ X \big(\sum_{d | P_R} \frac{1}{d}\big)^3 = X \prod_{p \leq R} \big(1 + \frac{1}{p}\big)^3 \ll X \log^3 R.\]
This completes the proof.
\end{proof}

\begin{lemma}\label{enhanced-mobius}
We have
\[ 1_{(m,n) = \Delta} = \sum_{d : \Delta | d | n} \mu\big(\frac{d}{\Delta}\big) 1_{m \equiv 0 \mdsublem{d}}.\]
\end{lemma}
\begin{proof}
Suppose first that the sum on the right is empty, that is to say there is no $d$ with $\Delta \mid d \mid n$ and $d \mid m$. Then we cannot have $(m,n) = \Delta$, since in that case $\Delta$ itself is an example of such a $d$. This means that both the LHS and the RHS are zero.

Now suppose that the sum on the right is non-empty, that is to say there is some $d$ with $\Delta \mid d \mid n$ and $d \mid m$. In particular, $\Delta \mid (m,n)$. Write $m = \Delta m'$, $n = \Delta n'$, and substitute $d = \Delta d'$ in the sum. Then the RHS is
\[ \sum_{d' | n'} \mu(d') 1_{d' | m'} = \sum_{d' | (m', n')} \mu(d') = 1_{(m', n') = 1} = 1_{(m,n) = \Delta},\] by the usual version of the M\"obius inversion formula. 
\end{proof}

\section{Dirichlet characters}\label{char-app}

We need the following result about decomposing Dirichlet characters as products of local characters, which is well-known but not always covered in analytic number theory texts. 

\begin{lemma}
Suppose that $\chi$ is a Dirichlet character to modulus $q = p_1^{e_1} \cdots p_m^{e_m}$.  Then there are unique Dirichlet characters $\chi_{p_i}$ to moduli $p_i^{e_i}$ such that $\chi = \prod_{i = 1}^m \chi_{p_i}$. Moreover $\chi$ is primitive if and only if all of the $\chi_{p_i}$ are, that is to say if and only if none of the $\chi_{p_i}$ is the principal character $1_{(n, p_i) = 1}$.
\end{lemma}
\begin{proof}
Consider the isomorphism $\pi : (\Z/q\Z)^{\times} \rightarrow \prod_{i = 1}^m (\Z/p_i^{e_i} \Z)^{\times}$ given by the Chinese remainder theorem. This induces a natural homomorphism $\pi_* : \prod_{i = 1}^m \widehat{(\Z/p_i^{e_i} \Z)^{\times} } \rightarrow \widehat{(\Z/q\Z)^{\times}}$ defined by 
\[ ( \pi_*(\chi_1,\dots, \chi_m))(x) := \chi_1((\pi(x))_1) \cdots \chi_m((\pi(x)_m),\] where $(\pi(x))_i$ denotes the $i$th coordinate of $\pi(x)$ or, more colloquially, $(\pi(x))_i = x \md{p_i^{e_i}}$. It is easy to see that $\pi_*$ is injective. Therefore, since the cardinality of the domain and range of $\pi_*$ are the same, $\pi_*$ is an isomorphism. Lifting characters on $(\Z/q\Z)^{\times}$ and on the $(\Z/p_i^{e_i}\Z)^{\times}$ to Dirichlet characters on $\Z$, the first statement of the lemma follows. 

Turning to the statement about primitivity, denote by $\widehat{(\Z/r\Z)^{\times}_{\prim}}$ the set of primitive characters mod $r$. Now observe that if for some $j$ we have $\chi_j \notin \widehat{(\Z/p_j^{e_j}\Z)^{\times}_{\prim}}$ (i.e. it is principal) then $\pi_*(\chi_1,\dots, \chi_m) \notin \widehat{(\Z/q\Z)^{\times}_{\prim}}$, since the latter is induced from a character modulo $\prod_{i \neq j} p_i^{e_i}$. It follows that $\pi_* \big( \prod_{i = 1}^m  \widehat{(\Z/p_j^{e_j}\Z)^{\times}_{\prim}} \big) \supseteq  \widehat{(\Z/q\Z)^{\times}_{\prim}}$.
To show that the containment is in fact an equality, it is enough to show that $\phi^*(r) := |\widehat{(\Z/r\Z)^{\times}_{\prim}}|$ is a multiplicative function of $r$. This follows from the fact that $\phi = 1 \star \phi^*$ and hence $\phi^* = \mu \star \phi$, as explained just after \cite[equation (3.7)]{ik} (where a formula for $\phi^*(r)$ is given).
\end{proof}

\section{An exponential sum estimate}

The following lemma is of a well-known type and (apart from constants) is special case of a classical and much more general result of Hua. For references and a detailed bibliography of results of this type, see \cite{cochrane-zheng}.

\begin{lemma}\label{quadratic-bd}
Suppose that $p$ is a prime and that $\min_{i = 1, 2} v_p(a_i) = n - r$, $r \geq 1$. Then if $p$ is odd we have 
\[\bigg| \sum_{x \in \Z/p^n \Z} e\big(\frac{a_1 x + a_2 x^2}{p^n}\big)\bigg| \leq p^{n - r/2}.\]
A similar bound holds when $p = 2$, but with an extra factor of $2^{1/2}$.
\end{lemma}
\begin{proof}
Suppose first that $p$ is odd. The claim is trivial for $r = 0$, so suppose $r \geq 1$.  Write $\Sigma$ for the sum. Then $|\Sigma|^2 = \sum_{x,y \in \Z/p^n \Z} e(\frac{1}{p^n} (a_1 (x - y) + a_2 (x^2 - y^2)))$. Making the substitution $h := x - y$, $k := x + y$ gives that this is $\sum_{h,k \in \Z/p^n \Z} e(\frac{1}{p^r} h (a'_1 + k a'_2))$, where we have written $a_1 = p^{n - r} a'_1$, $a_2 = p^{n - r}a'_2$ with at least one of $a'_1, a'_2$ coprime to $p$. Performing the sum over $h$, we see that this is $p^{2n - r} \# \{ k \in \Z/p^r \Z : a'_1 + k a'_2 \equiv 0 \md{p^r}\}$. If $(a'_2, p) = 1$ then there is a unique value of $k$ here, so $|\Sigma|^2 = p^{2n - r}$, which concludes the proof in this case. If $p | a'_2$ then $(a'_1, p) = 1$, and there are no values of $k$, so $\Sigma = 0$.  This concludes the proof in the case $p$ odd.

Now suppose $p = 2$. The claim is trivial when $r = 0$ or $1$, so suppose $r \geq 2$. The preceding argument needs a small modification, because on making the substitution $h$ and $k$ do not range over all of $(\Z/2^n \Z)^2$ but rather over the set of pairs with $h + k \equiv 0 \md{2}$ (twice each). That is, 
\[ |\Sigma|^2 = 2 \sum_{h,k \in \Z/2^n \Z} 1_{h + k \equiv 0 \mdsub{2}} e\big(\frac{1}{2^r} h(a'_1 + ka'_2)\big),\] with the notation the same as before. Suppose first that $r \neq 1$. Expanding $2 \cdot 1_{m \equiv 0 \mdsub{2}}$ as $1 + e(\frac{m}{2})$, we obtain
\[ |\Sigma|^2 = \sum_{h,k \in \Z/2^n \Z}  e\big(\frac{1}{2^r} h(a'_1 + ka'_2)\big) + \sum_{h,k \in \Z/2^n \Z}  e\big(\frac{1}{2^r} h(a'_1 + ka'_2 + 2^{r-1}) + \frac{1}{2}k\big).\]
The first term is $\leq 2^{2n - r}$, exactly as before.  For the second term we again perform the sum over $h$, obtaining that this is 
\[ 2^{2n - r} \sum_{k \in \Z/2^r \Z} 1_{a'_1 + ka'_2 \equiv 2^{r-1} \mdsub{2^r}} e\big(\frac{k}{2}\big).\]
If $a'_2$ is odd then there is a unique value of $k$ here, giving a contribution of $2^{2n - r}$. If $a'_2$ is even then $a'_1$ is odd, and there are no values of $k$. 

Putting all this together gives an upper bound $|\Sigma|^2 \leq 2^{2n - r + 1}$, which is what we wanted to prove. \end{proof}

\section{An explicit estimate for the $\Gamma$-function}

In this section we record an explicit estimate for the size of the $\Gamma$-function in the strip $\frac{1}{4} \leq \Re z \leq \frac{1}{2}$. An estimate such as this is necessary if one wishes to compute an explicit exponent in our main theorem.

\begin{lemma}\label{explicit-gam}
Suppose that $\frac{1}{4} \leq x \leq \frac{1}{2}$. Then $|\Gamma(x + iy)| \leq 7 e^{-\pi |y|/2}$.
\end{lemma}
\begin{proof}
We may assume $y > 0$; the case $y < 0$ follows from this by conjugation. We use the explicit version of Stirling's formula from \cite[p.294]{special-funct}, which is valid for $|\arg z| \leq \frac{\pi}{2}$:
\begin{equation}\label{explicit-stirling} \log \Gamma(z) = \big(z - \frac{1}{2}\big) \log z - z + \frac{1}{2} \log 2\pi + \frac{\theta(z)}{6|z|},\end{equation}
where $|\theta(z)| \leq 1$. Substituting $z = x + iy$, we have
\begin{align}\nonumber
\log | \Gamma(x + iy)| & = \Re \log \Gamma(x + iy) \\ \nonumber & = \big(x - \frac{1}{2}\big) \log\sqrt{x^2 + y^2} - y \arg (x + iy)  \\ & \nonumber \qquad \qquad - x + \frac{1}{2} \log 2\pi + \Re \frac{\theta(z)}{6|z|} \\ & \leq \frac{1}{4}\log 4 - y \arg (x + iy) - x + \frac{1}{2} \log 2\pi + \frac{2}{3}.\label{to-use-gam}
\end{align}
Here, we used the fact that $|z| \geq x \geq \frac{1}{4}$ to bound the last term, and also the inequality $(x - \frac{1}{2}) \log\sqrt{x^2 + y^2} \leq \frac{1}{4} \log 4$, which follows from the fact that $\log \sqrt{x^2 + y^2} \geq -\log 4$ for $x \geq \frac{1}{4}$.

Now for $t \geq 0$ we have the inequality
\begin{equation}\label{inverse-tan} \tan^{-1} t \geq \frac{\pi}{2} - \frac{1}{t}.\end{equation} To prove this, observe that $f(t) := \tan^{-1} t - \frac{\pi}{2} + \frac{1}{t}$ tends to zero as $t \rightarrow \infty$, but $f'(t) = \frac{1}{1 + t^2} - \frac{1}{t^2} < 0$ for all $t$. 

The inequality \eqref{inverse-tan} implies that 
\[ \arg(x + iy) = \tan^{-1} \big(\frac{y}{x}\big) \geq \frac{\pi}{2} - \frac{x}{y}\] for $x, y \geq 0$. Substituting into \eqref{to-use-gam} gives
\[ \log |\Gamma(x + iy)| \leq \frac{1}{4} \log 4 -\frac{\pi}{2} y  + \frac{1}{2} \log 2\pi + \frac{2}{3} < -\frac{\pi}{2} y + \log 7,\] from which the lemma follows.
\end{proof}

\end{document}